\documentclass[11pt]{article}
\usepackage{amssymb,amsfonts}%amsLatex
\usepackage{amsmath,amsthm,amsxtra}

\usepackage{graphicx}
\usepackage{amscd}
\usepackage{ulem}
\usepackage{makeidx}
\usepackage{fancybox}

\usepackage[linktocpage=true, colorlinks=true, linkcolor=blue, citecolor=red, urlcolor=green]{hyperref}

\newcommand{\ccc}{{\mathbf C}}

\newcommand{\nnn}{{\mathbf N}}

\newcommand{\zzz}{{\mathbf Z}}
\renewcommand{\ggg}{{\frak{g}}}
\newcommand{\hhh}{{\frak{h}}}

\newtheorem{thm}{Theorem}[section]
\newtheorem{prop}{Proposition}[section]
\newtheorem{lemma}{Lemma}[section]
\newtheorem{cor}{Corollary}[section]

\newtheorem{ex}{Example}[section]

\newtheorem{note}{Note}[section]

\numberwithin{equation}{section}

\begin{document}

\title{On the characters of a certain series of \\
N=4 superconformal modules}

\author{\footnote{12-4 Karato-Rokkoudai, Kita-ku, Kobe 651-1334, 
Japan, \qquad
wakimoto.minoru.314@m.kyushu-u.ac.jp, \hspace{5mm}
wakimoto@r6.dion.ne.jp 
}{ Minoru Wakimoto}}

\date{\empty}

\maketitle

\begin{center}
Abstract
\end{center}

In this paper we study the N=4 superconformal modules obtained from 
the quantum Hamiltonian reduction of principal admissible representations of
the affine Lie superalgebra $\widehat{A}(1,1)$, and show that there exists
a series of N=4 superconformal modules whose characters are modular
functions and written explicitly by the Mumford's theta functions.

\tableofcontents

\section{Introduction}
\label{sec:introduction}

For N=4 superconformal modules, properties of modified characters
are studied in \cite{KW2017b}. In the current paper we consider the 
\lq \lq honest" characters of N=4 superconformal modules, 
namely characters before modification.
The method in this paper is very simple as follows.

\begin{itemize}
\item As we see in \cite{KW2017b}, the formula for the characters 
of N=4 modules contains the differential of mock theta functions 
$\Phi^{[m,s]}$.

\item These differential disappear if we consider suitable sum 
of two irreducible N=4 modules.

\item There exist cases in which the one of two irreducible 
components vanishes. 
Then, in such cases, the character of the other irreducible component 
which survives is written only by $\Phi^{[m,s]}$'s without 
their differential.

\item Furthermore, in the case $m=1$, the function $\Phi^{[1,s]}$ 
$(s \in \zzz)$ can be written explicitly by the Mumford's 
theta functions $\vartheta_{ab}(\tau,z)$.
\end{itemize}

The $\lambda$-brackets of the generating fields of the N=4 
superconformal algebra obtained from the quantum Hamiltonian 
reduction of the affine Lie superalgebra $\widehat{A}(1,1)=
(\widehat{sl(2|2)/\ccc I})$ are obtained in \cite{KW2004}.

An irreducible highest weight N=4 superconformal module $(\pi, V)$ 
is determined by 3 parameters $(c_V, h_V, s_V)$, where $c_V$ 
is the central charge of the Virasoro field $L$ and $h_V$ 
(resp. $s_V$) is the eigenvalue of $L_0$ (resp. $J_0$) on the 
highest weight vector $v_0$ in $V$.
For $M \in \nnn$ we put 
%(label=n4:eqn:2022-1203a) %%
\begin{equation}
\begin{array}{lclcc}
I^{[M]} &:=& \big\{j \, \in \, \frac12 \zzz_{\rm odd} 
& ; &
-\frac{M-1}{2} \, \leq \, j \, \leq \, \frac{M}{2}\big\} 
\\[3mm]
I^{[M], R} &:=& \big\{j \, \in \, \zzz 
& ; & 
- \frac{M-1}{2} \, \leq \, j \, \leq \, \frac{M}{2}
\big\} 
\end{array}
\label{n4:eqn:2022-1203a}
\end{equation}
Our main result in this paper is the following: 

\begin{thm} \,\ 
\label{n4:thm:2022-1203a}
\begin{enumerate}
\item[{\rm 1)}] Let $M \in \nnn$ and $j \in I^{[M]}$, and 
$V^{[M,j]}$ be the N=4 module such that 
{\allowdisplaybreaks
\begin{eqnarray}
c_{V^{[M,j]}}&=& \frac{6(1-M)}{M} \,\ (=:c^{[M]})
\nonumber
\\[0mm]
h_{V^{[M,j]}} &=& \frac{j^2}{M}+\frac{1}{4M}-\frac12 \,\ (=: h^{[M,j]}), 
\quad 
s_{V^{[M,j]}} = \frac{2j}{M}-1 \,\ (=:s^{[M,j]})
\label{n4:eqn:2023-106a}
\end{eqnarray}}
Then the character ${\rm ch}^{(+)}$ and super-character ${\rm ch}^{(-)}$ 
of $V^{[M,j]}$ are given by the following
formulas:
{\allowdisplaybreaks 
\begin{eqnarray*}
& & \hspace{-7mm}
{\rm ch}^{(+)}_{V^{[M,j]}}(\tau,z)
\,\ = \,\ 
- \, {\rm sgn}(j) \, q^{\frac{1}{M}j^2} e^{\frac{4\pi i}{M}jz}
\\[2mm]
& &
\times \,\ \frac{
\vartheta_{00}(M\tau, \, z+j\tau)
\vartheta_{01}(M\tau, \, z+j\tau)
\vartheta_{11}(M\tau, \, z+j\tau)
}{\vartheta_{10}(M\tau, \, z+j\tau)} \cdot
\frac{\vartheta_{00}(\tau,z)}{
\vartheta_{01}(\tau,z)\vartheta_{10}(\tau,z)\vartheta_{11}(\tau,z)}
% ch^{(-)}
\\[3mm]
& & \hspace{-7mm}
{\rm ch}^{(-)}_{V^{[M,j]}}(\tau,z)
\,\ = \,\ {\rm sgn}(j) \, 
q^{\frac{1}{M}j^2} e^{\frac{4\pi i}{M}jz}
\\[2mm]
& &
\times \,\ \dfrac{
\vartheta_{00}(M\tau, \, z+j\tau)
\vartheta_{01}(M\tau, \, z+j\tau)
\vartheta_{10}(M\tau, \, z+j\tau)
}{\vartheta_{11}(M\tau, \, z+j\tau)} \cdot
\frac{\vartheta_{01}(\tau,z)}{
\vartheta_{00}(\tau,z)\vartheta_{10}(\tau,z)\vartheta_{11}(\tau,z)}
\end{eqnarray*}}
% Ramond
\item[{\rm 2)}] Let $M \in \nnn$ and $j \in I^{[M], R}$, and 
$V^{[M,j]R}$ be the Ramond twisted N=4 module such that 
{\allowdisplaybreaks
\begin{eqnarray}
c_{V^{[M,j]R}} &=& \frac{6(1-M)}{M} \,\ (=: c^{[M]R})
\nonumber
\\[0mm] 
h_{V^{[M,j]R}} &=& \frac{j^2}{M}+\frac{1}{4M}-\frac14 \,\ 
(=:h^{[M,j]R}), \quad 
s_{V^{[M,j]R}} = \frac{2j}{M} \,\ (=: s^{[M,j]R})
\label{n4:eqn:2023-106b}
\end{eqnarray}}
Then the character ${\rm ch}^{(+)}$ and super-character ${\rm ch}^{(-)}$ 
of $V^{[M,j]R}$ are given by the following formulas:
{\allowdisplaybreaks 
\begin{eqnarray*}
& & \hspace{-7mm}
{\rm ch}^{(+)}_{V^{[M,j]R}}(\tau,z) \,\ = \,\ 
- \, {\rm sgn}(j) \, q^{\frac{1}{M}j^2} e^{\frac{4\pi i}{M}jz}
\\[2mm]
& &
\times \,\ \frac{
\vartheta_{00}(M\tau, \, z+j\tau)
\vartheta_{01}(M\tau, \, z+j\tau)
\vartheta_{11}(M\tau, \, z+j\tau)
}{\vartheta_{10}(M\tau, \, z+j\tau)}
\cdot \frac{\vartheta_{10}(\tau,z)}{
\vartheta_{00}(\tau,z)\vartheta_{01}(\tau,z)\vartheta_{11}(\tau,z)}
\\[3mm]
& & \hspace{-7mm}
{\rm ch}^{(-)}_{V^{[M,j]R}}(\tau,z) \,\ = \,\ 
- \, {\rm sgn}(j) \, 
q^{\frac{1}{M}j^2} e^{\frac{4\pi i}{M}jz}
\\[2mm]
& &
\times \,\ \frac{
\vartheta_{00}(M\tau, \, z+j\tau)
\vartheta_{01}(M\tau, \, z+j\tau)
\vartheta_{10}(M\tau, \, z+j\tau)
}{\vartheta_{11}(M\tau, \, z+j\tau)}
\cdot \frac{\vartheta_{11}(\tau,z)}{
\vartheta_{00}(\tau,z)\vartheta_{01}(\tau,z)\vartheta_{10}(\tau,z)}
\end{eqnarray*}}
\end{enumerate}
where \qquad $
{\rm sgn}(j) \,\ := \,\ \left\{
\begin{array}{rcl}
1 & & {\rm if} \,\ j \, > \, 0 \\[1mm]
-1 & & {\rm if} \,\ j \, \leq \, 0
\end{array}\right. $
\end{thm}

%\medskip

This paper is organized as follows.
In section \ref{sec:preliminaries}, we recall mock theta functions
$\Psi^{[M,m,s;\varepsilon]}_{j,k;\varepsilon'}$
and their properties from \cite{KW2017b}.
In sections \ref{sec:A(11):integrable} and \ref{sec:A(11):admissible},
we compute the characters of integrable and principal 
admissible $\widehat{A}(1,1)$-modules and, 
in section \ref{sec:quantum:character}, deduce the formulas 
for the characters of N=4 modules obtained from the quantum 
Hamiltonian reduction of $\widehat{A}(1,1)$-modules.

In section \ref{sec:h-lambda}, we deduce formulas for 
$h_{\lambda}$, $s_{\lambda}$, $h^{\rm tw}_{\lambda}$ 
and $s^{\rm tw}_{\lambda}$, which are important quantities 
characterizing non-twisted and twisted N=4 modules. 

In section \ref{sec:non-irred}, we show that 
the character of suitable sum of two irreducible N=4 modules 
can be written by  mock theta functions $\Phi^{[m,s]}$ without 
their differentials.
In section \ref{sec:vanishing} we study conditions for vanishing 
of quantum Hamiltonian reduction which, together with the results 
in section \ref{sec:non-irred}, lead us to section \ref{sec:nice}. \,\
In section \ref{sec:nice}, we complete the proof of 
Theorem \ref{n4:thm:2022-1203a}.

In section \ref{sec:examples}, we consider the cases $M=1$ and $M=2$.
Since the case $M=1$ is the trivial N=4 representation,
the case $M=2$ gives the simplest non-trivial N=4 superconformal 
modules. Applying Theorem \ref{n4:thm:2022-1203a} to the case $M=2$,
we obtain the characters of N=4 modules with central charge $=-3$.
Finally in section \ref{sec:SL(2Z)-invariance}, we show that 
the non-twisted and twisted (super)characters studied in 
section \ref{sec:non-irred}  span $SL_2(\mathbf{Z})$-invariant spaces
in the case $m=1$.

\medskip

In this paper, we follow notations and definitions from \cite{KRW}, 
\cite{W2022a}, \cite{W2022b}, \cite{W2022d} and \cite{W2022e}.

\section{Preliminaries}
\label{sec:preliminaries}
%(line=303)  %%

%\medskip %%

Using the mock theta function 
$\Phi^{[m,s]}(\tau, z_1,z_2,t)$ and its Zwegers' modification
$\widetilde{\Phi}^{[m,s]}(\tau, z_1,z_2,t)$
defined in \cite{W2022a}, we define the functions 
$\Psi^{[M,m;s; \varepsilon]}_{j,k;\varepsilon'}(\tau, z_1,z_2,t)$
and 
$\widetilde{\Psi}^{[M,m;s; \varepsilon]}_{j,k;\varepsilon'}(\tau, z_1,z_2,t)$
by the following formulas:
%(label=eqn:2022-1017a2) %%
%(label=eqn:2022-1017a3) %%
%\vspace{-4mm} %%
\begin{subequations}
{\allowdisplaybreaks
\begin{eqnarray}
& &\hspace{-15mm}
\Psi^{[M,m;s; \varepsilon]}_{j,k;\varepsilon'}(\tau, z_1,z_2,t)
\nonumber
\\[3mm]
& & \hspace{-10mm}
:= \,\ 
q^{\frac{m}{M}jk} \, 
e^{\frac{2\pi im}{M}(kz_1+jz_2)} \, 
\Phi^{[m;s]}\Big(
M\tau, \, z_1+j\tau+\varepsilon, \, z_2+k\tau-\varepsilon, \, \frac{t}{M}
\Big)
\label{eqn:2022-1017a2}
\\[3mm]
& & \hspace{-15mm}
\widetilde{\Psi}^{[M,m;s; \varepsilon]}_{j,k;\varepsilon'}(\tau, z_1,z_2,t)
\nonumber
\\[3mm]
& & \hspace{-10mm}
:= \,\ 
q^{\frac{m}{M}jk} \, 
e^{\frac{2\pi im}{M}(kz_1+jz_2)} \, 
\widetilde{\Phi}^{[m;s]}\Big(
M\tau, \, z_1+j\tau+\varepsilon, \, z_2+k\tau-\varepsilon, \, \frac{t}{M}
\Big)
\label{eqn:2022-1017a3}
\end{eqnarray}}
\end{subequations}
where $m \in \frac12 \nnn$ and $M$ is a positive odd integer coprime to $2m$,
or $m=1$ and $M \in \nnn$, and $s \in \frac12 \zzz$, and 
$\varepsilon, \, \varepsilon' \in \{0, \frac12\}$ and $j,k \in \varepsilon'+\zzz$.
Since 
%(label=n4:eqn:2022-1210a) %%
\begin{equation}
\widetilde{\Phi}^{[1,s]}(\tau,z_1,z_2,t) \, = \, 
\Phi^{[1,s]}(\tau,z_1,z_2,t) \, = \, 
-i \, e^{-2\pi it} \, 
\frac{\eta(\tau)^3 \, \vartheta_{11}(\tau, z_1+z_2)}{
\vartheta_{11}(\tau, z_1) \, \vartheta_{11}(\tau, z_2)}
\quad (s \in \zzz)
\label{n4:eqn:2022-1210a}
\end{equation}
by Lemma 2.7 in \cite{W2022a}, one has
%(label=n4:eqn:2022-1208a) %%
{\allowdisplaybreaks
\begin{eqnarray}
& & \hspace{-10mm}
\widetilde{\Psi}^{[M,1;s; \varepsilon]}_{j,k;\varepsilon'}(\tau, z_1, z_2, t) 
\,\ = \,\ 
\Psi^{[M,1;s; \varepsilon]}_{j,k;\varepsilon'}(\tau, z_1, z_2, t)
\nonumber
\\[1.5mm]
&=&
- \, i \, e^{-\frac{2\pi it}{M}} \, 
q^{\frac{jk}{M}} \, 
e^{\frac{2\pi i}{M}(kz_1+jz_2)} \,\ 
\frac{\eta(M\tau)^3 \, 
\vartheta_{11}\big(M\tau, \, z_1+z_2+(j+k)\tau\big)}{
\vartheta_{11}\big(M\tau, \, z_1+j\tau+\varepsilon\big) \, 
\vartheta_{11}\big(M\tau, \, z_2+k\tau-\varepsilon\big)}
\label{n4:eqn:2022-1208a}
\end{eqnarray}}
for $s \in \zzz$.  
Then by this equation \eqref{n4:eqn:2022-1208a}, the formulas for 
$\widetilde{\Psi}^{[M,m;s; \varepsilon]}_{j,k;\varepsilon'}
(\tau, z_1, z_2, t)$ proved in \cite{KW2017b} hold for 
$\Psi^{[M,1;s; \varepsilon]}_{j,k;\varepsilon'}(\tau, z_1, z_2, t)$
in the case $s \in \zzz$. 
Then by (1.17) and (1.18) in \cite{KW2017b}, we have the following:

\medskip %%

%(line=401) %%
%(label=n4:lemma:2022-1207a) %%
\begin{lemma} \,\ 
\label{n4:lemma:2022-1207a}
For $M \in \nnn$ and $\varepsilon, \, \varepsilon' \in \{0, \frac12\}$, 
the following formulas hold:
\begin{enumerate}
\item[{\rm 1)}] \,\ $\Psi^{[M,1;0; \varepsilon]}_{j,k;\varepsilon'}
\Big(-\dfrac{1}{\tau}, \dfrac{z_1}{\tau}, \dfrac{z_2}{\tau}, t\Big)
\, = \, 
\dfrac{\tau}{M} \, e^{\frac{2\pi i}{M\tau}z_1z_2}
\hspace{-7mm}
\sum\limits_{\substack{\\[0.5mm] (a,b) \, \in \, 
(\varepsilon+\zzz/M\zzz)^2}} \hspace{-7mm}
e^{-\frac{2\pi i}{M}(ak+bj)} \, 
\Psi^{[M,1;0; \varepsilon']}_{a,b;\varepsilon}(\tau, z_1, z_2,t)$
\item[{\rm 2)}] \,\ $
\Psi^{[M,1;0; \varepsilon]}_{j,k; \varepsilon'}(\tau+1, z_1, z_2,t)
\,\ = \,\ 
e^{\frac{2\pi i}{M}jk} \, 
\Psi^{[M,1;0; \varepsilon+\varepsilon']}_{j,k; \varepsilon'}(\tau, z_1, z_2,t)$
\end{enumerate}
\end{lemma}

\medskip

%(line=388) %%
%(labellemma:2022-1021a) %%
\begin{lemma} \,\ 
\label{lemma:2022-1021a}
For $M \in \nnn$ and $\varepsilon, \, \varepsilon' \in \{0, \frac12\}$, 
the following formulas hold:
\begin{enumerate}
\item[{\rm 1)}] \,\
$\Psi^{[M,1;0; \varepsilon]}_{j+aM,k+bM;\varepsilon'}(\tau, z_1, z_2, 0)
\, = \,\ 
e^{2\pi i(a-b)\varepsilon} \, 
\Psi^{[M,1;0; \varepsilon]}_{j,k;\varepsilon'}(\tau, z_1, z_2, 0)$
\quad for \,\ ${}^{\forall}a, \, {}^{\forall}b \, \in \, \zzz$ 

\item[{\rm 2)}] \,\ 
$\Psi^{[M,1;0; \varepsilon]}_{j,k;\varepsilon'}(\tau, -z_1, -z_2, 0)
\,\ = \,\ - \, 
\Psi^{[M,1;0; \varepsilon]}_{-k,-j;\varepsilon'}(\tau, z_2, z_1, 0)$

\vspace{1mm}

\item[{\rm 3)}] \,\
$\Psi^{[M,1;0; \varepsilon]}_{j,k;\varepsilon'}(\tau, z_2, z_1, 0)
\hspace{8mm} = \hspace{5.5mm} 
\Psi^{[M,1;0; \varepsilon]}_{k,j;\varepsilon'}(\tau, z_1, z_2, 0)$

\vspace{1mm}

\item[{\rm 4)}] \,\
$\Psi^{[M,1;0; \varepsilon]}_{j,k;\varepsilon'}(\tau, -z_1, -z_2, 0)
\,\ = \,\ - \, 
\Psi^{[M,1;0; \varepsilon]}_{-j,-k;\varepsilon'}(\tau, z_1, z_2, 0)$
\end{enumerate}
\end{lemma}

\medskip

Next, in order to describe the characters of integrable 
$\widehat{A}(1,1)$-modules, we consider the following functions 
defined for $m \in \frac12\nnn$ and $s \in \frac12 \zzz$:
%(label=eqn:2022-1015a1)
%(label=eqn:2022-1015a2)
%(label=eqn:2022-1015a3)
\begin{subequations}
{\allowdisplaybreaks
\begin{eqnarray}
\Phi^{(A(1|1))[m,s]}_1(\tau, z_1,z_2,t)
&:=&
e^{-2\pi imt} \sum_{j \in \zzz} \, \frac{
e^{2\pi imj(z_1+z_2)+2\pi is z_1} \, q^{mj^2+sj}
}{(1-e^{2\pi iz_1} q^j)^2}
\label{eqn:2022-1015a1}
\\[2mm]
\Phi^{(A(1|1))[m,s]}_2(\tau, z_1,z_2,t)
&:=&
e^{-2\pi imt} \sum_{j \in \zzz} \, \frac{
e^{-2\pi imj(z_1+z_2)-2\pi is z_2} \, q^{mj^2+sj}
}{(1-e^{-2\pi iz_2} q^j)^2}
\label{eqn:2022-1015a2}
\\[3mm]
\Phi^{(A(1|1))[m,s]}(\tau, z_1,z_2,t)
&:=& \big[\Phi^{(A(1|1))[m,s]}_1
\, - \, \Phi^{(A(1|1))[m,s]}_2\big] \, (\tau, z_1,z_2,t)
\label{eqn:2022-1015a3}
\end{eqnarray}}
\end{subequations}

The following very easy formula plays an important role in our arguments:

\medskip

%(line=384) %%
%(label=n4:note:2022-1223a) %%
%(label=n4:eqn:2022-1223a) %%
\begin{note}
\label{n4:note:2022-1223a}
Let $m \in \frac12\nnn$ and $s \in \frac12 \zzz$. Then 
\begin{equation}
\Phi^{(A(1|1))[m,s]}(\tau, z_1,z_2,t)
-\Phi^{(A(1|1))[m,s+1]}(\tau, z_1,z_2,t)
\,\ = \,\ 
\Phi^{[m,s]}(\tau,z_1,z_2,t)
\label{n4:eqn:2022-1223a}
\end{equation}
\end{note}

\begin{proof} \,\ First we compute
{\allowdisplaybreaks
\begin{eqnarray*}
& & \hspace{-5mm}
\Phi^{(A(1|1))[m,s]}_1(\tau, z_1,z_2,t)
-\Phi^{(A(1|1))[m,s+1]}_1(\tau, z_1,z_2,t)
\\[2mm]
&=& 
e^{-2\pi imt} \bigg\{
\sum_{j \, \in \, \zzz} \frac{
e^{2\pi i\{mj(z_1+z_2)+sz_1\}} q^{mj^2+sj}
}{(1- e^{2\pi iz_1}q^j)^2}
- 
\sum_{j \, \in \, \zzz} 
\underbrace{\frac{e^{2\pi i\{mj(z_1+z_2)+(s+1)z_1\}} 
q^{mj^2+(s+1)j}
}{(1- e^{2\pi iz_1}q^j)^2}}_{\substack{|| \\[0mm] 
{\displaystyle \hspace{-3mm}
\frac{ e^{2\pi iz_1}q^j \cdot 
e^{2\pi i\{mj(z_1+z_2)+sz_1\}} q^{mj^2+sj}
}{(1- e^{2\pi iz_1}q^j)^2}
}}} 
\bigg\} 
\\[0mm]
&=&
e^{-2\pi imt} \, 
\sum_{j \, \in \, \zzz} \frac{
e^{2\pi i\{mj(z_1+z_2)+sz_1\}} \, q^{mj^2+sj}
}{1- e^{2\pi iz_1}q^j}
\,\ = \,\ 
\Phi^{[m, s]}_1 (\tau, z_1, z_2 , t)
\end{eqnarray*}}
and similarly
$$
\Phi^{(A(1|1))[m,s]}_2(\tau, z_1,z_2,t)
-\Phi^{(A(1|1))[m,s+1]}_2(\tau, z_1,z_2,t)
\,\ = \,\ 
\Phi^{[m, s]}_2 (\tau, z_1, z_2 , t)
$$
Thus we obtain the formula \eqref{n4:eqn:2022-1223a},
proving Note \ref{n4:note:2022-1223a}.
\end{proof}

\section{Integrable $\widehat{A}(1,1)$-modules and their characters}
\label{sec:A(11):integrable}
%(label=sec:A(11):integrable)
%(line=448)

We consider the Dynkin diagram  of the affine Lie superalgebra
$\widehat{A}(1,1) = \widehat{(sl(2|2)/\ccc I)}$ \\
%\,\ :$ 
\setlength{\unitlength}{1mm}
\begin{picture}(32,15)
\put(5,4){\circle{3}}
\put(14,4){\circle{3}}
\put(23,4){\circle{3}}
\put(5,4){\makebox(0,0){$\times$}}
%\put(14,4){\makebox(0,0){$\times$}}
\put(23,4){\makebox(0,0){$\times$}}
\put(5,0){\makebox(0,0){$\alpha_1$}}
\put(14,0){\makebox(0,0){$\alpha_2$}}
\put(23,0){\makebox(0,0){$\alpha_3$}}
\put(6.5,4){\line(1,0){6}}
\put(15.5,4){\line(1,0){6}}
\put(14,10){\circle{3}}
%\put(14,10){\makebox(0,0){$\times$}}
\put(18,11){\makebox(0,0){$\alpha_0$}}
\put(6,5){\line(3,2){6.5}}
\put(22,5){\line(-3,2){6.5}}
\put(9.5,1.5){\makebox(0,0){$1$}}
\put(18.5,1.5){\makebox(0,0){$1$}}
\put(6,9){\makebox(0,0){$-1$}}
\put(21,8.5){\makebox(0,0){$-1$}}
\end{picture}
with the inner product $( \,\ | \,\ )$ such that \\
$\Big((\alpha_i|\alpha_j)\Big)_{i,j=0,1,2,3} = 
\left(
\begin{array}{rrrr}
2 & -1 & 0 & -1 \\[0mm]
-1 & 0 & 1 & 0 \\[0mm]
0 & 1 & -2 & 1 \\[0mm]
-1 & 0 & 1 & 0
\end{array} \right)$ . Then the dual Coxeter number 
of $\widehat{A}(1,1)$ is $h^{\vee}=0$. Let $\hhh$ 
(resp. $\overline{\hhh}$) be the Cartan subalgebra of $\widehat{A}(1,1)$ 
(resp. $A(1,1)$) and $\Lambda_0$ be the element in $\hhh^{\ast}$ 
satisfying the conditions 
$(\Lambda_0|\alpha_j) = \delta_{j,0}$ and $(\Lambda_0|\Lambda_0)=0$.
Let $\delta=\sum_{i=0}^3\alpha_i$ be the primitive imaginary root
and $\rho= -\frac12 (\alpha_1+\alpha_3)$ be the Weyl vector.

We put
\begin{equation}
\begin{array}{ccl}
K(m) &:=& -m \\[2mm]
\Lambda^{[K(m),m_2]} &:=& 
K(m)\Lambda_0 \, - \, \dfrac{m_2}{2} \, (\alpha_1+\alpha_3)
\,\ = \,\ 
- \, m \Lambda_0 \, - \, \dfrac{m_2}{2} \, (\alpha_1+\alpha_3)
\end{array}
\label{n4:eqn:2022-1203b}
\end{equation}
Note that the weight $\Lambda^{[K(m),m_2]}$ is atypical 
with respect to $\alpha_1$ and $\alpha_3$, namely \\
$(\Lambda^{[K(m),m_2]}+\rho | \, \alpha_i) =0$ \,\ $(i=1,3)$.

\medskip

%(line=502) %%
%(label=n4:lemma:2022-1206a) %%
\begin{lemma} \,\ 
\label{n4:lemma:2022-1206a}
The weight $\Lambda^{[K(m),m_2]}$ is integrable with respect to 
$\alpha_2$ and $\delta-\alpha_2$ if and only if $m$ and 
$m_2$ are non-negative integers satisfying $m_2 \leq m$.
\end{lemma}

In this paper, an irredicible $\widehat{A}(1,1)$-module 
$L(\Lambda)$ which is integrable with respect to 
$\alpha_2$ and $\delta-\alpha_2$ is called simply an 
\lq \lq integrable" $\widehat{A}(1,1)$-module, 
and $\Lambda$ is called simply an \lq \lq integrable" weight.
For an integrable weight $\Lambda$ which is atypical with respect to 
$\alpha_1$ and $\alpha_3$, the character ${\rm ch}_{\Lambda}^{(+)}$ 
and the supercharacter ${\rm ch}_{\Lambda}^{(-)}$ of $L(\Lambda)$
are obtained by the formulas
%(label=n4:eqn:2022-1207a1) %
%(label=n4:eqn:2022-1207a2) %
\begin{subequations}
{\allowdisplaybreaks
\begin{eqnarray}
\widehat{R}^{(+)} \, {\rm ch}_{L(\Lambda)}^{(+)} &=&
\sum_{w \in \langle r_{\alpha_2}, \, r_{\delta-\alpha_2}\rangle}
\varepsilon(w) \, w\bigg(\frac{e^{\Lambda+\rho}}{
(1+e^{-\alpha_1})(1+e^{-\alpha_3})}\bigg)
\nonumber
\\[1mm]
&=&
\sum_{w \in \langle r_{\alpha_2}\rangle} \varepsilon(w)
w \Bigg(
\sum_{j \in \zzz}t_{j\alpha_2^{\vee}}
\bigg(\frac{e^{\Lambda+\rho}}{
(1+e^{-\alpha_1})(1+e^{-\alpha_3})}\bigg)\Bigg)
\label{n4:eqn:2022-1207a1}
\\[1mm]
\widehat{R}^{(-)} \, {\rm ch}_{L(\Lambda)}^{(-)} &=&
\sum_{w \in \langle r_{\alpha_2}, \, r_{\delta-\alpha_2}\rangle}
\varepsilon(w) \, w\bigg(\frac{e^{\Lambda+\rho}}{
(1-e^{-\alpha_1})(1-e^{-\alpha_3})}\bigg)
\nonumber
\\[1mm]
&=&
\sum_{w \in \langle r_{\alpha_2}\rangle} \varepsilon(w)
w \Bigg(
\sum_{j \in \zzz}t_{j\alpha_2^{\vee}}
\bigg(\frac{e^{\Lambda+\rho}}{
(1-e^{-\alpha_1})(1-e^{-\alpha_3})}\bigg)\Bigg)
\label{n4:eqn:2022-1207a2}
\end{eqnarray}}
\end{subequations}
where $\widehat{R}^{(+)}$ (resp. $\widehat{R}^{(-)}$) is the denominator
(resp. super-denominator) of $\widehat{A}(1,1)$ and \\
$\alpha_2^{\vee} := \frac12 (\alpha_2|\alpha_2)\alpha_2 = -\alpha_2$, 
and $t_{\alpha}$ ($\alpha \in \hhh$) is the linear automorphism 
of $\hhh$ defined, in \cite{K1}, by
$$
t_{\alpha}(\lambda) \, := \, \lambda +(\lambda|\delta)\alpha
- \Big\{\frac{(\alpha|\alpha)}{2}(\lambda|\delta)+(\lambda|\alpha)
\Big\} \, \delta
$$
Putting 
\begin{subequations}
{\allowdisplaybreaks
\begin{eqnarray}
F^{(+)}_{\Lambda+\rho} & := & 
\sum_{j \, \in \, \zzz}t_{j\alpha_2^{\vee}}\bigg(
\frac{e^{\Lambda+\rho}}{(1+ e^{-\alpha_1})(1+e^{-\alpha_3})}\bigg)
\label{n4:eqn:2022-1207b1}
\\[1mm]
F^{(-)}_{\Lambda+\rho} & := & 
\sum_{j \, \in \, \zzz}t_{j\alpha_2^{\vee}}\bigg(
\frac{e^{\Lambda+\rho}}{(1- e^{-\alpha_1})(1-e^{-\alpha_3})}\bigg)
\label{n4:eqn:2022-1207b2}
\end{eqnarray}}
\end{subequations}
the formulas \eqref{n4:eqn:2022-1207a1} and \eqref{n4:eqn:2022-1207a2}
are written as follows:
%(label=n4:eqn:2022-1207c1) %%
%(label=n4:eqn:2022-1207c2) %%
\begin{subequations}
{\allowdisplaybreaks
\begin{eqnarray}
\widehat{R}^{(+)} \, {\rm ch}_{L(\Lambda)}^{(+)} &=& 
F^{(+)}_{\Lambda+\rho}-r_{\alpha_2}(F^{(+)}_{\Lambda+\rho})
\label{n4:eqn:2022-1207c1}
\\[1mm]
\widehat{R}^{(-)} \, {\rm ch}_{L(\Lambda)}^{(-)} &:=& 
F^{(-)}_{\Lambda+\rho}-r_{\alpha_2}(F^{(-)}_{\Lambda+\rho})
\label{n4:eqn:2022-1207c2}
\end{eqnarray}}
\end{subequations}

Noticing that 
\begin{subequations}
{\allowdisplaybreaks
\begin{eqnarray}
t_{j\alpha_2^{\vee}}(\Lambda_0) &=&
\Lambda_0-j\alpha_2+j^2\delta 
\label{n4:eqn:2022-1207d1}
\\[1mm]
t_{j\alpha_2^{\vee}}(\alpha_i) &=&\left\{
\begin{array}{lcl}
\alpha_i+j\delta & & (i \, = \, 1, \, 3) \\[1mm]
\alpha_2-2\delta & & (i \, = \, 2)
\end{array}\right.
\label{n4:eqn:2022-1207d2}
\end{eqnarray}}
\end{subequations}
we have the following:

\medskip

%(line=623) %%
%(label=note:2022-1012e) %%
\begin{note} 
\label{note:2022-1012e}
Let $\Lambda \, = \, \Lambda^{[K(m),m_2]} \, = \, 
-m \Lambda_0-\dfrac{m_2}{2}(\alpha_1+\alpha_3)$. Then 
\begin{enumerate}
\item[{\rm 1)}]
\begin{enumerate}
\item[{\rm (i)}] \quad $F^{(+)}_{\Lambda+\rho} \,\ = \,\ 
e^{-m\Lambda_0}
\sum\limits_{j \, \in \, \zzz} \dfrac{
e^{mj\alpha_2-\frac{m_2+1}{2}(\alpha_1+\alpha_3)} \, 
q^{mj^2+(m_2+1)j}
}{(1+ e^{-\alpha_1}q^j)(1+e^{-\alpha_3}q^j)}$
\item[{\rm (ii)}] \quad $r_{\alpha_2}(F^{(+)}_{\Lambda+\rho}) 
\,\ = \,\ 
e^{-m\Lambda_0} \, 
\sum\limits_{j \, \in \, \zzz} \dfrac{
e^{-mj\alpha_2-\frac{m_2+1}{2}(\alpha_1+2\alpha_2+\alpha_3)} \, 
q^{mj^2+(m_2+1)j}
}{(1+ e^{-\alpha_1-\alpha_2}q^j)(1+e^{-\alpha_2-\alpha_3}q^j)}$
\end{enumerate}
\item[{\rm 2)}]
\begin{enumerate}
\item[{\rm (i)}] \quad $F^{(-)}_{\Lambda+\rho} \,\ = \,\ 
e^{-m\Lambda_0}
\sum\limits_{j \, \in \, \zzz} \dfrac{
e^{mj\alpha_2-\frac{m_2+1}{2}(\alpha_1+\alpha_3)} \, 
q^{mj^2+(m_2+1)j}
}{(1- e^{-\alpha_1}q^j)(1-e^{-\alpha_3}q^j)}$
\item[{\rm (ii)}] \quad $r_{\alpha_2}(F^{(-)}_{\Lambda+\rho}) 
\,\ = \,\ 
e^{-m\Lambda_0} \, 
\sum\limits_{j \, \in \, \zzz} \dfrac{
e^{-mj\alpha_2-\frac{m_2+1}{2}(\alpha_1+2\alpha_2+\alpha_3)} \, 
q^{mj^2+(m_2+1)j}
}{(1- e^{-\alpha_1-\alpha_2}q^j)(1-e^{-\alpha_2-\alpha_3}q^j)}$
\end{enumerate}
\end{enumerate}
where $q = e^{-\delta}$.
\end{note}

\begin{proof} \,\ By \eqref{n4:eqn:2022-1207d1} and 
\eqref{n4:eqn:2022-1207d2}, we have 
\begin{equation}
t_{j\alpha_2^{\vee}}(\Lambda+\rho) \, = \,
-m\Lambda_0+mj\alpha_2-\frac{m_2+1}{2}(\alpha_1+\alpha_3)
-\big\{mj^2+(m_2+1)j\big\} \, \delta
\label{n4:eqn:2022-1207e}
\end{equation}
Then the formulas in Note \ref{note:2022-1012e} 
follow immediately from \eqref{n4:eqn:2022-1207e}.
\end{proof}

\medskip

Define the coordinates on the Cartan subalgebra $\hhh$ of 
$\widehat{A}(1,1)$ by 
\begin{equation}
2\pi i \Big(-\tau\Lambda_0+
\frac{z_2-z_1}{2}(\alpha_1+\alpha_3)-z_1\alpha_2+t\delta\Big) 
\,\ =: \,\ 
(\tau, z_1, z_2, t)
\label{n4:eqn:2022-1203d}
\end{equation}

%(line=690) %%
%(label=n4:note:2022-1203a) %%
\begin{note} \,\ 
\label{n4:note:2022-1203a}
The following formulas hold for $h=(\tau, z_1, z_2,t) \in \hhh$
and $z = (0,z_1, z_2,0) \in \overline{\hhh}$.
\begin{enumerate}
\item[{\rm 1)}] \,\ $\left\{
\begin{array}{lclcl}
e^{-\alpha_1(h)} &=& e^{-\alpha_3(h)} &=& e^{2\pi iz_1} \\[1mm]
e^{-(\alpha_1+\alpha_2)(h)} &=& e^{-(\alpha_2+\alpha_3)(h)} &=& e^{-2\pi iz_2} 
\\[1mm]
e^{-\alpha_2(h)} & & &=& e^{-2\pi i(z_1+z_2)} \\[1mm]
e^{-(\alpha_1+2\alpha_2+\alpha_3)(h)} & & &=& e^{-4\pi iz_2} 
\end{array} \right. $
\item[{\rm 2)}] \quad $(z|z) \,\ = \,\ -2z_1z_2$
\end{enumerate}
\end{note}

Then the formulas in Note \ref{note:2022-1012e} are written 
in these coordinates as follows:

\medskip 
%(line=833) %%
%(label=lemma:2022-1014a) %%
\begin{lemma} \,\ 
\label{lemma:2022-1014a}
Let $\Lambda \, = \, \Lambda^{[K(m),m_2]} \, = \, 
-m \Lambda_0-\dfrac{m_2}{2}(\alpha_1+\alpha_3)$ be integrable. Then
\begin{enumerate}
\item[{\rm 1)}]
\begin{enumerate}
\item[{\rm (i)}] \quad $F^{(+)}_{\Lambda+\rho} (\tau, z_1,z_2,t)
\,\ = \,\ 
(-1)^{m_2+1} \, \Phi^{(A(1|1))[m, m_2+1]}_1
(\tau, z_1+\frac12,z_2+\frac12,t)$
\item[{\rm (ii)}] \quad $(r_{\alpha_2}F^{(+)}_{\Lambda+\rho}) (\tau, z_1,z_2,t)
\,\ = \,\ 
(-1)^{m_2+1} \, \Phi^{(A(1|1))[m, m_2+1]}_2
(\tau, z_1+\frac12,z_2+\frac12,t)$
\end{enumerate}
\item[{\rm 2)}]
\begin{enumerate}
\item[{\rm (i)}] \quad $F^{(-)}_{\Lambda+\rho} (\tau, z_1,z_2,t)
\,\ = \,\ 
\Phi^{(A(1|1))[m, m_2+1]}_1(\tau, z_1,z_2,t)$
\item[{\rm (ii)}] \quad $(r_{\alpha_2}F^{(-)}_{\Lambda+\rho}) (\tau, z_1,z_2,t)
\,\ = \,\ 
\Phi^{(A(1|1))[m, m_2+1]}_2(\tau, z_1,z_2,t)$
\end{enumerate}
\end{enumerate}
\end{lemma}

\begin{proof} 1) (i) \,\ By Note \ref{note:2022-1012e} and 
Note \ref{n4:note:2022-1203a}, we have 
{\allowdisplaybreaks
\begin{eqnarray*}
& & \hspace{-15mm}
F^{(+)}_{\Lambda+\rho}(\tau, z_1,z_2,t) 
\,\ = \,\ 
e^{-2\pi imt}
\sum_{j \, \in \, \zzz} \frac{
e^{2\pi imj(z_1+z_2)+2\pi i(m_2+1)z_1} \, 
q^{mj^2+(m_2+1)j}
}{(1+ e^{2\pi iz_1}q^j)^2}
\\[1mm]
&=&
(-1)^{m_2+1} \, e^{-2\pi imt}
\sum_{j \, \in \, \zzz} \frac{
e^{2\pi imj(z_1+z_2+1)+2\pi i(m_2+1)(z_1+\frac12)} \, 
q^{mj^2+(m_2+1)j}
}{(1- e^{2\pi i(z_1+\frac12)}q^j)^2}
\\[1mm]
&=&
(-1)^{m_2+1} \, \Phi^{(A(1|1))[m, m_2+1]}_1
(\tau, z_1+\tfrac12,z_2+\tfrac12,t)
\end{eqnarray*}}
The proof of the rests is quite similar.
\end{proof}

By this Lemma \ref{lemma:2022-1014a} together with 
\eqref{n4:eqn:2022-1207c1} and \eqref{n4:eqn:2022-1207c2}, 
we obtain the following:

\medskip 

%(line=896)  %%
%(label=lemma:2022-1120a) %%
\begin{lemma} 
\label{lemma:2022-1120a}
Let $m$ and $m_2$ be non-negative integers such that $0 \leq m_2 \leq m$.
Then the character and the super-character of the integrable 
$N=4$ module $L(\Lambda^{[K(m),m_2]})$ are given by the following 
formulas.
\begin{enumerate}
\item[{\rm 1)}] \quad $\big[
\widehat{R}^{(+)} \cdot {\rm ch}^{(+)}_{\Lambda^{[K(m),m_2]}}\big]
(\tau, z_1,z_2,t) \,\ = \,\ 
(-1)^{m_2+1} \, 
\Phi^{(A(1|1))[m,m_2+1]}(\tau, \, z_1+\frac12, \, z_2+\frac12, \, t) $
\item[{\rm 2)}] \quad $\big[
\widehat{R}^{(-)} \cdot {\rm ch}^{(-)}_{\Lambda^{[K(m),m_2]}}\big]
(\tau, z_1,z_2,t) \,\ = \,\ 
\Phi^{(A(1|1))[m,m_2+1]}(\tau, z_1,z_2,t) $
\end{enumerate}
\end{lemma}

\section{Characters of principal admissible representations 
of $\widehat{A}(1,1)$}
\label{sec:A(11):admissible}

\subsection{Principal admissible simple subsets of $\widehat{A}(1,1)$}
%(line=921) %%

The list of principal admissible simple subsets of $\widehat{A}(1,1)$
is shown in \S 8 of \cite{KW2017b} as follows where $k_3=k_1$: 
{\allowdisplaybreaks
\begin{eqnarray*}
\Pi^{(M), \, {\rm (I)}}_{k_1,k_2} \hspace{2mm} &=&\big\{
k_0\delta+\alpha_0, \,\ 
k_1\delta+\alpha_1, \,\ 
k_2\delta+\alpha_2, \,\ 
k_3\delta+\alpha_3\big\}, \,\ M =\sum_{i=0}^3k_i+1
\\[1mm]
\Pi^{(M), \, {\rm (II)}}_{k_1,k_2} \hspace{1mm} &=&\big\{
k_0\delta-\alpha_0, \,\ 
k_1\delta-\alpha_1, \,\ 
k_2\delta-\alpha_2, \,\ 
k_3\delta-\alpha_3\big\}, \,\ M =\sum_{i=0}^3k_i-1
\\[1mm]
\Pi^{(M), \, {\rm (III)}}_{k_1,k_2} &=&\big\{
k_0\delta+\alpha_0, \,\ 
k_1\delta+\alpha_1+\alpha_2, \,\ 
k_2\delta-\alpha_2, \,\ 
k_3\delta+\alpha_2+\alpha_3\big\}, \,\ M =\sum_{i=0}^3k_i+1
\\[1mm]
\Pi^{(M), \, {\rm (IV)}}_{k_1,k_2} &=&\big\{
k_0\delta-\alpha_0, \,\ 
k_1\delta-(\alpha_1+\alpha_2), \,\ 
k_2\delta+\alpha_2, \,\ 
k_3\delta-(\alpha_2+\alpha_3)\big\}, \,\ M =\sum_{i=0}^3k_i-1
\end{eqnarray*}}
Note that the range of the parameters $(k_1,k_2)$ is as follows:
%(label=n4:eqn:2023-101a) %%
\begin{equation} \hspace{-15mm}
\begin{array}{lclcl}
{\rm for} \,\ \Pi^{(M), \, {\rm (I)}}_{k_1,k_2}& : &
k_1, \,\ k_2 \, \geq \, 0 & \text{and} & 2k_1+k_2 \, \leq \, M-1
\\[3mm]
{\rm for} \,\ \Pi^{(M), \, {\rm (II)}}_{k_1,k_2}& : &
k_1, \,\ k_2 \, \geq \, 1 & \text{and} & 2k_1+k_2 \, \leq \, M
\\[3mm]
{\rm for} \,\ \Pi^{(M), \, {\rm (III)}}_{k_1,k_2}& : &
k_1 \, \geq \, 0, \quad k_2 \, \geq \, 1 & \text{and} & 2k_1+k_2 \, \leq \, M-1
\\[3mm]
{\rm for} \,\ \Pi^{(M), \, {\rm (IV)}}_{k_1,k_2}& : &
k_1 \, \geq \, 1, \quad k_2 \, \geq \, 0 & \text{and} & 2k_1+k_2 \, \leq \, M
\end{array}
\label{n4:eqn:2023-101a}
\end{equation}
These principal admissible simple subsets can be written as 
$\Pi^{(M), \, (\heartsuit)}_{k_1,k_2} = 
t_{\beta}\overline{y}(\Pi^{(M), \, {\rm (I)}}_{0,0})$ \,\ 
($\heartsuit \, = $ I $\sim$ IV), where 
$(\overline{y}, \beta)$ are as follows in each case:

%(label=n4:eqn:2022-1203c) %%
%(line=980) %%
\begin{equation}
\begin{array}{lclcl}
{\rm for} \,\ \Pi^{(M), \, {\rm (I)}}_{k_1,k_2}& : &\overline{y}=1, & &
\beta=-\dfrac{2k_1+k_2}{2}(\alpha_1+\alpha_3)-k_1\alpha_2
\\[3mm]
{\rm for} \,\ \Pi^{(M), \, {\rm (II)}}_{k_1,k_2}& : 
&\overline{y}=r_{\theta}r_{\alpha_2}, & &
\beta= \hspace{3mm}
\dfrac{2k_1+k_2}{2}(\alpha_1+\alpha_3)+k_1\alpha_2
\\[3mm]
{\rm for} \,\ \Pi^{(M), \, {\rm (III)}}_{k_1,k_2}& : 
&\overline{y}=r_{\alpha_2}, & &
\beta=-\dfrac{2k_1+k_2}{2}(\alpha_1+\alpha_3)-(k_1+k_2)\alpha_2
\\[3mm]
{\rm for} \,\ \Pi^{(M), \, {\rm (IV)}}_{k_1,k_2}& : 
&\overline{y}=r_{\theta}, & &
\beta= \hspace{3mm}
\dfrac{2k_1+k_2}{2}(\alpha_1+\alpha_3)+(k_1+k_2)\alpha_2
\end{array}
\label{n4:eqn:2022-1203c}
\end{equation}

The following formulas will be used in the next section to compute 
characters of principal admissible $\widehat{A}(1,1)$-modules. 

\medskip

%(line=1004) %%
%(Vol.325 p.73)
%(label=note:Vol.325p.73) %%
\begin{note} 
\label{note:Vol.325p.73}
Let $z \, = \, (z_1, z_2) \, := \, 
(z_2-z_1)\dfrac{\alpha_1+\alpha_3}{2} - z_1\alpha_2
\, \in \, \overline{\hhh}$. Then 
$\overline{y}^{-1}z$ and $\overline{y}^{-1}(z+\tau \beta)$ are 
as follow:
\begin{enumerate}
\item[{\rm (I)}] \,\ {\rm for} \,\ $\Pi^{(M), \, {\rm (I)}}_{k_1,k_2}$
\,\ : \,\
$\left\{
\begin{array}{lcl}
\overline{y}^{-1}z &=& z \,\ = \,\ (z_1, \, z_2) \\[1mm]
\overline{y}^{-1}(z+\tau \beta) &=& 
\big(z_1+k_1\tau, \,\ z_2-(k_1+k_2)\tau\big)
\end{array}\right. $
\item[{\rm (II)}] \,\ {\rm for} \,\ $\Pi^{(M), \, {\rm (II)}}_{k_1,k_2}$
\,\ : \,\
$\left\{
\begin{array}{lcl}
\overline{y}^{-1}z &=& (z_2, \, z_1) \\[1mm]
\overline{y}^{-1}(z+\tau \beta) &=& 
\big(z_2+k_1\tau, \,\ z_1-(k_1+k_2)\tau\big)
\end{array}\right. $
\item[{\rm (III}] \,\ {\rm for} \,\ $\Pi^{(M), \, {\rm (III)}}_{k_1,k_2}$
\,\ : \,\
$\left\{
\begin{array}{lcl}
\overline{y}^{-1}z &=& (-z_2, \, -z_1) \\[1mm]
\overline{y}^{-1}(z+\tau \beta) &=& 
\big(-z_2+k_1\tau, \,\ -z_1-(k_1+k_2)\tau\big)
\end{array}\right. $
\item[{\rm (IV)}] \,\ {\rm for} \,\ $\Pi^{(M), \, {\rm (IV)}}_{k_1,k_2}$
\,\ : \,\
$\left\{
\begin{array}{lcl}
\overline{y}^{-1}z &=& (-z_1, \, -z_2) \\[1mm]
\overline{y}^{-1}(z+\tau \beta) &=& 
\big(-z_1+k_1\tau, \,\ -z_2-(k_1+k_2)\tau\big)
\end{array}\right. $
\end{enumerate}
\end{note}

\medskip %%

%(line=1055) %%
%(Vol.325 p.74) %%
%(label=lemma:Vol.325p.74) %%
\begin{lemma} 
\label{lemma:Vol.325p.74}
In each case $\Pi^{(M), \,  (\heartsuit)}_{k_1,k_2}$  
$(\heartsuit =$ {\rm I} $\sim$ {\rm IV)}, the element
$$
\bigg(M\tau, \,\ \overline{y}^{-1}(z+\tau \beta), \,\ 
\frac{1}{M}\Big(t+(z|\beta)+\frac{\tau |\beta|^2}{2}\Big)\bigg)
$$
in $\hhh$ is explicitly written as follows:
\begin{enumerate}
\item[{\rm (I)}] \,\ {\rm for} \,\ $\Pi^{(M), \, {\rm (I)}}_{k_1,k_2}
\,\ : $
$$
\Big(M\tau, \hspace{4.5mm} 
z_1+k_1\tau, \hspace{4.5mm} 
z_2-(k_1+k_2)\tau, \,\ \dfrac{1}{M}
\big[t+(k_1+k_2)z_1-k_1z_2+k_1(k_1+k_2)\tau\big]\Big)
$$
\item[{\rm (II)}] \,\ {\rm for} \,\ $\Pi^{(M), \, {\rm (II)}}_{k_1,k_2}
\,\ : $ 
$$
\Big(M\tau, \,\ -z_1+k_1\tau, \,\ 
-z_2-(k_1+k_2)\tau, \,\ \dfrac{1}{M}
\big[t-(k_1+k_2)z_1+k_1z_2+k_1(k_1+k_2)\tau\big]\Big)
$$
\item[{\rm (III)}] \,\ {\rm for} \,\ $\Pi^{(M), \, {\rm (III)}}_{k_1,k_2}
\,\ : $ 
$$
\Big(M\tau, \,\ -z_2+k_1\tau, \,\ 
-z_1-(k_1+k_2)\tau, \,\ \dfrac{1}{M}
\big[t+k_1z_1-(k_1+k_2)z_2+k_1(k_1+k_2)\tau\big]\Big)
$$
\item[{\rm (IV)}] \,\ {\rm for} \,\ $\Pi^{(M), \, {\rm (IV)}}_{k_1,k_2}
\,\ : $ 
$$
\Big(M\tau, \hspace{4.5mm} 
z_2+k_1\tau, \hspace{4.5mm} 
z_1-(k_1+k_2)\tau, \,\ \dfrac{1}{M}
\big[t-k_1z_1+(k_1+k_2)z_2+k_1(k_1+k_2)\tau\big]\Big)
$$
\end{enumerate}
\end{lemma}

The following formulas can be checked easily and are used to compute 
$h_{\lambda}$ and $s_{\lambda}$ in section \ref{sec:h-lambda}
and the integrability of principal admissible weights 
with respect to $\alpha_0$ in section \ref{sec:vanishing}.

\medskip

%(line=1111) %%
%(label=note:2022-1128d) %%
\begin{note}
\label{note:2022-1128d}
For $\beta$ and $\overline{y}$ defined by \eqref{n4:eqn:2022-1203c}, 
the following formulas hold:
\begin{enumerate}
\item[{\rm 1)}] \,\ $(t_{\beta}\overline{y}(\alpha_1+\alpha_3)| \, \alpha_2)
\,\ = \,\ 
(\overline{y}(\alpha_1+\alpha_3)| \, \alpha_2)
\,\ = \,\ 
\left\{
\begin{array}{rcl}
2 & \,\ {\rm if} & \heartsuit = {\rm I \,\ or \,\ IV}
\\[1mm]
-2 & \,\ {\rm if} & \heartsuit = {\rm II \,\ or \,\ III}
\end{array}\right. $

\item[{\rm 2)}] \,\ $(\beta| \, \alpha_2) \,\ = \,\ 
\left\{
\begin{array}{rcl}
- \, k_2 & \,\ {\rm if} & \heartsuit = {\rm I \,\ or \,\ IV}
\\[1mm]
k_2 & \,\ {\rm if} & \heartsuit = {\rm II \,\ or \,\ III}
\end{array}\right. $
\end{enumerate}
\end{note}

\medskip

%(line=1141)
%(label=note:2022-1128c)
\begin{note} 
\label{note:2022-1128c}
For $\beta$ and $\overline{y}$ defined by \eqref{n4:eqn:2022-1203c}, 
the following formulas hold:
\begin{enumerate}
\item[{\rm 1)}] \,\ $(t_{\beta}\overline{y}(\alpha_1+\alpha_3)| \, \theta)
\,\ = \,\ 
(\overline{y}(\alpha_1+\alpha_3)| \, \theta)
\,\ = \,\ 
\left\{
\begin{array}{rcl}
2 & \,\ {\rm if} & \heartsuit = {\rm I \,\ or \,\ III} 
\\[1.5mm]
-2 & \,\ {\rm if} & \heartsuit = {\rm II \,\ or \,\ IV} 
\end{array}\right. $
\item[{\rm 2)}] \,\ $(\beta| \, \theta) \,\ = \,\ 
\left\{
\begin{array}{ccl}
-(2k_1+k_2) & \,\ {\rm if} & \heartsuit = {\rm I \,\ or \,\ III} 
\\[1.5mm]
2k_1+k_2 & \,\ {\rm if} & \heartsuit = {\rm II \,\ or \,\ IV}
\end{array}\right. $
\end{enumerate}
\end{note}

\subsection{Characters of principal admissible $\widehat{A}(1,1)$-modules}
%(line=1048) %%

To describe and compute the characters of principal admissible modules 
we use notations and formulas from \cite{KW1989} and section 3 of 
\cite{W2001b}, which hold in the case of Lie superalgebras as well.
For $M, m \in \nnn$ such that ${\rm gcd}(M,m)=1$ and $\heartsuit =$ I $\sim$ IV, 
let $\Lambda^{(M) [K(m),m_2] (\heartsuit)}_{k_1,k_2}$ 
denote the principal admissible weight $\lambda$ of level $\frac{-m}{M}$ 
such that $\Pi_{\lambda}= \Pi^{(M) \, (\heartsuit)}_{k_1, k_2}$ and 
$\lambda^0 = \Lambda^{[K(m), m_2]}$. Then, by Theorem 2.1 in \cite{KW1989}
or by Lemma 3.2.4 in \cite{W2001b}, this weight is given by the 
following formula with $(\beta, \overline{y})$ given in 
\eqref{n4:eqn:2022-1203c}: 
%(label=n4:eqn:2022-1211c) %%
%(label=n4:eqn:2022-1211b) %%
\begin{subequations}
{\allowdisplaybreaks
\begin{eqnarray}
& & \hspace{-15mm}
\Lambda^{(M) [K(m),m_2] (\heartsuit)}_{k_1,k_2}
\,\ = \,\ 
(t_{\beta}\overline{y}).\Big(
\Lambda^{[K(m), m_2]}-(M-1)\Big(-\frac{m}{M}\Big)\Lambda_0\Big)
\label{n4:eqn:2022-1211c}
\\[2mm]
&=&
- \, \frac{m}{M} \Lambda_0
- \frac{m}{M} \beta
- \frac{m_2+1}{2} \overline{y}(\alpha_1+\alpha_3)
- \rho
+ \Big[\frac{mk_1(k_1+k_2)}{M}-k_1(m_2+1)\Big] \delta
\label{n4:eqn:2022-1211b}
\end{eqnarray}}
\end{subequations}

The character of a principal admissible module $L(\lambda)$, where
$\lambda \, = \, 
(t_{\beta}\overline{y}).(\lambda^0-(M-1)(K+h^{\vee})\Lambda_0)$, 
is given by Theorem 3.2 in \cite{KW1989} 
or Theorem 3.3.4 in \cite{W2001b}:
$$
\big[\widehat{R}^{(\pm)} \cdot {\rm ch}_{L(\lambda)}^{(\pm)}\big](\tau,z,t) 
\,\ = \,\ 
\big[\widehat{R}^{(\pm)} \cdot {\rm ch}_{L(\lambda^0)}^{(\pm)}\big]
\Big(M\tau, \,\ \overline{y}^{-1}(z+\tau \beta), \,\ 
\frac{1}{M}\big(t+(z|\beta)+\frac{\tau |\beta|^2}{2}\big)\Big).
$$
Using this formula, the numerators of the character and the super-character 
of the princilal admissible module 
$L(\Lambda^{(M) [K(m),m_2] (\heartsuit)}_{k_1,k_2})$ 
$(\heartsuit = {\rm I} \sim {\rm IV})$ are obtained as folows:

\medskip

%(line=1106) %%
%(label=lemma:2022-1018a) %%
\begin{lemma} \,\ 
\label{lemma:2022-1018a}
\begin{enumerate}
\item[{\rm 1)}] \,\ $\big[\widehat{R}^{(\pm)} \cdot 
{\rm ch}^{(\pm)}_{L(\Lambda^{(M) [K(m),m_2] ({\rm (I))}}_{k_1,k_2})}\big]
(\tau, z_1, z_2,t) \,\ =$
$$
\big(\widehat{R}^{(\pm)} \cdot {\rm ch}^{(\pm)}_{L(\Lambda^{[K(m), m_2]})}\big)
\Big(M\tau, \, 
z_1+k_1\tau, \, 
z_2-(k_1+k_2)\tau, \, 
\dfrac{1}{M}\big[t+(k_1+k_2)z_1-k_1z_2+k_1(k_1+k_2)\tau\big]\Big)
$$
% 2)
\item[{\rm 2)}] \,\ $\big[\widehat{R}^{\pm} \cdot 
{\rm ch}_{L(\Lambda^{(M) [K(m),m_2] ({\rm (II))}}_{k_1,k_2})}\big]
(\tau, z_1, z_2,t) \,\ =$
$$ \hspace{-5mm}
\big(R^{(\pm)} \cdot {\rm ch}^{(\pm)}_{L(\Lambda^{[K(m), m_2]})}\big)
\Big(M\tau, \, 
-z_1+k_1\tau, \, 
-z_2-(k_1+k_2)\tau, \, 
\dfrac{1}{M}\big[t-(k_1+k_2)z_1+k_1z_2+k_1(k_1+k_2)\tau\big]\Big)
$$
% 3)
\item[{\rm 3)}] \,\ $\big[\widehat{R}^{\pm} \cdot 
{\rm ch}_{L(\Lambda^{(M) [K(m),m_2] ({\rm (III))}}_{k_1,k_2})}\big]
(\tau, z_1, z_2,t) \,\ =$
$$ \hspace{-5mm}
\big(R^{(\pm)} \cdot {\rm ch}^{(\pm)}_{L(\Lambda^{[K(m), m_2]})}\big)
\Big(M\tau, \, -z_2+k_1\tau, \, 
-z_1-(k_1+k_2)\tau, \, 
\dfrac{1}{M}\big[t+k_1z_1-(k_1+k_2)z_2+k_1(k_1+k_2)\tau\big]\Big)
$$
% 4)
\item[{\rm 4)}] \,\ $\big[\widehat{R}^{\pm} \cdot 
{\rm ch}_{L(\Lambda^{(M) [K(m),m_2] ({\rm (IV))}}_{k_1,k_2})}\big]
(\tau, z_1, z_2,t) \,\ =$
$$
\big(R^{(\pm)} \cdot {\rm ch}^{(\pm)}_{L(\Lambda^{[K(m), m_2]})}\big)
\Big(M\tau, \, 
z_2+k_1\tau, \,  
z_1-(k_1+k_2)\tau, \, 
\dfrac{1}{M}\big[t-k_1z_1+(k_1+k_2)z_2+k_1(k_1+k_2)\tau\big]\Big)
$$
\end{enumerate}
\end{lemma}

\medskip

Using Lemma \ref{lemma:2022-1120a}, these formulas are rewritten as follows:

\medskip 

%(line=1162)  %%
%(label=n4:prop:2022-1212a) %%
\begin{prop} \,\
\label{n4:prop:2022-1212a}
\begin{enumerate}
\item[{\rm 1)}]  
\begin{enumerate}
\item[{\rm (i)}] \,\ $\big[\widehat{R}^{(+)} \cdot 
{\rm ch}^{(+)}_{L(\Lambda^{(M) [K(m),m_2] {\rm (I))}}_{k_1,k_2})}\big]
(\tau, z_1, z_2,t)
\,\ = \,\ 
(-1)^{m_2+1}
e^{-\frac{2\pi im}{M}[t+(k_1+k_2)z_1-k_1z_2]}$
$$
\times \,\ 
q^{-\frac{m}{M}k_1(k_1+k_2)} \, 
\Phi^{(A(1|1))[m,m_2+1]}
(M\tau, \,\ z_1+k_1\tau+\tfrac12, \,\ z_2-(k_1+k_2)\tau+\tfrac12, \,\ 0)
$$
% 1) (ii)
\item[{\rm (ii)}] \,\ $\big[\widehat{R}^{(+)} \cdot 
{\rm ch}^{(+)}_{L(\Lambda^{(M)[K(m), m_2]{\rm (II)}}_{k_1,k_2})}\big]
(\tau, z_1, z_2,t) 
\,\ = \,\ 
(-1)^{m_2+1}
e^{-\frac{2\pi im}{M}[t-(k_1+k_2)z_1+k_1z_2]} $
$$
\times \,\ 
q^{-\frac{m}{M}k_1(k_1+k_2)} \, 
\Phi^{(A(1|1))[m,m_2+1]}
(M\tau, \,\ -z_1+k_1\tau+\tfrac12, \,\ -z_2-(k_1+k_2)\tau+\tfrac12, \,\ 0)
$$
% 1) (iii)
\item[{\rm (iii)}] \,\ $\big[\widehat{R}^{(+)} \cdot 
{\rm ch}^{(+)}_{L(\Lambda^{(M)[K(m), m_2]{\rm (III)}}_{k_1,k_2})}\big]
(\tau, z_1, z_2,t) 
\,\ = \,\ 
(-1)^{m_2+1}
e^{-\frac{2\pi im}{M}[t+k_1z_1-(k_1+k_2)z_2]} $
$$
\times \,\ 
q^{-\frac{m}{M}k_1(k_1+k_2)} \, 
\Phi^{(A(1|1))[m,m_2+1]}
(M\tau, \,\ -z_2+k_1\tau+\tfrac12, \,\ -z_1-(k_1+k_2)\tau+\tfrac12, \,\ 0)
$$
% 1) (iv)
\item[{\rm (iv)}] \,\ $\big[\widehat{R}^{(+)} \cdot 
{\rm ch}^{(+)}_{L(\Lambda^{(M)[K(m), m_2]{\rm (IV)}}_{k_1,k_2})}\big]
(\tau, z_1, z_2,t) 
\,\ = \,\ 
(-1)^{m_2+1}
e^{-\frac{2\pi im}{M}[t-k_1z_1+(k_1+k_2)z_2]} $
$$
\times \,\ 
q^{-\frac{m}{M}k_1(k_1+k_2)} \, 
\Phi^{(A(1|1))[m,m_2+1]}
(M\tau, \,\ z_2+k_1\tau+\tfrac12, \,\ z_1-(k_1+k_2)\tau+\tfrac12, \,\ 0)
$$
\end{enumerate}
% 2)
\item[{\rm 2)}]  
\begin{enumerate}
\item[{\rm (i)}] \,\ $\big[\widehat{R}^{(-)} \cdot 
{\rm ch}^{(-)}_{L(\Lambda^{(M) [K(m),m_2] {\rm (I))}}_{k_1,k_2})}\big]
(\tau, z_1, z_2,t) \,\ = \,\ 
e^{-\frac{2\pi im}{M}[t+(k_1+k_2)z_1-k_1z_2]} $
$$
\times \,\ 
q^{-\frac{m}{M}k_1(k_1+k_2)} \, 
\Phi^{(A(1|1))[m,m_2+1]}
(M\tau, \, z_1+k_1\tau, \, z_2-(k_1+k_2)\tau, \, 0)
$$
% 2) (ii)
\item[{\rm (ii)}] \,\ $\big[\widehat{R}^{(-)} \cdot 
{\rm ch}^{(-)}_{L(\Lambda^{(M)[K(m), m_2]{\rm (II)}}_{k_1,k_2})}\big]
(\tau, z_1, z_2,t) \,\ = \,\ 
e^{-\frac{2\pi im}{M}[t-(k_1+k_2)z_1+k_1z_2]} $
$$
\times \,\ 
q^{-\frac{m}{M}k_1(k_1+k_2)} \, 
\Phi^{(A(1|1))[m,m_2+1]}
(M\tau, \, -z_1+k_1\tau, \, -z_2-(k_1+k_2)\tau, \, 0)
$$
% 2) (iii)
\item[{\rm (iii)}] \,\ $\big[\widehat{R}^{(-)} \cdot 
{\rm ch}^{(-)}_{L(\Lambda^{(M)[K(m), m_2]{\rm (III)}}_{k_1,k_2})}\big]
(\tau, z_1, z_2,t) \,\ = \,\ 
e^{-\frac{2\pi im}{M}[t+k_1z_1-(k_1+k_2)z_2]} $
$$
\times \,\ 
q^{-\frac{m}{M}k_1(k_1+k_2)} \, 
\Phi^{(A(1|1))[m,m_2+1]}
(M\tau, \, -z_2+k_1\tau, \, -z_1-(k_1+k_2)\tau, \, 0)
$$
% 2) (iv)
\item[{\rm (iv)}] \,\ $\big[\widehat{R}^{(-)} \cdot 
{\rm ch}^{(-)}_{L(\Lambda^{(M)[K(m), m_2]{\rm (IV)}}_{k_1,k_2})}\big]
(\tau, z_1, z_2,t) \,\ = \,\
e^{-\frac{2\pi im}{M}[t-k_1z_1+(k_1+k_2)z_2]} $
$$
\times \,\ 
q^{-\frac{m}{M}k_1(k_1+k_2)} \, 
\Phi^{(A(1|1))[m,m_2+1]}
(M\tau, \, z_2+k_1\tau, \, z_1-(k_1+k_2)\tau, \, 0)
$$
\end{enumerate}
\end{enumerate}
\end{prop}

\subsection{Characters twisted by $r_{\alpha_2}t_{\frac12 \alpha_2}$}
\label{subse:twisted:characters}
%(label=subse:twisted:characters) %%
%(line=1276) %%

%\medskip %%

In this section, we consider the $\widehat{A}(1,1)$-characters 
twisted by $w_0 \, := \, r_{\alpha_2}t_{\frac12 \alpha_2}$.
The action of $w_0$ on $\hhh$ is given by
%(label=n4:eqn:2022-1230a) %%
%(label=eqn:Vol.324p.109) %%
\begin{equation}
\left\{
\begin{array}{lcl}
w_0(\alpha_0) &=& \alpha_0 \\[1mm]
w_0(\alpha_1) &=& \alpha_1+\alpha_2-\frac12 \delta \\[1mm]
w_0(\alpha_2) &=& -\alpha_2+\delta \\[1mm]
w_0(\alpha_3) &=& \alpha_2+\alpha_3-\frac12 \delta
\end{array}\right. 
\quad \text{and} \quad 
w_0(\Lambda_0) \,\ = \,\ 
\Lambda_0 \, - \, \dfrac12 \, \alpha_2 \, + \, \dfrac14 \, \delta
\label{n4:eqn:2022-1230a}
\end{equation}
so
\begin{equation}
w_0(\tau, z_1, z_2,t) \,\ = \,\ \Big(\tau, \,\ 
-z_2-\frac{\tau}{2}, \,\ -z_1-\frac{\tau}{2}, \,\ 
t-\frac{z_1+z_2}{2}-\frac{\tau}{4}\Big)
\label{eqn:Vol.324p.109}
\end{equation}
Then the twisted characters 
%(label=n4:eqn:2022-1212a) %%
\begin{equation}
{\rm ch}^{(\pm) \, {\rm tw}}_{L(\Lambda^{(M)[K(m), m_2](\heartsuit)}_{k_1,k_2})}
(\tau, z_1,z_2,t) \, := \, 
{\rm ch}^{(\pm)}_{L(\Lambda^{(M)[K(m), m_2](\heartsuit)}_{k_1,k_2})}
\big(w_0 (\tau, z_1,z_2,t)\big)
\label{n4:eqn:2022-1212a}
\end{equation}
of the principal admissible $\widehat{A}(1,1)$-modules 
$L(\Lambda^{(M)[K(m), m_2](\heartsuit)}_{k_1,k_2})$ are given 
as follows:

\medskip

%(line=1322) %%
%(label=n4:prop:2022-1212b) %%
\begin{prop} \,\
\label{n4:prop:2022-1212b}
\begin{enumerate}
\item[{\rm 1)}]  
\begin{enumerate}
\item[{\rm (i)}] \,\ $\big[\widehat{R}^{(+){\rm tw}} \cdot 
{\rm ch}^{(+){\rm tw}}_{L(\Lambda^{(M)[K(m), m_2]{\rm (I)}}_{k_1,k_2})}\big]
(\tau, z_1, z_2,t)$
{\allowdisplaybreaks
\begin{eqnarray*}
&=& (-1)^{m_2+1}
e^{-\frac{2\pi im}{M}t} \, 
e^{\frac{2\pi im}{M}[-(k_1-\frac12)z_1+(k_1+k_2+\frac12)z_2]} \, 
q^{-\frac{m}{M}(k_1-\frac12)(k_1+k_2+\frac12)} 
\\[2mm]
& & 
\times \,\ \Phi^{(A(1|1))[m,m_2+1]}
(M\tau, \,\ -z_2+(k_1-\tfrac12)\tau+\tfrac12, \,\ 
-z_1-(k_1+k_2+\tfrac12)\tau+\tfrac12, \,\ 0)
\end{eqnarray*}}
% 1) (ii)
\item[{\rm (ii)}] \,\ $\big[\widehat{R}^{(+){\rm tw}} \cdot 
{\rm ch}^{(+){\rm tw}}_{L(\Lambda^{(M)[K(m), m_2]{\rm (II)}}_{k_1,k_2})}
\big](\tau, z_1, z_2,t)$
{\allowdisplaybreaks
\begin{eqnarray*}
&=&(-1)^{m_2+1}
e^{-\frac{2\pi im}{M}t} \, 
e^{\frac{2\pi im}{M}[(k_1+\frac12)z_1-(k_1+k_2-\frac12)z_2]} \, 
q^{-\frac{m}{M}(k_1+\frac12)(k_1+k_2-\frac12)} \, 
\\[2mm]
& &
\times \,\ \Phi^{(A(1|1))[m,m_2+1]}
(M\tau, \,\ z_2+(k_1+\tfrac12)\tau+\tfrac12, \,\ 
z_1-(k_1+k_2-\tfrac12)\tau+\tfrac12, \,\ 0)
\end{eqnarray*}}
% 1) (iii)
\item[{\rm (iii)}] \,\ $\big[\widehat{R}^{(+){\rm tw}} \cdot 
{\rm ch}^{(+){\rm tw}}_{L(\Lambda^{(M)[K(m), m_2]{\rm (III)}}_{k_1,k_2})}
\big](\tau, z_1, z_2,t)$
{\allowdisplaybreaks
\begin{eqnarray*}
&=&(-1)^{m_2+1}
e^{-\frac{2\pi im}{M}t} \, 
e^{\frac{2\pi im}{M}[-(k_1+k_2-\frac12)z_1+(k_1+\frac12)z_2]} \, 
q^{-\frac{m}{M}(k_1+\frac12)(k_1+k_2-\frac12)}
\\[2mm]
& &
\times \,\ \Phi^{(A(1|1))[m,m_2+1]}
(M\tau, \,\ z_1+(k_1+\tfrac12)\tau+\tfrac12, \,\ 
z_2-(k_1+k_2-\tfrac12)\tau+\tfrac12, \,\ 0)
\end{eqnarray*}}
% 1) (iv)
\item[{\rm (iv)}] \,\ $\big[\widehat{R}^{(+){\rm tw}} \cdot 
{\rm ch}^{(+){\rm tw}}_{L(\Lambda^{(M)[K(m), m_2]{\rm (IV)}}_{k_1,k_2})}
\big](\tau, z_1, z_2,t)$
{\allowdisplaybreaks
\begin{eqnarray*}
&=&(-1)^{m_2+1}
e^{-\frac{2\pi im}{M}t} \, 
e^{\frac{2\pi im}{M}[(k_1+k_2+\frac12)z_1-(k_1-\frac12)z_2]} \, 
q^{-\frac{m}{M}(k_1-\frac12)(k_1+k_2+\frac12)}
\\[2mm]
& &
\times \,\ \Phi^{(A(1|1))[m,m_2+1]}
(M\tau, \,\ -z_1+(k_1-\tfrac12)\tau+\tfrac12, \,\ 
-z_2-(k_1+k_2+\tfrac12)\tau+\tfrac12, \,\ 0)
\end{eqnarray*}}
\end{enumerate}
% 2)
\item[{\rm 2)}]  
\begin{enumerate}
\item[{\rm (i)}] \,\ $\big[\widehat{R}^{(-){\rm tw}} \cdot 
{\rm ch}^{(-){\rm tw}}_{L(\Lambda^{(M)[K(m), m_2]{\rm (I)}}_{k_1,k_2})}\big]
(\tau, z_1, z_2,t)$
{\allowdisplaybreaks
\begin{eqnarray*}
&=&
e^{-\frac{2\pi im}{M}t} \, 
e^{\frac{2\pi im}{M}[-(k_1-\frac12)z_1+(k_1+k_2+\frac12)z_2]} \, 
q^{-\frac{m}{M}(k_1-\frac12)(k_1+k_2+\frac12)} 
\\[2mm]
& & 
\times \,\ \Phi^{(A(1|1))[m,m_2+1]}
(M\tau, \,\ -z_2+(k_1-\tfrac12)\tau, \,\ -z_1-(k_1+k_2+\tfrac12)\tau, \,\ 0)
\end{eqnarray*}}
% 2) (ii)
\item[{\rm (ii)}] \,\ $\big[\widehat{R}^{(-){\rm tw}} \cdot 
{\rm ch}^{(-){\rm tw}}_{L(\Lambda^{(M)[K(m), m_2]{\rm (II)}}_{k_1,k_2})}
\big](\tau, z_1, z_2,t)$
{\allowdisplaybreaks
\begin{eqnarray*}
&=&
e^{-\frac{2\pi im}{M}t} \, 
e^{\frac{2\pi im}{M}[(k_1+\frac12)z_1-(k_1+k_2-\frac12)z_2]} \, 
q^{-\frac{m}{M}(k_1+\frac12)(k_1+k_2-\frac12)} \, 
\\[2mm]
& &
\times \,\ \Phi^{(A(1|1))[m,m_2+1]}
(M\tau, \,\ z_2+(k_1+\tfrac12)\tau, \,\ z_1-(k_1+k_2-\tfrac12)\tau, \,\ 0)
\end{eqnarray*}}
% 2) (iii)
\item[{\rm (iii)}] \,\ $\big[\widehat{R}^{(-){\rm tw}} \cdot 
{\rm ch}^{(-){\rm tw}}_{L(\Lambda^{(M)[K(m), m_2]{\rm (III)}}_{k_1,k_2})}
\big](\tau, z_1, z_2,t)$
{\allowdisplaybreaks
\begin{eqnarray*}
&=&
e^{-\frac{2\pi im}{M}t} \, 
e^{\frac{2\pi im}{M}[-(k_1+k_2-\frac12)z_1+(k_1+\frac12)z_2]} \, 
q^{-\frac{m}{M}(k_1+\frac12)(k_1+k_2-\frac12)}
\\[2mm]
& &
\times \,\ \Phi^{(A(1|1))[m,m_2+1]}
(M\tau, \,\ z_1+(k_1+\tfrac12)\tau, \,\ z_2-(k_1+k_2-\tfrac12)\tau, \,\ 0)
\end{eqnarray*}}
% 2) (iv)
\item[{\rm (iv)}] \,\ $\big[\widehat{R}^{(-){\rm tw}} \cdot 
{\rm ch}^{(-){\rm tw}}_{L(\Lambda^{(M)[K(m), m_2]{\rm (IV)}}_{k_1,k_2})}
\big](\tau, z_1, z_2,t)$
{\allowdisplaybreaks
\begin{eqnarray*}
&=&
e^{-\frac{2\pi im}{M}t} \, 
e^{\frac{2\pi im}{M}[(k_1+k_2+\frac12)z_1-(k_1-\frac12)z_2]} \, 
q^{-\frac{m}{M}(k_1-\frac12)(k_1+k_2+\frac12)}
\\[2mm]
& &
\times \,\ \Phi^{(A(1|1))[m,m_2+1]}
(M\tau, \,\ -z_1+(k_1-\tfrac12)\tau, \,\ -z_2-(k_1+k_2+\tfrac12)\tau, \,\ 0)
\end{eqnarray*}}
\end{enumerate}
\end{enumerate}
\end{prop}

\begin{proof} These formulas are obtained easily from 
\eqref{n4:eqn:2022-1212a} and Proposition \ref{n4:prop:2022-1212a}.
In the case 1) (i), its calculation is as follows:
%(line=1578) %%
{\allowdisplaybreaks
\begin{eqnarray*}
& & \hspace{-10mm}
\big[\widehat{R}^{(+){\rm tw}} \cdot 
{\rm ch}^{(+){\rm tw}}_{L(\Lambda^{(M)[K(m), m_2]{\rm (I)}}_{k_1,k_2})}
\big](\tau, z_1, z_2,t)
\, = \, 
\big[\widehat{R}^{(+)} \cdot 
{\rm ch}^{(+)}_{L(\Lambda^{(M)[K(m), m_2]{\rm (I)}}_{k_1,k_2})}\big]
\big(w_0(\tau, z_1, z_2,t)\big)
\\[2mm]
&=&
\big[\widehat{R}^{(+)} \cdot 
{\rm ch}^{(+)}_{L(\Lambda^{(M)[K(m), m_2]{\rm (I)}}_{k_1,k_2})}\big]
\Big(\tau, \, -z_2-\frac{\tau}{2}, \, -z_1-\frac{\tau}{2}, \, 
t-\frac{z_1+z_2}{2}-\frac{\tau}{4}\Big)
\\[2mm]
&=& (-1)^{m_2+1}
e^{-\frac{2\pi im}{M}[t-\frac{\tau}{4}-\frac{z_1+z_2}{2}+
(k_1+k_2)(-z_2-\frac{\tau}{2})-k_1(-z_1-\frac{\tau}{2})]} \, 
q^{-\frac{m}{M}k_1(k_1+k_2)} \, 
\\[2mm]
& & 
\times \,\ \Phi^{(A(1|1))[m,m_2+1]}
(M\tau, \,\ -z_2-\tfrac{\tau}{2}+k_1\tau+\tfrac12, \,\ 
-z_1-\tfrac{\tau}{2}-(k_1+k_2)\tau+\tfrac12, \,\ 0)
\\[2mm]
&=& (-1)^{m_2+1} \, 
e^{-\frac{2\pi im}{M}t} \, 
e^{\frac{2\pi im}{M}[-(k_1-\frac12)z_1+(k_1+k_2+\frac12)z_2]} \, 
\underbrace{
e^{\frac{2\pi im}{M} \cdot \frac{\tau}{4}} \, 
e^{\frac{\pi im}{M} k_2\tau} \, 
q^{-\frac{m}{M}k_1(k_1+k_2)}}_{\substack{|| \\[0mm] {\displaystyle 
q^{-\frac{m}{M}(k_1-\frac12)(k_1+k_2+\frac12)}
}}} \, 
\\[0mm]
& & 
\times \,\ \Phi^{(A(1|1))[m,m_2+1]}
(M\tau, \,\ -z_2+(k_1-\tfrac12)\tau+\tfrac12, \,\ 
-z_1-(k_1+k_2+\tfrac12)\tau+\tfrac12, \,\ 0)
\end{eqnarray*}}
proving 1) (i). The proof of the rests is quite similar.
\end{proof}

\section{Characters of quantum Hamiltonian reduction}
\label{sec:quantum:character}
%(label=sec:quantum:character) %%
%(line=1631) %%

We now consider the quantum Hamiltonian reduction associated to
the pair $(x=\frac12 \theta, \, f=e_{-\theta})$, where 
$\theta=\alpha_1+\alpha_2+\alpha_3$ is the highest root of the 
finite-dimensional Lie superalgebra $A(1,1)$. Taking a basis
$J_0=\alpha_2^{\vee}=-\alpha_2$ of $\overline{\hhh}^f$, we have 
%(label=n4:eqn:2022-1212b) %%
{\allowdisplaybreaks
\begin{eqnarray}
& &
2\pi i \, \Big\{-\tau \Lambda_0
\, - \, \tau x
\, + \, zJ_0 \, +\, \frac{\tau}{2}(x|x)\delta\Big\}
\nonumber
\\[1mm]
&=&
2\pi i \, \Big\{-\tau \Lambda_0
\, - \, \frac{\tau}{2} (\alpha_1+\alpha_3)
\, - \, \Big(z+\frac{\tau}{2}\Big)\alpha_2
\, + \, \frac{\tau}{4} \, \delta\Big\}
\nonumber
\\[1mm]
&=&
\Big(\tau, \,\ z+\frac{\tau}{2}, \,\ z-\frac{\tau}{2}, \,\
\frac{\tau}{4}\Big)
\label{n4:eqn:2022-1212b}
\end{eqnarray}}
Then the (super)characters of the quantum Hamiltonian reduction 
$H(\lambda)$ of $\widehat{A}(1,1)$-module $L(\lambda)$ 
and its twisted module $H^{\rm tw}(\lambda)$
are obtained by the formulas 
%(label=n4:eqn:2022-1212c) %%
%(label=n4:eqn:2022-1212c2) %%
\begin{subequations}
{\allowdisplaybreaks
\begin{eqnarray}
& & \hspace{-10mm}
\big[\overset{N=4}{R}{}^{(\pm)} \cdot {\rm ch}_{H(\lambda)}^{(\pm)}\big]
(\tau,z) 
\,\ = \,\ 
\big[\widehat{R}^{(\pm)} \cdot {\rm ch}_{L(\lambda)}^{(\pm)}\big]
\Big(2\pi i \, \Big\{-\tau \Lambda_0
- \tau x + zJ_0 + \frac{\tau}{2}(x|x)\delta\Big\}\Big)
\nonumber
\\[2mm]
&=&
\big[\widehat{R}^{(\pm)} \cdot {\rm ch}_{L(\lambda)}^{(\pm)}\big]
\Big(\tau, \,\ z+\frac{\tau}{2}, \,\ z-\frac{\tau}{2}, \,\
\frac{\tau}{4}\Big)
\label{n4:eqn:2022-1212c}
\\[2mm]
& & \hspace{-10mm}
\big[\overset{N=4}{R}{}^{(\pm){\rm tw}} \cdot 
{\rm ch}_{H(\lambda)}^{(\pm){\rm tw}}\big]
(\tau,z) 
\,\ = \,\ 
\big[\widehat{R}^{(\pm){\rm tw}} \cdot 
{\rm ch}_{L(\lambda)}^{(\pm){\rm tw}}\big]
\Big(2\pi i \, \Big\{-\tau \Lambda_0
- \tau x + zJ_0 + \frac{\tau}{2}(x|x)\delta\Big\}\Big)
\nonumber
\\[2mm]
&=&
\big[\widehat{R}^{(\pm){\rm tw}} \cdot 
{\rm ch}_{L(\lambda)}^{(\pm){\rm tw}}\big]
\Big(\tau, \,\ z+\frac{\tau}{2}, \,\ z-\frac{\tau}{2}, \,\
\frac{\tau}{4}\Big)
\label{n4:eqn:2022-1212c2}
\end{eqnarray}}
\end{subequations}
where $\overset{N=4}{R}{}^{(+)}$ (resp.$\overset{N=4}{R}{}^{(-)}$)
is the denominator (resp. super-denominator) and 
$\overset{N=4}{R}{}^{(+) {\rm tw}}$ (resp. $\overset{N=4}{R}{}^{(-) {\rm tw}}$)
is the twisted denominator (resp. twisted super-denominator) of 
the N=4 superconformal algebra. These denominators are denoted also by 
$\overset{N=4}{R}{}^{(\varepsilon)}_{\varepsilon'}$ $(
\varepsilon, \, \varepsilon' \, \in \, \{0, \frac12\})$, by putting 
$\overset{N=4}{R}{}^{(\frac12)}_{\frac12} := \overset{N=4}{R}{}^{(+)}$, 
$\overset{N=4}{R}{}^{(0)}_{\frac12} := \overset{N=4}{R}{}^{(-)}$ and  
$\overset{N=4}{R}{}^{(\frac12)}_{0} := \overset{N=4}{R}{}^{(+) {\rm tw}}$, 
$\overset{N=4}{R}{}^{(0)}_{0} := \overset{N=4}{R}{}^{(-) {\rm tw}}$,
and they are written as follows:
%(label=n4:eqn:2022-1210b) %%
%(line=1695) %%
{\allowdisplaybreaks
\begin{eqnarray}
\overset{N=4}{R}{}^{(\varepsilon)}_{\varepsilon'}(\tau, z) 
&:=& 
(-1)^{2\varepsilon} \, i \,\ 
\eta(\tau)^3 \, \frac{
\vartheta_{11}(\tau, 2z)}{
\vartheta_{1-2\varepsilon', \, 1-2\varepsilon}(\tau,z)^2}
\nonumber
\\[2mm]
&=&
(-1)^{2\varepsilon} \, i \,\ 
\frac{\vartheta_{00}(\tau, z) \, \vartheta_{01}(\tau, z) \, 
\vartheta_{10}(\tau, z) \, \vartheta_{11}(\tau, z)
}{\vartheta_{1-2\varepsilon', \, 1-2\varepsilon}(\tau, \, z)^2}
\label{n4:eqn:2022-1210b}
\end{eqnarray}}
namely,
\begin{subequations}
{\allowdisplaybreaks
\begin{eqnarray}
\overset{N=4}{R}{}^{(+)}(\tau,z) &=& - \, i \,\ 
\eta(\tau)^3 \, 
\frac{\vartheta_{11}(\tau, \, 2z)}{\vartheta_{00}(\tau, z)^2}
\,\ = \,\ 
- \, i \,\ \frac{
\vartheta_{01}(\tau, z) \, \vartheta_{10}(\tau, z) \, \vartheta_{11}(\tau, z)
}{\vartheta_{00}(\tau, z)}
\label{n4:eqn:2022-1219c1}
\\[2mm]
\overset{N=4}{R}{}^{(-)}(\tau,z) &=& i \,
\eta(\tau)^3 \, 
\frac{\vartheta_{11}(\tau, \, 2z)}{\vartheta_{01}(\tau, z)^2}
\,\ = \,\ 
i \,\ \frac{
\vartheta_{00}(\tau, z) \, \vartheta_{10}(\tau, z) \, \vartheta_{11}(\tau, z)
}{\vartheta_{01}(\tau, z)}
\label{n4:eqn:2022-1219c2}
\\[2mm]
\overset{N=4}{R}{}^{(+) \, {\rm tw}}(\tau,z) &=& - \, i \, 
\eta(\tau)^3 \, 
\frac{\vartheta_{11}(\tau, \, 2z)}{\vartheta_{10}(\tau, z)^2}
\,\ = \,\ 
- \, i \,\ \frac{
\vartheta_{00}(\tau, z) \, \vartheta_{01}(\tau, z) \, \vartheta_{11}(\tau, z)
}{\vartheta_{10}(\tau, z)}
\label{n4:eqn:2022-1219c3}
\\[2mm]
\overset{N=4}{R}{}^{(-) \, {\rm tw}}(\tau,z) &=& i \, 
\eta(\tau)^3 \, 
\dfrac{\vartheta_{11}(\tau, \, 2z)}{\vartheta_{11}(\tau, z)^2}
\,\ = \,\ 
i \,\ \frac{
\vartheta_{00}(\tau, z) \, \vartheta_{01}(\tau, z) \, \vartheta_{10}(\tau, z)
}{\vartheta_{11}(\tau, z)}
\label{n4:eqn:2022-1219c4}
\end{eqnarray}}
\end{subequations}
The modular transformation properties of these denominators are given
by the following formulas:
%(label=n4:eqn:2022-1210c1) %%
%(label=n4:eqn:2022-1210c2) %%
\begin{subequations}
{\allowdisplaybreaks
\begin{eqnarray}
\overset{\rm N=4}{R}{}^{(\varepsilon)}_{\varepsilon'}
\Big(-\frac{1}{\tau}, \, \frac{z}{\tau}\Big) 
&=&
(-1)^{4\varepsilon \varepsilon'} \, \tau \, 
e^{\frac{2\pi iz^2}{\tau}} \, 
\overset{\rm N=4}{R}{}^{(\varepsilon')}_{\varepsilon}(\tau, \, z)
\label{n4:eqn:2022-1210c1}
\\[2mm]
\overset{\rm N=4}{R}{}^{(\varepsilon)}_{\varepsilon'}(\tau+1, \, z)
&=&
e^{-\pi i \varepsilon'} \, 
\overset{\rm N=4}{R}{}^{(\varepsilon+\varepsilon')}_{\varepsilon'}(\tau, \, z)
\label{n4:eqn:2022-1210c2}
\end{eqnarray}}
\end{subequations}
Note also that 
\begin{equation}
\overset{N=4}{R}{}^{(\pm)}(\tau,z+\tfrac12)
=\overset{N=4}{R}{}^{(\mp)}(\tau,z) \quad
\text{and} \quad
\overset{N=4}{R}{}^{(\pm) \, {\rm tw}}(\tau,z+\tfrac12)
=\overset{N=4}{R}{}^{(\mp) \, {\rm tw}}(\tau,z)
\label{n4:eqn:2023-111a}
\end{equation}

The numerators of the N=4 superconformal modules 
$H(\Lambda^{(M)[K(m),m_2](\heartsuit)}_{k_1,k_2})$ 
$(\heartsuit = {\rm I} \sim {\rm IV})$ constructed from the quantum 
Hamiltonian reduction of principal admissible $\widehat{A}(1,1)$-modules 
$L(\Lambda^{(M)[K(m),m_2](\heartsuit)}_{k_1,k_2})$
are obtained from the formula \eqref{n4:eqn:2022-1212c}
and they are given as follows:

\medskip

%(line=1810) %%
%(label=n4:prop:2022-1212c) %%
\begin{prop} \,\ 
\label{n4:prop:2022-1212c}
\begin{enumerate}
\item[{\rm 1)}]  
\begin{enumerate}
\item[{\rm (i)}] \,\ $\big[\overset{N=4}{R}{}^{(+)} \cdot 
{\rm ch}^{(+)}_{H(\Lambda^{(M)[K(m),m_2]{\rm (I)}}_{k_1,k_2})}\big](\tau,z)
\,\ = \,\ 
(-1)^{m_2+1} \, 
e^{-\frac{2\pi im}{M}k_2z} \, 
q^{-\frac{m}{M}(k_1+\frac12)(k_1+k_2+\frac12)}$
$$
\times \,\ \Phi^{(A(1|1))[m,m_2+1]}
(M\tau, \,\ z+(k_1+\tfrac12)\tau+\tfrac12, \,\ 
z-(k_1+k_2+\tfrac12)\tau+\tfrac12, \,\ 0)
$$
% 1) (ii)
\item[{\rm (ii)}] \,\ $\big[\overset{N=4}{R}{}^{(+)} \cdot 
{\rm ch}^{(+)}_{H(\Lambda^{(M)[K(m),m_2]{\rm (II)}}_{k_1,k_2})}\big](\tau,z)
\,\ = \,\ 
(-1)^{m_2+1} \, 
e^{\frac{2\pi im}{M}k_2z} \, 
q^{-\frac{m}{M}(k_1-\frac12)(k_1+k_2-\frac12)} $
$$
\times \,\ \Phi^{(A(1|1))[m,m_2+1]}
(M\tau, \,\ -z+(k_1-\tfrac12)\tau+\tfrac12, \,\ 
-z-(k_1+k_2-\tfrac12)\tau+\tfrac12, \,\ 0)
$$
% 1) (iii)
\item[{\rm (iii)}] \,\ $\big[\overset{N=4}{R}{}^{(+)} \cdot 
{\rm ch}^{(+)}_{H(\Lambda^{(M)[K(m),m_2]{\rm (III)}}_{k_1,k_2})}\big](\tau,z)
\,\ = \,\ 
(-1)^{m_2+1} \, 
e^{\frac{2\pi im}{M}k_2z} \, 
q^{-\frac{m}{M}(k_1+\frac12)(k_1+k_2+\frac12)}$
$$
\times \,\ \Phi^{(A(1|1))[m,m_2+1]}
(M\tau, \,\ -z+(k_1+\tfrac12)\tau+\tfrac12, \,\ 
-z-(k_1+k_2+\tfrac12)\tau+\tfrac12, \,\ 0)
$$
% 1) (iv)
\item[{\rm (iv)}] \,\ $\big[\overset{N=4}{R}{}^{(+)} \cdot 
{\rm ch}^{(+)}_{H(\Lambda^{(M)[K(m),m_2]{\rm (IV)}}_{k_1,k_2})}\big](\tau,z)
\,\ = \,\ 
(-1)^{m_2+1} \, 
e^{-\frac{2\pi im}{M}k_2z} \, 
q^{-\frac{m}{M}(k_1-\frac12)(k_1+k_2-\frac12)}$
$$
\times \,\ \Phi^{(A(1|1))[m,m_2+1]}
(M\tau, \,\ z+(k_1-\tfrac12)\tau+\tfrac12, \,\ 
z-(k_1+k_2-\tfrac12)\tau+\tfrac12, \,\ 0)
$$
\end{enumerate}
% 2)
\item[{\rm 2)}]  
\begin{enumerate}
\item[{\rm (i)}] \,\ $\big[\overset{N=4}{R}{}^{(-)} \cdot 
{\rm ch}^{(-)}_{H(\Lambda^{(M)[K(m),m_2]{\rm (I)}}_{k_1,k_2})}\big](\tau,z)
\,\ = \,\ 
e^{-\frac{2\pi im}{M}k_2z} \, 
q^{-\frac{m}{M}(k_1+\frac12)(k_1+k_2+\frac12)}$
$$
\times \,\ \Phi^{(A(1|1))[m,m_2+1]}
(M\tau, \,\ z+(k_1+\tfrac12)\tau, \,\ z-(k_1+k_2+\tfrac12)\tau, \,\ 0)
$$
% 2) (ii)
\item[{\rm (ii)}] \,\ $\big[\overset{N=4}{R}{}^{(-)} \cdot 
{\rm ch}^{(-)}_{H(\Lambda^{(M)[K(m),m_2]{\rm (II)}}_{k_1,k_2})}\big](\tau,z)
\,\ = \,\ 
e^{\frac{2\pi im}{M}k_2z} \, 
q^{-\frac{m}{M}(k_1-\frac12)(k_1+k_2-\frac12)} $
$$
\times \,\ \Phi^{(A(1|1))[m,m_2+1]}
(M\tau, \,\ -z+(k_1-\tfrac12)\tau, \,\ -z-(k_1+k_2-\tfrac12)\tau, \,\ 0)
$$
% 2) (iii)
\item[{\rm (iii)}] \,\ $\big[\overset{N=4}{R}{}^{(-)} \cdot 
{\rm ch}^{(-)}_{H(\Lambda^{(M)[K(m),m_2]{\rm (III)}}_{k_1,k_2})}\big](\tau,z)
\,\ = \,\ 
e^{\frac{2\pi im}{M}k_2z} \, 
q^{-\frac{m}{M}(k_1+\frac12)(k_1+k_2+\frac12)}$
$$
\times \,\ \Phi^{(A(1|1))[m,m_2+1]}
(M\tau, \,\ -z+(k_1+\tfrac12)\tau, \,\ -z-(k_1+k_2+\tfrac12)\tau, \,\ 0)
$$
% 2) (iv)
\item[{\rm (iv)}] \,\ $\big[\overset{N=4}{R}{}^{(-)} \cdot 
{\rm ch}^{(-)}_{H(\Lambda^{(M)[K(m),m_2]{\rm (IV)}}_{k_1,k_2})}\big](\tau,z)
\,\ = \,\ 
e^{-\frac{2\pi im}{M}k_2z} \, 
q^{-\frac{m}{M}(k_1-\frac12)(k_1+k_2-\frac12)}$
$$
\times \,\ \Phi^{(A(1|1))[1,m_2+1]}
(M\tau, \,\ z+(k_1-\tfrac12)\tau, \,\ z-(k_1+k_2-\tfrac12)\tau, \,\ 0)
$$
\end{enumerate}
\end{enumerate}
\end{prop}

\begin{proof} These formulas are obtained easily from \eqref{n4:eqn:2022-1212c}
and Proposition \ref{n4:prop:2022-1212a}. In the case 1) (i), 
its calculation is as follows:
%(line=1699) %%
{\allowdisplaybreaks
\begin{eqnarray*}
& & \hspace{-10mm}
\Big[\overset{N=4}{R}{}^{(+)} \cdot 
{\rm ch}^{(+)}_{H(\Lambda^{(M)[K(m),m_2]{\rm (I)}}_{k_1,k_2})}\Big](\tau,z)
\, = \,\ 
\big[
\widehat{R}^{(+)} \cdot 
{\rm ch}^{(+)}_{L(\Lambda^{(M)[K(m),m_2]{\rm (I)}}_{k_1,k_2})}\big]
\Big(\tau, \, z+\dfrac{\tau}{2}, \, 
z-\dfrac{\tau}{2}, \, \frac{\tau}{4}\Big)
\\[2mm]
&=& (-1)^{m_2+1}
e^{-\frac{2\pi im}{M}[\frac{\tau}{4}+(k_1+k_2)
(z+\frac{\tau}{2})-k_1(z-\frac{\tau}{2})]} \, 
q^{-\frac{m}{M}k_1(k_1+k_2)}
\\[2mm]
& & 
\times \,\ \Phi^{(A(1|1))[m,m_2+1]}
\Big(M\tau, \,\ z+\frac{\tau}{2}+k_1\tau+\frac12, \,\ 
z-\frac{\tau}{2}-(k_1+k_2)\tau+\frac12, \,\ 0\Big)
\\[2mm]
&=& (-1)^{m_2+1} \, 
e^{-\frac{2\pi im}{M}k_2z} \, 
\underbrace{q^{-\frac{m}{4M}} \, q^{-\frac{m}{M}(2k_1+k_2)} \, 
q^{-\frac{m}{M}k_1(k_1+k_2)}}_{\substack{|| \\[0mm] {\displaystyle 
q^{-\frac{m}{M}(k_1+\frac12)(k_1+k_2+\frac12)}
}}} 
\\[2mm]
& & 
\times \,\ \Phi^{(A(1|1))[m,m_2+1]}
(M\tau, \,\ z+(k_1+\tfrac12)\tau+\tfrac12, \,\ 
z-(k_1+k_2+\tfrac12)\tau+\tfrac12, \,\ 0)
\end{eqnarray*}}
proving 1) (i). The proof of the rest cases is quite similar.
\end{proof}

\medskip

Similarly the twisted characters of the quantum Hamiltonian reductions 
are obtained as follows:

\medskip

%(line=1960) %%
%(label=n4:prop:2022-1212d) %%
\begin{prop} \,\ 
\label{n4:prop:2022-1212d}
\begin{enumerate}
\item[{\rm 1)}]  
\begin{enumerate}
\item[{\rm (i)}] \,\ $\big[\overset{N=4}{R}{}^{(+){\rm tw}} \cdot 
{\rm ch}^{(+) \, 
{\rm tw}}_{H(\Lambda^{(M)[K(m),m_2]{\rm (I)}}_{k_1,k_2})}\big](\tau,z)
\,\ = \,\ 
(-1)^{m_2+1} \, 
e^{\frac{2\pi im}{M}(k_2+1)z} \, q^{-\frac{m}{M}k_1(k_1+k_2+1)}$
$$
\times \,\ \Phi^{(A(1|1))[m,m_2+1]}
(M\tau, \,\ -z+k_1\tau+\tfrac12, \,\ -z-(k_1+k_2+1)\tau+\tfrac12, \,\ 0)
$$
% 1) (ii)
\item[{\rm (ii)}] \,\ $\big[\overset{N=4}{R}{}^{(+){\rm tw}} \cdot 
{\rm ch}^{(+) \, 
{\rm tw}}_{H(\Lambda^{(M)[K(m),m_2]{\rm (II)}}_{k_1,k_2})}\big](\tau,z)
\,\ = \,\ 
(-1)^{m_2+1} \, 
e^{-\frac{2\pi im}{M}(k_2-1)z} \, q^{-\frac{m}{M}k_1(k_1+k_2-1)}$
$$
\times \,\ \Phi^{(A(1|1))[m,m_2+1]}
(M\tau, \,\ z+k_1\tau+\tfrac12, \,\ z-(k_1+k_2-1)\tau+\tfrac12, \,\ 0)
$$
% 1) (iii)
\item[{\rm (iii)}] \,\ $\big[\overset{N=4}{R}{}^{(+){\rm tw}} \cdot 
{\rm ch}^{(+) \, 
{\rm tw}}_{H(\Lambda^{(M)[K(m),m_2]{\rm (III)}}_{k_1,k_2})}\big](\tau,z)
\,\ = \,\ 
(-1)^{m_2+1} \, 
e^{-\frac{2\pi im}{M}(k_2-1)z} \, q^{-\frac{m}{M}(k_1+1)(k_1+k_2)}$
$$
\times \,\ \Phi^{(A(1|1))[m,m_2+1]}
(M\tau, \,\ z+(k_1+1)\tau+\tfrac12, \,\ z-(k_1+k_2)\tau+\tfrac12, \,\ 0)
$$
% 1) (iv)
\item[{\rm (iv)}] \,\ $\big[\overset{N=4}{R}{}^{(+){\rm tw}} \cdot 
{\rm ch}^{(+) \, 
{\rm tw}}_{H(\Lambda^{(M)[K(m),m_2]{\rm (IV)}}_{k_1,k_2})}\big](\tau,z)
\,\ = \,\ 
(-1)^{m_2+1} \, 
e^{\frac{2\pi im}{M}(k_2+1)z} \, q^{-\frac{m}{M}(k_1-1)(k_1+k_2)}$
$$
\times \,\ \Phi^{(A(1|1))[m,m_2+1]}
(M\tau, \,\ -z+(k_1-1)\tau+\tfrac12, \,\ -z-(k_1+k_2)\tau+\tfrac12, \,\ 0)
$$
\end{enumerate}
% 2) 
\item[{\rm 2)}]  
\begin{enumerate}
\item[{\rm (i)}] \,\ $\big[\overset{N=4}{R}{}^{(-){\rm tw}} \cdot 
{\rm ch}^{(-) \, 
{\rm tw}}_{H(\Lambda^{(M)[K(m),m_2]{\rm (I)}}_{k_1,k_2})}\big](\tau,z)
\,\ = \,\ 
e^{\frac{2\pi im}{M}(k_2+1)z} \, q^{-\frac{m}{M}k_1(k_1+k_2+1)}$
$$
\times \,\ \Phi^{(A(1|1))[m,m_2+1]}
(M\tau, \,\ -z+k_1\tau, \,\ -z-(k_1+k_2+1)\tau, \,\ 0)
$$
% 2) (ii)
\item[{\rm (ii)}] \,\ $\big[\overset{N=4}{R}{}^{(-){\rm tw}} \cdot 
{\rm ch}^{(-) \, 
{\rm tw}}_{H(\Lambda^{(M)[K(m),m_2]{\rm (II)}}_{k_1,k_2})}\big](\tau,z)
\,\ = \,\ 
e^{-\frac{2\pi im}{M}(k_2-1)z} \, q^{-\frac{m}{M}k_1(k_1+k_2-1)}$
$$
\times \,\ \Phi^{(A(1|1))[m,m_2+1]}
(M\tau, \,\ z+k_1\tau, \,\ z-(k_1+k_2-1)\tau, \,\ 0)
$$
% 2) (iii)
\item[{\rm (iii)}] \,\ $\big[\overset{N=4}{R}{}^{(-){\rm tw}} \cdot 
{\rm ch}^{(-) \, 
{\rm tw}}_{H(\Lambda^{(M)[K(m),m_2]{\rm (III)}}_{k_1,k_2})}\big](\tau,z)
\,\ = \,\ 
e^{-\frac{2\pi im}{M}(k_2-1)z} \, q^{-\frac{m}{M}(k_1+1)(k_1+k_2)}$
$$
\times \,\ \Phi^{(A(1|1))[m,m_2+1]}
(M\tau, \,\ z+(k_1+1)\tau, \,\ z-(k_1+k_2)\tau, \,\ 0)
$$
% 2) (iv)
\item[{\rm (iv)}] \,\ $\big[\overset{N=4}{R}{}^{(-){\rm tw}} \cdot 
{\rm ch}^{(-) \, 
{\rm tw}}_{H(\Lambda^{(M)[K(m),m_2]{\rm (IV)}}_{k_1,k_2})}\big](\tau,z)
\,\ = \,\ 
e^{\frac{2\pi im}{M}(k_2+1)z} \, q^{-\frac{m}{M}(k_1-1)(k_1+k_2)}$
$$
\times \,\ \Phi^{(A(1|1))[m,m_2+1]}
(M\tau, \,\ -z+(k_1-1)\tau, \,\ -z-(k_1+k_2)\tau, \,\ 0)
$$
\end{enumerate}
\end{enumerate}
\end{prop}

\begin{proof} These formulas are obtained easily from \eqref{n4:eqn:2022-1212c}
and Proposition \ref{n4:prop:2022-1212b}. In the case 1) (i), 
its calculation is as follows:
%(line=1948) %%
{\allowdisplaybreaks
\begin{eqnarray*}
& & \hspace{-10mm}
\big[\overset{N=4}{R}{}^{(+){\rm tw}} \cdot 
{\rm ch}^{(+) \, {\rm tw}}_{H(\Lambda^{(M)[K(m),m_2]{\rm (I)}}_{k_1,k_2})}\big](\tau,z)
= \, 
\big[\widehat{R}^{(+){\rm tw}} \cdot 
{\rm ch}^{(+) \, {\rm tw}}_{L(\Lambda^{(M)[K(m),m_2]{\rm (I)}}_{k_1,k_2})}\big]
\Big(\tau, \, z+\dfrac{\tau}{2}, \, 
z-\dfrac{\tau}{2}, \, \frac{\tau}{4}\Big)
\\[2mm]
&=& (-1)^{m_2+1} \, 
e^{-\frac{2\pi im}{M} \cdot \frac{\tau}{4}} \, 
e^{\frac{2\pi im}{M}[-(k_1-\frac12)(z+\frac{\tau}{2})
+(k_1+k_2+\frac12)(z-\frac{\tau}{2})]} \, 
q^{-\frac{1}{M}(k_1-\frac12)(k_1+k_2+\frac12)}
\\[2mm]
& & 
\times \Phi^{(A(1|1))[m,m_2+1]}
\Big(M\tau, 
-\Big(z-\frac{\tau}{2}\Big)+\Big(k_1-\frac12\Big)\tau
+\frac12, 
-\Big(z+\frac{\tau}{2}\Big)-\Big(k_1+k_2+\frac12\Big)\tau
+\frac12, 0\Big)
\\[3mm]
&=&(-1)^{m_2+1} \, 
e^{\frac{2\pi im}{M}(k_2+1)z} 
\underbrace{q^{-\frac{m}{4M}} \, q^{-\frac{m}{2M}(2k_1+k_2)} \, 
q^{-\frac{m}{M}(k_1-\frac12)(k_1+k_2+\frac12)}}_{\substack{|| \\[0mm] 
{\displaystyle q^{-\frac{m}{M}k_1(k_1+k_2+1)}
}}} 
\\[2mm]
& &
\times \,\ \Phi^{(A(1|1))[m,m_2+1]}
(M\tau, \,\ -z+k_1\tau+\tfrac12, \,\ -z-(k_1+k_2+1)\tau+\tfrac12, \,\ 0)
\end{eqnarray*}}
proving 1) (i). The proof of the rests is quite similar.
\end{proof}

\section{$(h_{\lambda}, \, s_{\lambda})$ and 
$(h_{\lambda}^{\rm tw}, \, s_{\lambda}^{\rm tw})$}
\label{sec:h-lambda}
%(label=sec:h-lambda) %%
%(line=1995) %%

Let $h_{\lambda}$ (resp. $h_{\lambda}^{\rm tw}$) be the eigenvalue 
of the Virasoro operator $L_0$ (resp. twisted Virasoro operator 
$L_0^{\rm tw}$) on the highest weight vector in the N=4 module 
$H(\lambda)$ (resp. twisted N=4 module $H^{\rm tw}(\lambda)$), 
and let $s_{\lambda}$ (resp. $s_{\lambda}^{\rm tw}$) be the 
eigenvalue of $J^{\{\alpha_2^{\vee}\}}_0$ 
(resp. $J^{\{\alpha_2^{\vee}\} \, {\rm tw}}_0$)
on the highest weight vector in  $H(\lambda)$ 
(resp. $H^{\rm tw}(\lambda)$). The numbers $s_{\lambda}$ and 
$s_{\lambda}^{\rm tw}$ can be computed by the formulas:
%(label=n4:eqn:2022-1229a1) %%
%(label=n4:eqn:2022-1229a2) %%
%(line=2040) %%
\begin{subequations}
\begin{eqnarray}
s_{\lambda} \hspace{1.5mm} &=& (\lambda \, | \, \alpha_2^{\vee})
\,\ = \,\ - \, (\lambda \, | \, \alpha_2)
\label{n4:eqn:2022-1229a1}
\\[2mm]
s_{\lambda}^{\rm tw} 
&=& (\lambda^{\rm tw} \, | \, \alpha_2^{\vee})-1
\,\ = \,\ - \, (w_0(\lambda) \, | \, \alpha_2)-1
\label{n4:eqn:2022-1229a2}
\end{eqnarray}
\end{subequations}
where the term \lq \lq $-1$" in the RHS of \eqref{n4:eqn:2022-1229a2} 
takes place by applying similar arguments in section 5.4 in 
\cite{KW2005} to $w_0=r_{\alpha_2}t_{\frac12 \alpha_2}$.
For $\lambda= \Lambda^{(M)[K(m), m_2](\heartsuit)}_{k_1,k_2}$ 
$(\heartsuit =$ I $\sim$ IV), the numbers $s_{\lambda}$ and 
$s_{\lambda}^{\rm tw}$ are obtained as follows:

\medskip

%(line=2137) %%
%(label=n4:lemma:2022-1229a) %%
\begin{lemma} 
\label{n4:lemma:2022-1229a}
\label{note:2022-1221a}
Let $\lambda =\Lambda^{(M)[K(m),m_2](\heartsuit)}_{k_1,k_2}$  
$(\heartsuit =$ {\rm I} $\sim$ {\rm IV)}. Then 
\begin{enumerate}
\item[{\rm 1)}] \,\ $s_{\lambda} \,\ = \,\ 
\left\{
\begin{array}{ccl}
-\dfrac{mk_2}{M}+m_2 & \quad & {\rm if} \quad 
\heartsuit \, = \, {\rm I \,\ or \,\ IV}
\\[4mm]
\dfrac{mk_2}{M}-m_2-2 & \quad & {\rm if} \quad 
\heartsuit \, = \, {\rm II \,\ or \,\ III}
\end{array}\right. $
\item[{\rm 2)}] \,\ $s_{\lambda}^{\rm tw} \,\ = \,\ 
\left\{
\begin{array}{rcl}
\dfrac{m(k_2+1)}{M}-m_2-1 & \quad & {\rm if} \quad 
\heartsuit \, = \, {\rm I \,\ or \,\ IV}
\\[4mm]
-\dfrac{m(k_2-1)}{M}+m_2+1 & \quad & {\rm if} \quad 
\heartsuit \, = \, {\rm II \,\ or \,\ III}
\end{array}\right. $
\end{enumerate}
\end{lemma}

\begin{proof} By \eqref{n4:eqn:2022-1211b}, we have
$$
\lambda = 
\Lambda^{(M)[K(m),m_2](\heartsuit)}_{k_1,k_2} \, \equiv \,
-  \frac{m}{M} \, \Lambda_0
-  \frac{m}{M} \beta
-  \frac{m_2+1}{2} \, \overline{y}(\alpha_1+\alpha_3)
-  \rho
\quad {\rm mod} \,\ \ccc \, \delta
$$
so
$$
(\lambda| \, \alpha_2)
\,\ = \,\
\, - \, \dfrac{m}{M} (\beta |\alpha_2)
\, - \, \dfrac{m_2+1}{2} \, (\overline{y}(\alpha_1+\alpha_3)| \, \alpha_2)
\, - \, \underbrace{(\rho|\alpha_2)}_{-1}
$$

\noindent
1) \,\ Using this formula and Note \ref{note:2022-1128d}, 
we compute $(\lambda|\alpha_2)$ as follows:

\medskip

(i) \,\ If \,\ $\heartsuit \, = \, $ I \, or \, IV,
$$
(\lambda \, | \, \alpha_2)
\,\ = \,\
\, - \, \frac{m}{M} 
\underbrace{(\beta |\alpha_2)}_{
\substack{|| \\[0mm] {\displaystyle -k_2
}}}
\, - \, \frac{m_2+1}{2} \, 
\underbrace{(\overline{y}(\alpha_1+\alpha_3)| \, \alpha_2)}_{
\substack{|| \\[0mm] {\displaystyle 2
}}}+1
\,\ = \,\ \frac{mk_2}{M}-m_2
$$

(ii) \,\ if \,\ $\heartsuit \, = \, $ II \, or \, III,
$$
(\lambda \, | \, \alpha_2)
\,\ = \,\
\, - \, \frac{m}{M} 
\underbrace{(\beta |\alpha_2)}_{
\substack{|| \\[0mm] {\displaystyle k_2
}}}
\, - \, \frac{m_2+1}{2} \, 
\underbrace{(\overline{y}(\alpha_1+\alpha_3)| \, \alpha_2)}_{
\substack{|| \\[0mm] {\displaystyle - 2
}}}+1
\,\ = \,\ - \, \frac{mk_2}{M}+m_2+2
$$
Then by \eqref{n4:eqn:2022-1229a1}, we obtain the formulas in 1).

\medskip

\noindent
2) \quad $(w_0(\lambda) \, | \, \alpha_2)
\,\ = \,\
(\lambda \, | \, w_0^{-1}\alpha_2)
\,\ = \,\
(\lambda \, | \, \delta-\alpha_2)
\,\ = \,\ 
-\frac{m}{M} - (\lambda \, | \, \alpha_2)$
{\allowdisplaybreaks
\begin{eqnarray*}
&=&
- \, \frac{m}{M}
\,\ - \,\ 
\left\{
\begin{array}{rcl}
\dfrac{mk_2}{M}-m_2 & \quad & {\rm if} \quad 
\heartsuit \, = \, {\rm I \,\ or \,\ IV}
\\[4mm]
- \, \dfrac{mk_2}{M}+m_2+2 & \quad & {\rm if} \quad 
\heartsuit \, = \, {\rm II \,\ or \,\ III}
\end{array}\right. 
\\[3mm]
&=&
\left\{
\begin{array}{rcl}
- \, \dfrac{m(k_2+1)}{M}+m_2 & \quad & {\rm if} \quad 
\heartsuit \, = \, {\rm I \,\ or \,\ IV}
\\[4mm]
\dfrac{m(k_2-1)}{M}-m_2-2 & \quad & {\rm if} \quad 
\heartsuit \, = \, {\rm II \,\ or \,\ III}
\end{array}\right. 
\end{eqnarray*}}
Then by \eqref{n4:eqn:2022-1229a2}, we obtain the formulas in 2).
\end{proof}

\medskip

The numbers $h_{\lambda}$ and $h_{\lambda}^{\rm tw}$,
for $\lambda= \Lambda^{(M)[K(m), m_2](\heartsuit)}_{k_1,k_2}$ 
$(\heartsuit =$ I $\sim$ IV), are given by Lemma 9.1 in 
\cite{KW2017b} as follows:

\medskip

%(line=2270) %%
%(label=n4:lemma:2022-1214a) %%
\begin{lemma} 
\label{n4:lemma:2022-1214a}
Let $\lambda= \Lambda^{(M)[K(m), m_2](\heartsuit)}_{k_1,k_2}$ 
$(\heartsuit = {\rm I} \, \sim \, {\rm IV})$. Then 
\begin{enumerate}
\item[{\rm 1)}] $h_{\lambda} = \left\{
\begin{array}{ll}
-\dfrac{m}{M} \Big(k_1+\dfrac12\Big)\Big(k_1+k_2+\dfrac12\Big)
+ (m_2+1)\Big(k_1+\dfrac12\Big)
- \dfrac{-\frac{m}{M}+2}{4} 
& \text{{\rm if} \, $\heartsuit =$ {\rm I} {\rm or} {\rm III} }
\\[5mm]
-\dfrac{m}{M} \Big(k_1-\dfrac12\Big)\Big(k_1+k_2-\dfrac12\Big)
+ (m_2+1)\Big(k_1-\dfrac12\Big)
- \dfrac{-\frac{m}{M}+2}{4} 
& \text{{\rm if} \, $\heartsuit =$ {\rm II} {\rm or} {\rm IV} }
\end{array}\right. $ 
\item[{\rm 2)}] $h_{\lambda}^{\rm tw} = \left\{
\begin{array}{lcl}
- \, \dfrac{m}{M} \, k_1(k_1+k_2+1) \, + \, (m_2+1)k_1
\, - \, \dfrac{-\frac{m}{M}+1}{4} & &
\text{{\rm if} \, $\heartsuit =$ {\rm I}}
\\[3mm]
- \, \dfrac{m}{M} \, k_1(k_1+k_2-1) \, + \, (m_2+1)k_1
\, - \, \dfrac{-\frac{m}{M}+1}{4} & &
\text{{\rm if} \, $\heartsuit =$ {\rm II}}
\\[3mm]
- \, \dfrac{m}{M} \, (k_1+1)(k_1+k_2) \, + \, (m_2+1)(k_1+1)
\, - \, \dfrac{-\frac{m}{M}+1}{4} & &
\text{{\rm if} \, $\heartsuit =$ {\rm III}}
\\[3mm]
- \, \dfrac{m}{M} \, (k_1-1)(k_1+k_2) \, + \, (m_2+1)(k_1-1)
\, - \, \dfrac{-\frac{m}{M}+1}{4} & &
\text{{\rm if} \, $\heartsuit =$ {\rm IV}}
\end{array} \right. $
\end{enumerate}
\end{lemma}

\medskip

We note that 
$$
\text{the central charge of} \, H(\lambda) \,\ = \,\ 
-6 \, \times \, \big\{
\text{the central charge of} \, L(\lambda) \, +1\big\}
$$
so
\begin{equation}
\text{the central charge of} \, 
H(\Lambda^{(M)[K(m),m_2](\heartsuit)}_{k_1,k_2})
\,\ = \,\ 
-6 \, \Big(-\frac{m}{M}+1\Big)
\,\ = \,\ 
6 \, \Big(\frac{m}{M}-1\Big)
\label{n4:eqn:2023-102a}
\end{equation}

From \eqref{n4:eqn:2023-101a} and Lemmas 
\ref{n4:lemma:2022-1229a} and \ref{n4:lemma:2022-1214a}, we 
obtain the equivalence of N=4 modules:

\medskip

%(line=2341) %%
%(label=n4:prop:2023-102a) %%
\begin{prop} 
\label{n4:prop:2023-102a}
Let $M$ and $m$ be coprime positive integers ,and 
$m_2$ be a non-negative integer such that $0 \leq m_2 \leq m$, 
and $k_1$ and  $k_2$ be integers satisfying 
\eqref{n4:eqn:2023-101a}. Then 
\begin{enumerate}
\item[{\rm 1)}] if \, $k_1, k_2 \geq 0$ \, and \, $2k_1+k_2 \leq M-2$,
$\left\{
\begin{array}{lcl}
H(\Lambda^{(M)[K(m),m_2]({\rm I})}_{k_1,k_2}) \hspace{-2mm}
&\cong& \hspace{-2mm}
H(\Lambda^{(M)[K(m),m_2]({\rm IV})}_{k_1+1,k_2})
\\[3mm]
H^{\rm tw}(\Lambda^{(M)[K(m),m_2]({\rm I})}_{k_1,k_2}) \hspace{-2mm}
&\cong& \hspace{-2mm}
H^{\rm tw}(\Lambda^{(M)[K(m),m_2]({\rm IV})}_{k_1+1,k_2})
\end{array}\right. $
\item[{\rm 2)}] if $\left\{
\begin{array}{l}
k_1\geq 0 \\[1mm]
k_2 \geq 1
\end{array}\right.
$ and \, $2k_1+k_2 \leq M-2$,
$\left\{
\begin{array}{lcl}
H(\Lambda^{(M)[K(m),m_2]({\rm III})}_{k_1,k_2}) \hspace{-2mm}
&\cong& \hspace{-2mm}
H(\Lambda^{(M)[K(m),m_2]({\rm II})}_{k_1+1,k_2})
\\[3mm]
H^{\rm tw}(\Lambda^{(M)[K(m),m_2]({\rm III})}_{k_1,k_2}) \hspace{-2mm}
&\cong& \hspace{-2mm}
H^{\rm tw}(\Lambda^{(M)[K(m),m_2]({\rm II})}_{k_1+1,k_2})
\end{array}\right. $
\end{enumerate}
So we need to consider the characters of 
$H(\Lambda^{(M)[K(m),m_2](\heartsuit)}_{k_1,k_2})$
only for $\heartsuit =$ {\rm I} and {\rm III}.
\end{prop}

\medskip

In the \lq \lq nice" cases of quantum Hamiltonian reduction,
which are going to be discussed in section \ref{sec:nice}, 
these numbers are given by the following formulas:

\medskip

%(line=2387) %%
%(label=prop:2022-1204a) %%
\begin{prop} 
\label{prop:2022-1204a}
Let $\lambda=\Lambda^{(M)[K(1),0](\heartsuit)}_{k_1,k_2}$
such that $\heartsuit = $ {\rm I} or {\rm III} and 
$2k_1+k_2=M-1$. Then $(h_{\lambda}, \, s_{\lambda})$ and 
$(h_{\lambda}^{\rm tw}, \, s_{\lambda}^{\rm tw})$ are as follows:
\begin{enumerate}
\item[{\rm 1)}] \quad $h_{\lambda} \,\ = \,\ 
\dfrac{1}{M}\Big(k_1+\dfrac12\Big)^2 \, + \, \dfrac{1}{4M} \, - \, \dfrac12$ 
\item[{\rm 2)}] \quad $s_{\lambda} \,\ = \,\ \left\{
\begin{array}{rcl}
\dfrac{2k_1+1}{M} -1 & & {\rm if} \quad \heartsuit \,\ = \,\ {\rm I}
\\[4mm]
- \dfrac{2k_1+1}{M} -1 & & {\rm if} \quad \heartsuit \,\ = \,\ {\rm III}
\end{array}\right. $ 
% twisted
\item[$1)^{\rm tw}$] \quad $h_{\lambda}^{\rm tw} \,\ = \,\ \left\{
\begin{array}{rcl}
\dfrac{k_1^2}{M} \, + \, \dfrac{1}{4M}-\dfrac14
& & {\rm if} \quad \heartsuit \,\ = \,\ {\rm I}
\\[4mm]
\dfrac{(k_1+1)^2}{M} \, + \, \dfrac{1}{4M}-\dfrac14
& & {\rm if} \quad \heartsuit \,\ = \,\ {\rm III}
\end{array}\right. $ 
\item[$2)^{\rm tw}$] \quad $s_{\lambda}^{\rm tw} \,\ = \,\ \left\{
\begin{array}{ccl}
- \, \dfrac{2k_1}{M} & & {\rm if} \quad \heartsuit \,\ = \,\ {\rm I}
\\[4mm]
\dfrac{2(k_1+1)}{M} & & {\rm if} \quad \heartsuit \,\ = \,\ {\rm III}
\end{array}\right. $ 
\end{enumerate}
\end{prop}

\begin{proof} \,\ These formulas are obtained easily 
from Lemmas \ref{n4:lemma:2022-1229a} and 
\ref{n4:lemma:2022-1214a} by using $k_2 = M-1-2k_1$ and 
$k_1+k_2 = M-1-k_1$ as follows:

\vspace{2mm}

\noindent
1) \quad $h_{\lambda} \, = \, 
- \, \dfrac{1}{M}\Big(k_1+\dfrac12\Big)\Big(k_1+k_2+\dfrac12\Big)
\, + \, k_1
\, + \, \dfrac{1}{4M}$
$$
= \, 
\frac{1}{M}\Big(k_1+\frac12\Big)\Big(k_1+\frac12-M
\Big)
+ \frac{1}{4M}
+ k_1
\,\ = \,\ 
\frac{1}{M}\Big(k_1+\frac12\Big)^2
+ \frac{1}{4M}
- \frac12
$$
proving 1).

\vspace{3mm}

\noindent
$1)^{\rm tw}$ : \quad If \,\ $\heartsuit \, =$ I, 
{\allowdisplaybreaks
\begin{eqnarray*}
h_{\lambda}^{\rm tw} &=& 
- \, \dfrac{1}{M} \, k_1(k_1+k_2+1) \, + \, k_1
\, - \, \frac{-\frac{1}{M}+1}{4}
\,\ = \,\ 
\frac{k_1(k_1-M)}{M}+k_1+\frac{1}{4M}-\frac14
\\[2mm]
&=&
\frac{k_1^{\, 2}}{M}+\frac{1}{4M}-\frac14
\end{eqnarray*}}

If \,\ $\heartsuit \, =$ III, 
{\allowdisplaybreaks
\begin{eqnarray*}
h_{\lambda}^{\rm tw}
&=& 
- \, \dfrac{1}{M} \, (k_1+1)(k_1+k_2) 
\,\ + \,\ (k_1+1)
\, - \, \dfrac{-\frac{1}{M}+1}{4}
\\[2mm]
&=& \frac{1}{M}(k_1+1)(k_1+1-M)
+ (k_1+1)
+ \frac{1}{4M}
- \frac14
\,\ = \,\ 
\frac{1}{M}(k_1+1)^2 + \frac{1}{4M} - \frac14
\end{eqnarray*}}
proving $1)^{\rm tw}$.

\vspace{3mm}

\noindent
2) \,\ If \,\ $\heartsuit \, =$ I, \quad $
s_{\lambda} \, = \, - \, \dfrac{k_2}{M}
\,\ = \,\ - \, \dfrac{M-1-2k_1}{M}
\,\ = \,\ \dfrac{2k_1+1}{M}-1$

\vspace{2mm}

If \,\ $\heartsuit \, =$ III, \quad  $
s_{\lambda} \, = \, \dfrac{k_2}{M}-2
\,\ = \,\ \dfrac{M-1-2k_1}{M}-2
\,\ = \,\ -\dfrac{2k_1+1}{M}-1$,

\medskip

\noindent
proving 2).

\vspace{2mm}

\noindent
$2)^{\rm tw}$ : \quad If \,\ $\heartsuit \, =$ I, \quad 
$s_{\lambda}^{\rm tw} \, = \, 
\dfrac{k_2+1}{M}-1
\,\ = \,\ 
\dfrac{1}{M}(M-2k_1)- 1
\,\ = \,\ 
- \, \dfrac{2k_1}{M}$

\vspace{3mm}

If \,\ $\heartsuit \, =$ III, \quad 
$s_{\lambda}^{\rm tw} \, = \, -\dfrac{k_2-1}{M}+1
\,\ = \,\ - \dfrac{1}{M}(M-2-2k_1) + 1
\,\ = \,\ \dfrac{2(k_1+1)}{M}$

\vspace{2mm}

\noindent
proving $2)^{\rm tw}$.
\end{proof}

\medskip

The above Proposition \ref{prop:2022-1204a} can be restated as follows:

\medskip

%(line=2538) %%
%(label=cor:2022-1222a) %%

%(label=eqn:2022-1222a) %%
%(label=eqn:2022-1222a2) %%

%(label=eqn:2022-1222b) %%
%(label=eqn:2022-1222b2) %%
\begin{cor} 
\label{cor:2022-1222a}
Let $\lambda=\Lambda^{(M)[K(1),0](\heartsuit)}_{k_1,k_2}$
such that $\heartsuit = $ {\rm I} or {\rm III} and 
$2k_1+k_2=M-1$. Then 
\begin{enumerate}
\item[{\rm 1)}] \,\ Putting 
\vspace{-5mm}
\begin{subequations}
\begin{equation}
j \,\ := \,\ \left\{
\begin{array}{ccl}
k_1+\frac12 &\quad & {\rm if} \quad \heartsuit \, = \, {\rm I} 
\\[2mm]
-(k_1+\frac12) & & {\rm if} \quad \heartsuit \, = \, {\rm III} 
\end{array}\right. 
\label{eqn:2022-1222a}
\end{equation}
we have
\begin{enumerate}
\item[{\rm (i)}] \quad $j \, \in \, \frac12 \, \zzz_{\rm odd}$ 
\quad {\rm s.t.} \quad 
$\left\{
\begin{array}{ccl}
\dfrac12 \, \leq \, j \, \leq \, \dfrac{M}{2}
& \quad & {\rm if} \quad \heartsuit \, = \, {\rm I} 
\\[4mm]
- \, \dfrac{M-1}{2} \, \leq \, j \, \leq \, - \, \dfrac12
& & {\rm if} \quad \heartsuit \, = \, {\rm III} 
\end{array}\right. $

\vspace{-13mm}

\begin{equation}
\,\
\label{eqn:2022-1222a2}
\end{equation}

\vspace{4mm}

\item[{\rm (ii)}] \quad $h_{\lambda} \,\ = \,\ 
\dfrac{j^2}{M} \, + \, \dfrac{1}{4M} \, - \, \dfrac12$
\item[{\rm (iii)}] \quad $s_{\lambda} \,\ = \,\ \dfrac{2j}{M}-1$
\end{enumerate}
\end{subequations}
\item[{\rm 2)}] Putting

\vspace{-8mm}

\begin{subequations}
\begin{equation}
j \,\ := \,\ \left\{
\begin{array}{ccl}
- \, k_1 &\quad & {\rm if} \quad \heartsuit \, = \, {\rm I} 
\\[2mm]
k_1+1 & & {\rm if} \quad \heartsuit \, = \, {\rm III} 
\end{array}\right. 
\label{eqn:2022-1222b}
\end{equation}
we have 
\begin{enumerate}
\item[{\rm (i)}] \quad $j \, \in \, \zzz$ \quad {\rm s.t.} \quad 
$\left\{
\begin{array}{ccl}
- \, \dfrac{M-1}{2} \, \leq \, j \, \leq \, 0
& \quad & {\rm if} \quad \heartsuit \, = \, {\rm I} 
\\[3mm]
1 \, \leq \, j \, \leq \, \dfrac{M}{2}
& & {\rm if} \quad \heartsuit \, = \, {\rm III} 
\end{array}\right. $

\vspace{-13mm}

\begin{equation}
\,\
\label{eqn:2022-1222b2}
\end{equation}

\vspace{4mm}
\item[{\rm (ii)}] \quad $h_{\lambda}^{\rm tw} \,\ = \,\ 
\dfrac{j^2}{M} \, + \, \dfrac{1}{4M} \, - \, \dfrac14$
\item[{\rm (iii)}] \quad $s_{\lambda}^{\rm tw} \,\ = \,\ \dfrac{2j}{M}$
\end{enumerate}
\end{subequations}
\end{enumerate}
\end{cor}

\begin{proof} The claims 1)(i) and 2)(i) as to the range of $j$ follow 
from \eqref{n4:eqn:2022-1213a}.
Claims (ii) and (iii) of 1) and 2) follow immediately from 
Proposition \ref{prop:2022-1204a}.
\end{proof}

\section{Non-irreducible N=4 modules}
\label{sec:non-irred}
%(label=sec:non-irred) %%
%(line=2649) %%

In this section we consider the non-irreducible $\widehat{A}(1,1)$-module 
$$
\ddot{L}(\Lambda^{[K(m),m_2]}) \,\ := \,\ 
L(\Lambda^{[K(m),m_2]}) \oplus L(\Lambda^{[K(m),m_2+1]})
$$
where $m \in \nnn$ and $m_2 \in \zzz_{\geq 0}$ such that $m_2 \leq m-1$. 
Since
$$
\Lambda^{[K(m),m_2+1]} \, = \, 
- m\Lambda_0-\dfrac{m_2+1}{2}(\alpha_1+\alpha_3)
\, = \, 
\Lambda^{[m,m_2]} - \dfrac12 (\alpha_1+\alpha_3)
\, = \, 
\Lambda^{[m,m_2]} - \alpha_1
$$
and $\alpha_1$ is an odd root, the parity of the highest weight 
of $L(\Lambda^{[K(m),m_2+1]})$
is opposite to that of $L(\Lambda^{[K(m),m_2]})$. So the 
character and the super-character of $\ddot{L}(\Lambda^{[K(m),m_2]})$
are given by the following formulas:
%(label=n4:eqn:2022-1210d)  %%
%\vspace{-3mm} %%
\begin{equation}
\begin{array}{ccc}
{\rm ch}^{(+)}_{\ddot{L}(\Lambda^{[K(m),m_2]})} &=&
{\rm ch}^{(+)}_{L(\Lambda^{[K(m),m_2]})}
+
{\rm ch}^{(+)}_{L(\Lambda^{[K(m),m_2+1]})}
\\[3mm]
{\rm ch}^{(-)}_{\ddot{L}(\Lambda^{[K(m),m_2]})} &=&
{\rm ch}^{(-)}_{L(\Lambda^{[K(m),m_2]})}
-
{\rm ch}^{(-)}_{L(\Lambda^{[K(m),m_2+1]})}
\end{array}
\label{n4:eqn:2022-1210d}
\end{equation}

\medskip

We consider the corresponding principal admissible $\widehat{A}(1,1)$-modules
$$
\ddot{L}(\Lambda_{k_1,k_2}^{(M)[K(m), m_2]{\rm (\heartsuit)}})
\, := \, 
L(\Lambda_{k_1,k_2}^{(M)[K(m), m_2]{\rm (\heartsuit)}}) \oplus 
L(\Lambda_{k_1,k_2}^{(M)[K(m), m_2+1]{\rm (\heartsuit)}}) \hspace{5mm}
\big(\heartsuit = {\rm I} \sim {\rm IV}\big)
$$
and N=4 modules 
$$
\ddot{H}(\Lambda_{k_1,k_2}^{(M)[K(m), m_2]{\rm (\heartsuit)}})
\, := \, 
H(\Lambda_{k_1,k_2}^{(M)[K(m), m_2]{\rm (\heartsuit)}}) \oplus 
H(\Lambda_{k_1,k_2}^{(M)[K(m), m_2+1]{\rm (\heartsuit)}}) \hspace{5mm}
\big(\heartsuit = {\rm I} \sim {\rm IV}\big)
$$
Then, by \eqref{n4:eqn:2022-1210d}, the characters of these 
$\widehat{A}(1,1)$-modules and N=4 modules are given by 
%(label=n4:eqn:2022-1225a) %%
%(label=n4:eqn:2022-1225b) %%
\begin{equation}
\begin{array}{ccc}
{\rm ch}^{(\pm)}_{\ddot{L}(\Lambda^{(M)[K(m),m_2] (\heartsuit)}_{k_1,k_2})} 
&=&
{\rm ch}^{(\pm)}_{L(\Lambda^{(M)[K(m),m_2] (\heartsuit)}_{k_1,k_2})} 
\, \pm \, 
{\rm ch}^{(\pm)}_{L(\Lambda^{(M)[K(m),m_2+1] (\heartsuit)}_{k_1,k_2})} 
\\[4mm]
{\rm ch}^{(\pm) {\rm tw}}_{\ddot{L}(\Lambda^{(M)[K(m),m_2] (\heartsuit)}_{k_1,k_2})} 
&=&
{\rm ch}^{(\pm) {\rm tw}}_{L(\Lambda^{(M)[K(m),m_2] (\heartsuit)}_{k_1,k_2})} 
\, \pm \, 
{\rm ch}^{(\pm) {\rm tw}}_{L(\Lambda^{(M)[K(m),m_2+1] (\heartsuit)}_{k_1,k_2})} 
\end{array}
\label{n4:eqn:2022-1225a}
\end{equation}
and
\begin{equation}
\begin{array}{ccc}
{\rm ch}^{(\pm)}_{\ddot{H}(\Lambda^{(M)[K(m),m_2] (\heartsuit)}_{k_1,k_2})} 
&=&
{\rm ch}^{(\pm)}_{H(\Lambda^{(M)[K(m),m_2] (\heartsuit)}_{k_1,k_2})} 
\, \pm \, 
{\rm ch}^{(\pm)}_{H(\Lambda^{(M)[K(m),m_2+1] (\heartsuit)}_{k_1,k_2})} 
\\[4mm]
{\rm ch}^{(\pm) {\rm tw}}_{\ddot{H}(\Lambda^{(M)[K(m),m_2] (\heartsuit)}_{k_1,k_2})} 
&=&
{\rm ch}^{(\pm) {\rm tw}}_{H(\Lambda^{(M)[K(m),m_2] (\heartsuit)}_{k_1,k_2})} 
\, \pm \, 
{\rm ch}^{(\pm) {\rm tw}}_{H(\Lambda^{(M)[K(m),m_2+1] (\heartsuit)}_{k_1,k_2})} 
\end{array}
\label{n4:eqn:2022-1225b}
\end{equation}

\medskip

%(line=2746) %%
%(label=n4:lemma:2023-102a) %%
\begin{lemma} 
\label{n4:lemma:2023-102a}
The numerators of these N=4 modules 
$\ddot{H}(\Lambda^{(M)[K(m),m_2] (\heartsuit)}_{k_1,k_2})$ are given as follows:
\begin{enumerate}
\item[{\rm 1)}]
\begin{enumerate}
\item[{\rm (i)}] $\big[\overset{N=4}{R}{}^{(+)} \cdot 
{\rm ch}^{(+)}_{\ddot{H}(\Lambda^{(M)[K(m), m_2]{\rm (I)}}_{k_1,k_2})}
\big](\tau,z)
= \, 
(-1)^{m_2+1} \, \Psi^{[M,m,m_2+1; \frac12]}_{k_1+\frac12, \, -(k_1+k_2+\frac12); \, \frac12}
(\tau, z, z,0)$
% 1) (ii)
\item[{\rm (ii)}] $\big[\overset{N=4}{R}{}^{(+)} \cdot 
{\rm ch}^{(+)}_{\ddot{H}(\Lambda^{(M)[K(m), m_2]{\rm (II)}}_{k_1,k_2})}\big](\tau,z)
= \, 
(-1)^{m_2+1} \, 
\Psi^{[M,m,m_2+1; \frac12]}_{k_1-\frac12, \, -(k_1+k_2-\frac12); \, \frac12}
(\tau, -z, -z,0)$
% 1) (iii)
\item[{\rm (iii)}] $\big[\overset{N=4}{R}{}^{(+)} \cdot 
{\rm ch}^{(+)}_{\ddot{H}(\Lambda^{(M)[K(m), m_2]{\rm (III)}}_{k_1,k_2})}\big](\tau,z)
= \, 
(-1)^{m_2+1} \, 
\Psi^{[M,m,m_2+1; \frac12]}_{k_1+\frac12, \, -(k_1+k_2+\frac12); \, \frac12}
(\tau, -z, -z,0)$
% 1) (iv)
\item[{\rm (iv)}] $\big[
\overset{N=4}{R}{}^{(+)} \cdot 
{\rm ch}^{(+)}_{\ddot{H}(\Lambda^{(M)[K(m), m_2]{\rm (IV)}}_{k_1,k_2})}\big](\tau,z)
= \, 
(-1)^{m_2+1} \, 
\Psi^{[M,m,m_2+1; \frac12]}_{k_1-\frac12, \, -(k_1+k_2-\frac12); \, \frac12}
(\tau, z, z,0)$
\end{enumerate}
% 2)
\item[{\rm 2)}]
\begin{enumerate}
\item[{\rm (i)}] $\big[\overset{N=4}{R}{}^{(-)} \cdot 
{\rm ch}^{(-)}_{\ddot{H}(\Lambda^{(M)[K(m), m_2]{\rm (I)}}_{k_1,k_2})}\big](\tau,z)
\,\ = \,\ 
\Psi^{[M,m,m_2+1; 0]}_{k_1+\frac12, \, -(k_1+k_2+\frac12); \, \frac12}
(\tau, z, z,0)$
% 2) (ii)
\item[{\rm (ii)}] $\big[\overset{N=4}{R}{}^{(-)} \cdot 
{\rm ch}^{(-)}_{\ddot{H}(\Lambda^{(M)[K(m), m_2]{\rm (II)}}_{k_1,k_2})}\big](\tau,z)
\,\ = \,\ 
\Psi^{[M,m,m_2+1; 0]}_{k_1-\frac12, \, -(k_1+k_2-\frac12); \, \frac12}
(\tau, -z, -z,0)$
% 2) (iii)
\item[{\rm (iii)}] $\big[\overset{N=4}{R}{}^{(-)} \cdot 
{\rm ch}^{(-)}_{\ddot{H}(\Lambda^{(M)[K(m), m_2]{\rm (III)}}_{k_1,k_2})}\big](\tau,z)
\,\ = \,\ 
\Psi^{[M,m,m_2+1; 0]}_{k_1+\frac12, \, -(k_1+k_2+\frac12); \, \frac12}
(\tau, -z, -z,0)$
% 2) (iv)
\item[{\rm (iv)}] $\big[\overset{N=4}{R}{}^{(-)} \cdot 
{\rm ch}^{(-)}_{\ddot{H}(\Lambda^{(M)[K(m), m_2]{\rm (IV)}}_{k_1,k_2})}\big](\tau,z)
\,\ = \,\ 
\Psi^{[M,m,m_2+1; 0]}_{k_1-\frac12, \, -(k_1+k_2-\frac12); \, \frac12}
(\tau, z, z,0)$
\end{enumerate}
\end{enumerate}
\end{lemma}

\begin{proof} \,\ These formulas are obtained easily from 
\eqref{n4:eqn:2022-1225b} and 
Note \ref{n4:note:2022-1223a} and
Proposition \ref{n4:prop:2022-1212c} 
and the formula \eqref{eqn:2022-1017a2}.
In the case 1) (i) and 2) (i), its calculation is as follows:

\medskip

\noindent
1) (i) \,\ By \eqref{n4:eqn:2022-1225b} and 
Proposition \ref{n4:prop:2022-1212c}, one has 
%(line=2721) %%
{\allowdisplaybreaks
\begin{eqnarray*}
& & \hspace{-10mm}
\big[\overset{N=4}{R}{}^{(+)} \cdot
{\rm ch}^{(+)}_{\ddot{H}(\Lambda^{(M)[K(m), m_2]{\rm (I)}}_{k_1,k_2})}\big](\tau,z)
\\[1mm]
&=& 
\big[\overset{N=4}{R}{}^{(+)} \cdot
{\rm ch}^{(+)}_{H(\Lambda^{(M)[K(m), m_2]{\rm (I)}}_{k_1,k_2})}\big](\tau,z)
+
\big[\overset{N=4}{R}{}^{(+)} \cdot
{\rm ch}^{(+)}_{H(\Lambda^{(M)[K(m), m_2+1]{\rm (I)}}_{k_1,k_2})}\big](\tau,z)
\\[2mm]
&=&
e^{-\frac{2\pi im}{M}k_2z} \, 
q^{-\frac{m}{M}(k_1+\frac12)(k_1+k_2+\frac12)}
\\[2mm]
& &
\times \,\ \Big\{(-1)^{m_2+1} \, \Phi^{(A(1|1))[m,m_2+1]}
(M\tau, \,\ z+(k_1+\tfrac12)\tau+\tfrac12, \,\ 
z-(k_1+k_2+\tfrac12)\tau+\tfrac12, \,\ 0)
\\[1mm]
& & \hspace{8mm} 
+ \,\ (-1)^{m_2+2} \, \Phi^{(A(1|1))[m,m_2+2]}
(M\tau, \,\ z+(k_1+\tfrac12)\tau+\tfrac12, \,\ 
z-(k_1+k_2+\tfrac12)\tau+\tfrac12, \,\ 0) \Big\}
\end{eqnarray*}}
Then by Note \ref{n4:note:2022-1223a} and \eqref{eqn:2022-1017a2}, 
this is rewitten as follows:
{\allowdisplaybreaks
\begin{eqnarray*}
&=&
(-1)^{m_2+1} \, 
e^{-\frac{2\pi im}{M}k_2z} \, 
q^{-\frac{m}{M}(k_1+\frac12)(k_1+k_2+\frac12)}
\\[3mm]
& &
\times \,\ 
\Phi^{[m,m_2+1]}(M\tau, \,\ z+(k_1+\tfrac12)\tau+\tfrac12, \,\ 
z-(k_1+k_2+\tfrac12)\tau+\tfrac12, \,\ 0)
\\[1mm]
&=&
(-1)^{m_2+1} \, 
\Psi^{[M,m,m_2+1; \frac12]}_{k_1+\frac12, \, -(k_1+k_2+\frac12); \, \frac12}
(\tau, z, z,0)
\end{eqnarray*}}
proving 1) (i).

\medskip

\noindent
2) (i) \,\ By \eqref{n4:eqn:2022-1225b} and 
Proposition \ref{n4:prop:2022-1212c}, one has 
%(line=2784) %%
{\allowdisplaybreaks
\begin{eqnarray*}
& & \hspace{-10mm}
\big[\overset{N=4}{R}{}^{(-)} \cdot
{\rm ch}^{(-)}_{\ddot{H}(\Lambda^{(M)[K(m), m_2]{\rm (I)}}_{k_1,k_2})}\big](\tau,z)
\\[1mm]
&=& 
\big[\overset{N=4}{R}{}^{(-)} \cdot
{\rm ch}^{(-)}_{H(\Lambda^{(M)[K(m), m_2]{\rm (I)}}_{k_1,k_2})}\big](\tau,z)
-
\big[\overset{N=4}{R}{}^{(-)} \cdot
{\rm ch}^{(-)}_{H(\Lambda^{(M)[K(m), m_2]{\rm (I)}}_{k_1,k_2})}\big](\tau,z)
\\[2mm]
&=&
e^{-\frac{2\pi im}{M}k_2z} \, 
q^{-\frac{m}{M}(k_1+\frac12)(k_1+k_2+\frac12)}
\\[2mm]
& &
\times \,\ \Big\{\Phi^{(A(1|1))[m,m_2+1]}
(M\tau, \,\ z+(k_1+\tfrac12)\tau, \,\ z-(k_1+k_2+\tfrac12)\tau, \,\ 0)
\\[1mm]
& & \hspace{2.7mm} 
- \,\ \Phi^{(A(1|1))[m,m_2+2]}
(M\tau, \,\ z+(k_1+\tfrac12)\tau, \,\ 
z-(k_1+k_2+\tfrac12)\tau, \,\ 0) \Big\}
\end{eqnarray*}}
Then by Note \ref{n4:note:2022-1223a} and \eqref{eqn:2022-1017a2}, 
this is rewitten as follows:
{\allowdisplaybreaks
\begin{eqnarray*}
&=&
e^{-\frac{2\pi im}{M}k_2z} \, 
q^{-\frac{m}{M}(k_1+\frac12)(k_1+k_2+\frac12)}
\Phi^{[m,m_2+1]}
(M\tau, \,\ z+(k_1+\tfrac12)\tau, \,\ z-(k_1+k_2+\tfrac12)\tau, \,\ 0)
\\[1mm]
&=&
\Psi^{[M,m,m_2+1; 0]}_{k_1+\frac12, \, -(k_1+k_2+\frac12); \, \frac12}
(\tau, z, z,0)
\end{eqnarray*}}
proving 2) (i).
The proof of the rests is quite similar.
\end{proof}

\medskip

As to the twisted numerators, we have the following formulas:

\medskip

%(line=2931) %%
%(label=n4:lemma:2023-102b) %%
\begin{lemma} 
\label{n4:lemma:2023-102b}
The twisted numerators of N=4 modules 
$\ddot{H}^{\rm tw}(\Lambda^{(M)[K(m),m_2] (\heartsuit)}_{k_1,k_2})$ 
are given as follows:
\begin{enumerate}
\item[{\rm 1)}]
\begin{enumerate}
\item[{\rm (i)}] $\big[\overset{N=4}{R}{}^{(+){\rm tw}} \cdot 
{\rm ch}^{(+){\rm tw}}_{\ddot{H}(\Lambda^{(M)[K(m), m_2]{\rm (I)}}_{k_1,k_2})}
\big](\tau,z)
\,\ = \,\ 
(-1)^{m_2+1} \, 
\Psi^{[M,m,m_2+1; \frac12]}_{k_1, \, -(k_1+k_2+1); \, 0}(\tau, -z, -z,0)$
% 1) (ii)
\item[{\rm (ii)}] $\big[\overset{N=4}{R}{}^{(+){\rm tw}} \cdot 
{\rm ch}^{(+){\rm tw}}_{\ddot{H}(\Lambda^{(M)[K(m), m_2]{\rm (II)}}_{k_1,k_2})}\big](\tau,z)
\,\ = \,\ 
(-1)^{m_2+1} \, 
\Psi^{[M,m,m_2+1; \frac12]}_{k_1, \, -(k_1+k_2-1); \, 0}(\tau, z, z,0)$
% 1) (iii)
\item[{\rm (iii)}] $\big[\overset{N=4}{R}{}^{(+){\rm tw}} \cdot 
{\rm ch}^{(+){\rm tw}}_{\ddot{H}(\Lambda^{(M)[K(m), m_2]{\rm (III)}}_{k_1,k_2})}\big](\tau,z)
\,\ = \,\ 
(-1)^{m_2+1} \, 
\Psi^{[M,m,m_2+1; \frac12]}_{k_1+1, \, -(k_1+k_2); \, 0}(\tau, z, z,0)$
% 1) (iv)
\item[{\rm (iv)}] $\big[
\overset{N=4}{R}{}^{(+){\rm tw}} \cdot 
{\rm ch}^{(+){\rm tw}}_{\ddot{H}(\Lambda^{(M)[K(m), m_2]{\rm (IV)}}_{k_1,k_2})}\big](\tau,z)
\,\ = \,\ 
(-1)^{m_2+1} \, 
\Psi^{[M,m,m_2+1; \frac12]}_{k_1-1, \, -(k_1+k_2); \, 0}(\tau, -z, -z,0)$
\end{enumerate}
% 2)
\item[{\rm 2)}]
\begin{enumerate}
\item[{\rm (i)}] $\big[\overset{N=4}{R}{}^{(-){\rm tw}} \cdot 
{\rm ch}^{(-){\rm tw}}_{\ddot{H}(\Lambda^{(M)[K(m), m_2]{\rm (I)}}_{k_1,k_2})}\big](\tau,z)
\,\ = \,\ 
\Psi^{[M,m,m_2+1; 0]}_{k_1, \, -(k_1+k_2+1); \, 0}(\tau, -z, -z,0)$
% 2) (ii)
\item[{\rm (ii)}] $\big[\overset{N=4}{R}{}^{(-){\rm tw}} \cdot 
{\rm ch}^{(-){\rm tw}}_{\ddot{H}(\Lambda^{(M)[K(m), m_2]{\rm (II)}}_{k_1,k_2})}\big](\tau,z)
\,\ = \,\ 
\Psi^{[M,m,m_2+1; 0]}_{k_1, \, -(k_1+k_2-1); \, 0}(\tau, z, z,0)$
% 2) (iii)
\item[{\rm (iii)}] $\big[\overset{N=4}{R}{}^{(-){\rm tw}} \cdot 
{\rm ch}^{(-){\rm tw}}_{\ddot{H}(\Lambda^{(M)[K(m), m_2]{\rm (III)}}_{k_1,k_2})}\big](\tau,z)
\,\ = \,\ 
\Psi^{[M,m,m_2+1; 0]}_{k_1+1, \, -(k_1+k_2); \, 0}(\tau, z, z,0)$
% 2) (iv)
\item[{\rm (iv)}] $\big[\overset{N=4}{R}{}^{(-){\rm tw}} \cdot 
{\rm ch}^{(-){\rm tw}}_{\ddot{H}(\Lambda^{(M)[K(m), m_2]{\rm (IV)}}_{k_1,k_2})}\big](\tau,z)
\,\ = \,\ 
\Psi^{[M,m,m_2+1; 0]}_{k_1-1, \, -(k_1+k_2); \, 0}(\tau, -z, -z,0)$
\end{enumerate}
\end{enumerate}
\end{lemma}

\begin{proof} \,\ These formulas are obtained easily from 
\eqref{n4:eqn:2022-1225b} and 
Note \ref{n4:note:2022-1223a} and
Proposition \ref{n4:prop:2022-1212d}
and the formula \eqref{eqn:2022-1017a2}.
In the case 1) (i) and 2) (i), its calculation is as follows:

\medskip

\noindent
1) (i) \,\ By \eqref{n4:eqn:2022-1225b} and 
Proposition \ref{n4:prop:2022-1212d}, one has 
%(line=2898) %%
{\allowdisplaybreaks
\begin{eqnarray*}
& & \hspace{-10mm}
\big[\overset{N=4}{R}{}^{(+){\rm tw}} \cdot
{\rm ch}^{(+){\rm tw}}_{\ddot{H}(\Lambda^{(M)[K(m), m_2]{\rm (I)}}_{k_1,k_2})}\big](\tau,z)
\\[1mm]
&=& 
\big[\overset{N=4}{R}{}^{(+){\rm tw}} \cdot
{\rm ch}^{(+){\rm tw}}_{H(\Lambda^{(M)[K(m), m_2]{\rm (I)}}_{k_1,k_2})}\big](\tau,z)
+
\big[\overset{N=4}{R}{}^{(+){\rm tw}} \cdot
{\rm ch}^{(+){\rm tw}}_{H(\Lambda^{(M)[K(m), m_2+1]{\rm (I)}}_{k_1,k_2})}\big](\tau,z)
\\[2mm]
&=&
e^{\frac{2\pi im}{M}(k_2+1)z} \, q^{-\frac{m}{M}k_1(k_1+k_2+1)}
\\[2mm]
& &
\times \,\ \Big\{(-1)^{m_2+1} \, \Phi^{(A(1|1))[m,m_2+1]}
(M\tau, \,\ -z+k_1\tau+\tfrac12, \,\ -z-(k_1+k_2+1)\tau+\tfrac12, \,\ 0)
\\[1mm]
& & \hspace{7mm} 
+ \,\ 
(-1)^{m_2+2} \, \Phi^{(A(1|1))[m,m_2+2]}
(M\tau, \,\ -z+k_1\tau+\tfrac12, \,\ -z-(k_1+k_2+1)\tau+\tfrac12, \,\ 0)
\end{eqnarray*}}
Then by Note \ref{n4:note:2022-1223a} and \eqref{eqn:2022-1017a2}, 
this is rewitten as follows:
{\allowdisplaybreaks
\begin{eqnarray*}
&=&
(-1)^{m_2+1} \, 
e^{\frac{2\pi im}{M}(k_2+1)z} \, q^{-\frac{m}{M}k_1(k_1+k_2+1)}
\\[3mm]
& &
\times \,\ 
\Phi^{[m,m_2+1]}
(M\tau, \,\ -z+k_1\tau+\tfrac12, \,\ -z-(k_1+k_2+1)\tau+\tfrac12, \,\ 0)
\\[1mm]
&=&
(-1)^{m_2+1} \, 
\Psi^{[M,m,m_2+1; \frac12]}_{k_1, \, -(k_1+k_2+1); \, 0}(\tau, -z, -z,0)
\end{eqnarray*}}
proving 1) (i).

\medskip

\noindent
2) (i) \,\ By \eqref{n4:eqn:2022-1225b} and 
Proposition \ref{n4:prop:2022-1212d}, one has 
%(line=3054) %%
{\allowdisplaybreaks
\begin{eqnarray*}
& & \hspace{-10mm}
\big[\overset{N=4}{R}{}^{(-){\rm tw}} \cdot
{\rm ch}^{(-){\rm tw}}_{\ddot{H}(\Lambda^{(M)[K(m), m_2]{\rm (I)}}_{k_1,k_2})}\big](\tau,z)
\\[1mm]
&=& 
\big[\overset{N=4}{R}{}^{(-){\rm tw}} \cdot
{\rm ch}^{(-){\rm tw}}_{H(\Lambda^{(M)[K(m), m_2]{\rm (I)}}_{k_1,k_2})}\big](\tau,z)
-
\big[\overset{N=4}{R}{}^{(-){\rm tw}} \cdot
{\rm ch}^{(-){\rm tw}}_{H(\Lambda^{(M)[K(m), m_2+1]{\rm (I)}}_{k_1,k_2})}\big](\tau,z)
\\[2mm]
&=&
e^{\frac{2\pi im}{M}(k_2+1)z} \, q^{-\frac{m}{M}k_1(k_1+k_2+1)}
\\[2mm]
& &
\times \,\ \Big\{\Phi^{(A(1|1))[m,m_2+1]}
(M\tau, \,\ -z+k_1\tau, \,\ -z-(k_1+k_2+1)\tau, \,\ 0)
\\[1mm]
& & \hspace{2.7mm} 
- \,\ \Phi^{(A(1|1))[m,m_2+2]}
(M\tau, \,\ -z+k_1\tau, \,\ -z-(k_1+k_2+1)\tau, \,\ 0)
\Big\}
\end{eqnarray*}}
Then by Note \ref{n4:note:2022-1223a} and \eqref{eqn:2022-1017a2}, 
this is rewitten as follows:
{\allowdisplaybreaks
\begin{eqnarray*}
&=&
e^{\frac{2\pi im}{M}(k_2+1)z} \, q^{-\frac{m}{M}k_1(k_1+k_2+1)}
\Phi^{[m,m_2+1]}
(M\tau, \,\ -z+k_1\tau, \,\ -z-(k_1+k_2+1)\tau, \,\ 0)
\\[1mm]
&=&
\Psi^{[M,m,m_2+1; 0]}_{k_1, \, -(k_1+k_2+1); \, 0}(\tau, -z, -z,0)
\end{eqnarray*}}
proving 2) (i). The proof of the rests is quite similar.
\end{proof}

\section{Vanishing of the quantum Hamiltonian reduction}
\label{sec:vanishing}
%(line=3128) %%

As to the vanishing of $W(\widehat{\ggg}, f,x)_K$-module $H(\lambda)$ 
and the twisted module $H^{\rm tw}(\lambda)$ obtained from
the quantum Hamiltonian reduction of a highest weight 
$\widehat{\ggg}$-module $L(\lambda)$, 
where $f=e_{-\theta}$ and $x=\frac12\theta$, 
the following lemma holds:

\medskip

%(line=3141) %%
%(label=lemma:2022-1220a) %%
\begin{lemma} 
\label{lemma:2022-1220a}
Assume that $\lambda$ is integrable with respect to 
$\alpha_0 = \delta-\theta$,
namely $(\lambda+\rho|\delta-\theta) \in \nnn$. Then 
\begin{enumerate}
\item[{\rm 1)}] \,\ $H(V) \, = \, \{0\} $
\item[{\rm 2)}] \,\ $H^{\rm tw}(V) \, = \, \{0\} $ \,\ 
if $\ggg=A(1,1)$ and the twist is given by 
$w_0 \, = \, r_{\alpha_2}t_{\frac12 \alpha_2}$. 
\end{enumerate}
\end{lemma}

\begin{proof} In the case $\lambda$ is integrable with respect to 
$\delta-\theta$, the numerator $\widehat{R} \cdot {\rm ch}_{L(\lambda)}$ 
of $L(\lambda)$ is divisible by $1-e^{-(\delta-\theta)}$, namely 
$\widehat{R}^{(+)} \cdot {\rm ch}^{(+)}_{L(\lambda)}$ is written in the following form:
\begin{equation}
\widehat{R}^{(+)} \cdot {\rm ch}_{L(\lambda)}^{(+)} \,\ = \,\ 
(1-e^{-(\delta-\theta)}) \cdot \, g
\qquad {\rm for} \quad {}^{\exists}g \, \in \, 
\ccc [[e^{-\alpha} \,\ ; \,\ \alpha \in \widehat{\Delta}_+]]
\label{n4:eqn:2022-1228a}
\end{equation}
An element $h$ in the Cartan subalgebra of N=4 SCA is given by 
\eqref{n4:eqn:2022-1212b} and the action of $w_0$ is given
by \eqref{n4:eqn:2022-1230a}: 
\begin{eqnarray*}
h \hspace{7mm} &=& 
2\pi i\, \Big\{-\tau(\Lambda_0+x)-z\alpha_2 -\frac{\tau}{4}\delta\Big\}
\\[1mm]
w_0^{-1}(\delta-\theta) &=& \delta-\theta
\end{eqnarray*}
Then, since $(\Lambda+x|\delta-\theta)=0$ and $(\alpha_2|\theta)=0$, one has
\begin{equation}\left\{
\begin{array}{lcccccc}
(\delta-\theta| h) &=& 0 & & & &
\\[2mm]
(\delta-\theta| w_0(h)) &=& (w_0^{-1}(\delta-\theta)| h)
&=& (\delta-\theta| h) &=& 0
\end{array} \right.
\label{n4:eqn:2023-110a}
\end{equation}
Then by \eqref{n4:eqn:2022-1228a} and \eqref{n4:eqn:2023-110a}, one has 
$$
\begin{array}{rclcc}
\overset{N=4}{R}{}^{(+)} \cdot {\rm ch}_{H(\lambda)}^{(+)}(\tau,z) 
&=& 
\widehat{R}^{(+)} \cdot {\rm ch}_{L(\lambda)}^{(+)}(h) &=& 0
\\[3mm]
\overset{N=4}{R}{}^{(+){\rm tw}} \cdot 
{\rm ch}_{H(\lambda)}^{(+){\rm tw}}(\tau,z) 
&=& 
\widehat{R}^{(+)} \cdot {\rm ch}_{L(\lambda)}^{(+)}(w_0(h)) &=& 0
\end{array}
$$
proving Lemma \ref{lemma:2022-1220a}.
\end{proof}

\medskip

%(line=3170) %%
%(label=lemma:2022-1221a) 55
\begin{lemma} 
\label{lemma:2022-1221a}
For $\heartsuit \, = $ {\rm I} $\sim$ {\rm IV}, \,\ 
$(\Lambda^{(M)[K(m),m_2](\heartsuit)}_{k_1,k_2}+\rho \, | \, 
\alpha_0)$ \, is given by the following:
\begin{enumerate}
\item[{\rm 1)}] \, 
$\big(\Lambda^{(M)[K(m),m_2](\heartsuit)}_{k_1,k_2} +\rho \, 
\big| \, \alpha_0\big)
\, = \left\{
\begin{array}{lccl}
-\dfrac{m(2k_1+k_2+1)}{M} \, + \, m_2+1 & & {\rm if} & 
\heartsuit \, = \, {\rm I \,\ or \,\ III}
\\[4mm]
\dfrac{m(2k_1+k_2-1)}{M} \, - \, m_2-1 & & {\rm if} & 
\heartsuit \, = \, {\rm II \,\ or \,\ IV}
\end{array}\right. $
\item[{\rm 2)}] \,\ $
(\Lambda^{(M)[K(m),m_2](\heartsuit)}_{k_1,k_2}+\rho 
\, | \, \alpha_0) \, \in \, \nnn
\quad \Longleftrightarrow \quad \left\{
\begin{array}{l}
\heartsuit \,\ = \,\ {\rm I \,\ or \,\ III} \\[1mm]
2k_1+k_2+1 \, = \, M \\[1mm]
m_2 \, = \, m 
\end{array}\right. $
\end{enumerate}
\end{lemma}

\begin{proof} 1) By \eqref{n4:eqn:2022-1211b}, one has
$$
\Lambda^{(M)[K(m),m_2](\heartsuit)}_{k_1,k_2}+\rho \,\ \equiv \,\ 
\, - \, \frac{m}{M} \, \Lambda_0
\, - \, \frac{m}{M} \, \beta
\, - \, \frac{m_2+1}{2}\overline{y}(\alpha_1+\alpha_2)
\quad {\rm mod} \,\ \ccc \, \delta
$$
so
{\allowdisplaybreaks
\begin{eqnarray*}
& & \hspace{-10mm}
\big(\Lambda^{(M)[K(m),m_2](\heartsuit)}_{k_1,k_2} +\rho \, 
\big| \, \delta-\theta\big)
\,\ = \,\ 
- \, \frac{m}{M} \, \underbrace{(\Lambda_0 |\delta)}_{1}
\, + \, \frac{m}{M} \, (\beta |\theta)
\, + \, \frac{m_2+1}{2} \, 
(\overline{y}(\alpha_1+\alpha_2) \, | \, \theta)
\\[1mm]
&=&
- \, \frac{m}{M} 
\, + \, \frac{m}{M} \, (\beta |\theta)
\, + \, \frac{m_2+1}{2} \, 
(\overline{y}(\alpha_1+\alpha_2) \, | \, \theta)
\end{eqnarray*}}
Then by Note \ref{note:2022-1128c}, this is computed in each case
as follows.

\medskip

\noindent
If \,\ $\heartsuit \, = $ I \, or \, III,
{\allowdisplaybreaks
\begin{eqnarray*}
& & \hspace{-15mm}
\big(\Lambda^{(M)[K(m),m_2](\heartsuit)}_{k_1,k_2}+\rho \, 
\big| \, \delta-\theta\big)
\,\ = \,\ 
- \, \frac{m}{M} 
\, + \, \frac{1}{M} \hspace{-13mm} 
\underbrace{(\beta |\theta)}_{\substack{|| \\[0mm] 
{\displaystyle \hspace{14mm}
- m(2k_1+k_2)
}}} \hspace{-8mm}
+ \,\ \frac{m_2+1}{2} \, 
\underbrace{(\overline{y}(\alpha_1+\alpha_3) \, | \, \theta)
}_{\substack{|| \\[0mm] {\displaystyle 2
}}}
\\[-5mm]
&=&
- \, \frac{m(2k_1+k_2+1)}{M}+ m_2+1
\end{eqnarray*}}

\vspace{-2mm}

\noindent
If \,\ $\heartsuit \, = $ II \, or \, IV,
{\allowdisplaybreaks
\begin{eqnarray*}
& & \hspace{-10mm}
\big(\Lambda^{(M)[K(m),m_2](\heartsuit)}_{k_1,k_2}+\rho \, 
\big| \, \delta-\theta\big)
\,\ = \,\ 
- \, \frac{m}{M} 
\, + \, \frac{1}{M} \hspace{-4mm} 
\underbrace{(\beta |\theta)}_{\substack{|| \\[0mm] 
{\displaystyle \hspace{5mm}
2k_1+k_2
}}} \hspace{-2mm}
+ \,\ \frac{m_2+1}{2} \, 
\underbrace{(\overline{y}(\alpha_1+\alpha_3) \, | \, \theta)
}_{\substack{|| \\[0mm] {\displaystyle - \, 2
}}}
\\[-5mm]
&=&
\frac{m(2k_1+k_2-1)}{M}-m_2-1
\end{eqnarray*}}
proving 1). \,\ 2) follows from 1) immediately.
\end{proof}

\medskip

Then by Lemmas \ref{lemma:2022-1220a} and \ref{lemma:2022-1221a},
we obtain the following:

\medskip

%(line=3290) %%
%(label=n4:prop:2022-1231a) %%
\begin{prop} 
\label{n4:prop:2022-1231a}
Let \, $\heartsuit \, = $ {\rm I}  or {\rm III} and $2k_1+k_2=M-1$. Then
\begin{enumerate}
\item[{\rm 1)}] \,\ $H(\Lambda^{(M)[K(m),m](\heartsuit)}_{k_1,k_2})
\,\ = \,\ 
H^{\rm tw}(\Lambda^{(M)[K(m),m](\heartsuit)}_{k_1,k_2})
\,\ = \,\ \{0\} $
\item[{\rm 2)}]
\begin{enumerate}
\item[{\rm (i)}] \,\
$\ddot{H}(\Lambda^{(M)[K(m),m-1](\heartsuit)}_{k_1,k_2})
= H(\Lambda^{(M)[K(m),m-1](\heartsuit)}_{k_1,k_2})$ 
\item[{\rm (ii)}] \,\ 
$\ddot{H}^{\rm tw}(\Lambda^{(M)[K(m),m-1](\heartsuit)}_{k_1,k_2})
= H^{\rm tw}(\Lambda^{(M)[K(m),m-1](\heartsuit)}_{k_1,k_2})$ 
\end{enumerate}
\end{enumerate}
\end{prop}

\medskip

Then by Proposition \ref{n4:prop:2022-1231a} and
Lemmas \ref{n4:lemma:2023-102a} and \ref{n4:lemma:2023-102b},
we obtain the following:

%(line=3317) %%
%(label=n4:prop:2023-103a) %%
\begin{prop} 
\label{n4:prop:2023-103a}
In the case $\heartsuit =$ {\rm I} or {\rm III} and $2k_1+k_2=M-1$,
the twisted and non-twisted numerators of irreducible N=4 modules 
$H(\Lambda^{(M)[K(m),m-1] (\heartsuit)}_{k_1,k_2})$ are given 
by the following formulas:
\begin{enumerate}
\item[{\rm 1)}]
\begin{enumerate}
\item[{\rm (i)}] $\big[\overset{N=4}{R}{}^{(+)} \cdot 
{\rm ch}^{(+)}_{H(\Lambda^{(M)[K(m), m-1]{\rm (I)}}_{k_1,k_2})}
\big](\tau,z)
\,\ = \,\ 
(-1)^m \, \Psi^{[M,m,m; \frac12]}_{k_1+\frac12, \, -(k_1+k_2+\frac12); \, \frac12}
(\tau, z, z,0)$
% 1) (iii)
\item[{\rm (ii)}] $\big[\overset{N=4}{R}{}^{(+)} \cdot 
{\rm ch}^{(+)}_{H(\Lambda^{(M)[K(m), m-1]{\rm (III)}}_{k_1,k_2})}\big](\tau,z)
\,\ = \,\ 
(-1)^m \, 
\Psi^{[M,m,m; \frac12]}_{k_1+\frac12, \, -(k_1+k_2+\frac12); \, \frac12}
(\tau, -z, -z,0)$
\end{enumerate}
% 2)
\item[{\rm 2)}]
\begin{enumerate}
\item[{\rm (i)}] $\big[\overset{N=4}{R}{}^{(-)} \cdot 
{\rm ch}^{(-)}_{H(\Lambda^{(M)[K(m), m-1]{\rm (I)}}_{k_1,k_2})}\big](\tau,z)
\,\ = \,\ 
\Psi^{[M,1,0; 0]}_{k_1+\frac12, \, -(k_1+k_2+\frac12); \, \frac12}
(\tau, z, z,0)$
% 2) (iii)
\item[{\rm (ii)}] $\big[\overset{N=4}{R}{}^{(-)} \cdot 
{\rm ch}^{(-)}_{H(\Lambda^{(M)[K(m), m-1]{\rm (III)}}_{k_1,k_2})}\big](\tau,z)
\,\ = \,\ 
\Psi^{[M,m,m; 0]}_{k_1+\frac12, \, -(k_1+k_2+\frac12); \, \frac12}
(\tau, -z, -z,0)$
\end{enumerate}
\item[$1)^{\rm tw}$]
\begin{enumerate}
\item[{\rm (i)}] 
$\big[\overset{N=4}{R}{}^{(+){\rm tw}} \cdot 
{\rm ch}^{(+){\rm tw}}_{H(\Lambda^{(M)[K(m), m-1]{\rm (I)}}_{k_1,k_2})}
\big](\tau,z)
\,\ = \,\ 
(-1)^m \, 
\Psi^{[M,m,m; \frac12]}_{k_1, \, -(k_1+k_2+1); \, 0}(\tau, -z, -z,0)$
% 1) (iii)
\item[{\rm (ii)}] $\big[\overset{N=4}{R}{}^{(+){\rm tw}} \cdot 
{\rm ch}^{(+){\rm tw}}_{H(\Lambda^{(M)[K(m), m-1]{\rm (III)}}_{k_1,k_2})}\big](\tau,z)
\,\ = \,\ 
(-1)^m \, \Psi^{[M,m,m; \frac12]}_{k_1+1, \, -(k_1+k_2); \, 0}(\tau, z, z,0)$
\end{enumerate}
% 2)
\item[$2)^{\rm tw}$]
\begin{enumerate}
\item[{\rm (i)}] $\big[\overset{N=4}{R}{}^{(-){\rm tw}} \cdot 
{\rm ch}^{(-){\rm tw}}_{H(\Lambda^{(M)[K(m), m-1]{\rm (I)}}_{k_1,k_2})}\big](\tau,z)
\,\ = \,\ 
\Psi^{[M,m,m; 0]}_{k_1, \, -(k_1+k_2+1); \, 0}(\tau, -z, -z,0)$
% 2) (iii)
\item[{\rm (ii)}] $\big[\overset{N=4}{R}{}^{(-){\rm tw}} \cdot 
{\rm ch}^{(-){\rm tw}}_{H(\Lambda^{(M)[K(m), m-1]{\rm (III)}}_{k_1,k_2})}\big](\tau,z)
\,\ = \,\ 
\Psi^{[M,m,m; 0]}_{k_1+1, \, -(k_1+k_2); \, 0}(\tau, z, z,0)$
\end{enumerate}
\end{enumerate}
\end{prop}

\section{Nice cases of quantum Hamiltonian reduction}
\label{sec:nice}
%(label=sec:nice) %%
%(line=3395) %%

In this section we consider the case $m=1$ and $m_2=0$ and $2k_1+k_2=M-1$.
The range of the parameters $(k_1,k_2)$, when $2k_1+k_2=M-1$, is 
%(label=n4:eqn:2022-1213a) %%
\begin{equation}
\begin{array}{lclcl}
0 \, \leq \, k_1 \, \leq \frac12(M-1) & {\rm and} & k_2 \, \geq \, 0 & \quad &
{\rm for} \quad \Pi^{(M) {\rm (I)}}_{k_1,k_2}
\\[3mm]
0 \, \leq \, k_1 \, \leq \frac12(M-2) & {\rm and} & k_2 \, \geq \, 1 & \quad &
{\rm for} \quad \Pi^{(M) {\rm (III)}}_{k_1,k_2}
\end{array}
\label{n4:eqn:2022-1213a}
\end{equation}

In this case, the formulas in Propositions \ref{n4:prop:2023-103a} 
give the following:

\medskip

%(line=3448) %%
%(label=lemma:2022-1227a) %%
\begin{lemma} 
\label{lemma:2022-1227a}
In the case $2k_1+k_2=M-1$, the following formulas hold:
\begin{enumerate}
\item[{\rm 1)}]
\begin{enumerate}
\item[{\rm (i)}] $\big[\overset{N=4}{R}{}^{(+)} \cdot 
{\rm ch}^{(+)}_{H(\Lambda^{(M)[K(1), 0]{\rm (I)}}_{k_1,k_2})}
\big](\tau,z)
\,\ = \,\ 
\Psi^{[M,1,0; \frac12]}_{k_1+\frac12, \, k_1+\frac12; \, \frac12}
(\tau, z, z,0)$
% 1) (iii)
\item[{\rm (ii)}] $\big[\overset{N=4}{R}{}^{(+)} \cdot 
{\rm ch}^{(+)}_{H(\Lambda^{(M)[K(1), 0]{\rm (III)}}_{k_1,k_2})}\big](\tau,z)
\,\ = \,\ 
- \, \Psi^{[M,1,0; \frac12]}_{-(k_1+\frac12), \, -(k_1+\frac12); \, \frac12}
(\tau, z, z,0)$
\end{enumerate}
% 2)
\item[{\rm 2)}]
\begin{enumerate}
\item[{\rm (i)}] $\big[\overset{N=4}{R}{}^{(-)} \cdot 
{\rm ch}^{(-)}_{H(\Lambda^{(M)[K(1), 0]{\rm (I)}}_{k_1,k_2})}\big](\tau,z)
\,\ = \,\ 
- \, \Psi^{[M,1,0; 0]}_{k_1+\frac12, \, k_1+\frac12; \, \frac12}
(\tau, z, z,0)$
% 2) (iii)
\item[{\rm (ii)}] $\big[\overset{N=4}{R}{}^{(-)} \cdot 
{\rm ch}^{(-)}_{H(\Lambda^{(M)[K(1), 0]{\rm (III)}}_{k_1,k_2})}\big](\tau,z)
\,\ = \,\ 
\Psi^{[M,1,0; 0]}_{-(k_1+\frac12), \, -(k_1+\frac12); \, \frac12}
(\tau, z, z,0)$
\end{enumerate}
\end{enumerate}
% twisted
\begin{enumerate}
\item[$1)^{\rm tw}$]
\begin{enumerate}
\item[{\rm (i)}] $\big[\overset{N=4}{R}{}^{(+){\rm tw}} \cdot 
{\rm ch}^{(+){\rm tw}}_{H(\Lambda^{(M)[K(1), 0]{\rm (I)}}_{k_1,k_2})}
\big](\tau,z)
\,\ = \,\ 
- \, \Psi^{[M,1,0; \frac12]}_{-k_1, \, -k_1; \, 0}(\tau, z, z,0)$
% 1) (iii)
\item[{\rm (ii)}] $\big[\overset{N=4}{R}{}^{(+){\rm tw}} \cdot 
{\rm ch}^{(+){\rm tw}}_{H(\Lambda^{(M)[K(1), 0]{\rm (III)}}_{k_1,k_2})}\big](\tau,z)
\,\ = \,\ 
\Psi^{[M,1,0; \frac12]}_{k_1+1, \, k_1+1; \, 0}(\tau, z, z,0)$
\end{enumerate}
% 2)
\item[$2)^{\rm tw}$]
\begin{enumerate}
\item[{\rm (i)}] $\big[\overset{N=4}{R}{}^{(-){\rm tw}} \cdot 
{\rm ch}^{(-){\rm tw}}_{H(\Lambda^{(M)[K(1), 0]{\rm (I)}}_{k_1,k_2})}\big](\tau,z)
\,\ = \,\ 
- \, \Psi^{[M,1,0; 0]}_{-k_1, \, -k_1; \, 0}(\tau, z, z,0)$
% 2) (iii)
\item[{\rm (ii)}] $\big[\overset{N=4}{R}{}^{(-){\rm tw}} \cdot 
{\rm ch}^{(-){\rm tw}}_{H(\Lambda^{(M)[K(1), 0]{\rm (III)}}_{k_1,k_2})}\big](\tau,z)
\,\ = \,\ 
\Psi^{[M,1,0; 0]}_{k_1+1, \, k_1+1; \, 0}(\tau, z, z,0)$
\end{enumerate}
\end{enumerate}
\end{lemma}

\begin{proof} We note that 
$$\left\{
\begin{array}{ccc}
k_1+k_2+\frac12-M &=& -(k_1+\frac12) \\[1.5mm]
M-(k_1+k_2+\frac12) &=& k_1+\frac12
\end{array}\right. \hspace{5mm} \left\{
\begin{array}{ccc}
k_1+k_2+1-M &=& -k_1 \\[1mm]
M-(k_1+k_2) &=& k_1+1
\end{array}\right. ,
$$
and $\Psi^{[M,1,s; \varepsilon]}_{j,k; \varepsilon'}
=
\Psi^{[M,1,0;\varepsilon]}_{j,k; \varepsilon'}$ \, for 
${}^{\forall} s \in \zzz$ \, by \eqref{n4:eqn:2022-1208a}.
Then the formulas in this proposition follow from 
Proposition \ref{n4:prop:2023-103a} and Lemma \ref{lemma:2022-1021a}.
\end{proof}

\medskip

In order to write down the characters of these \lq \lq nice" N=4 modules 
explicitly, we use the following:

\medskip

%(line=3510) %%
%(label=n4:lemma:2022-1219a) %%
%(label=note:2022-1217a) %%
\begin{lemma} 
\label{n4:lemma:2022-1219a}
\label{note:2022-1217a}
For $M \, \in \, \nnn$, the following formulas hold:
\begin{enumerate}
\item[{\rm 1)}]
\begin{enumerate}
\item[{\rm (i)}] \quad $\dfrac{
\Psi^{[M,1;0; \frac12]}_{j,j;\frac12}(\tau, z, z, 0)
}{\overset{N=4}{R}{}^{(+)}(\tau,z)}
\,\ = \,\ 
- \, q^{\frac{1}{M}j^2} e^{\frac{4\pi ij}{M}z}$
$$
\times \,\ \frac{
\vartheta_{00}(M\tau, \, z+j\tau)
\vartheta_{01}(M\tau, \, z+j\tau)
\vartheta_{11}(M\tau, \, z+j\tau)
}{\vartheta_{10}(M\tau, \, z+j\tau)}
\cdot \frac{\vartheta_{00}(\tau,z)}{
\vartheta_{01}(\tau,z)\vartheta_{10}(\tau,z)\vartheta_{11}(\tau,z)}
$$
\item[{\rm (ii)}] \quad $\dfrac{
\Psi^{[M,1;0; 0]}_{j,j;\frac12}(\tau, z, z, 0)
}{\overset{N=4}{R}{}^{(-)}(\tau,z)}
\,\ = \,\ 
- \, q^{\frac{1}{M}j^2} e^{\frac{4\pi ij}{M}z}$
$$
\times \,\ 
\dfrac{
\vartheta_{00}(M\tau, \, z+j\tau)
\vartheta_{01}(M\tau, \, z+j\tau)
\vartheta_{10}(M\tau, \, z+j\tau)
}{\vartheta_{11}(M\tau, \, z+j\tau)}
\cdot \frac{\vartheta_{01}(\tau,z)}{
\vartheta_{00}(\tau,z)\vartheta_{10}(\tau,z)\vartheta_{11}(\tau,z)}
$$
\end{enumerate}
\item[{\rm 2)}]
\begin{enumerate}
\item[{\rm (i)}] \quad $\dfrac{
\Psi^{[M,1;0; \frac12]}_{j,j;0}(\tau, z, z, 0)
}{\overset{N=4}{R}{}^{(+){\rm tw}}(\tau,z)}
\,\ = \,\ 
q^{\frac{1}{M}j^2} e^{\frac{4\pi ij}{M}z}$
$$
\times \,\ \frac{
\vartheta_{00}(M\tau, \, z+j\tau)
\vartheta_{01}(M\tau, \, z+j\tau)
\vartheta_{11}(M\tau, \, z+j\tau)
}{\vartheta_{10}(M\tau, \, z+j\tau)}
\cdot \frac{\vartheta_{10}(\tau,z)}{
\vartheta_{00}(\tau,z)\vartheta_{01}(\tau,z)\vartheta_{11}(\tau,z)}
$$
\item[{\rm (ii)}] \quad $\dfrac{
\Psi^{[M,1;0; 0]}_{j,j;0}(\tau, z, z, 0)
}{\overset{N=4}{R}{}^{(-){\rm tw}}(\tau,z)}
\,\ = \,\ - \, 
q^{\frac{1}{M}j^2} e^{\frac{4\pi ij}{M}z}$
$$
\times \,\ 
\dfrac{
\vartheta_{00}(M\tau, \, z+j\tau)
\vartheta_{01}(M\tau, \, z+j\tau)
\vartheta_{10}(M\tau, \, z+j\tau)
}{\vartheta_{11}(M\tau, \, z+j\tau)}
\cdot \frac{\vartheta_{11}(\tau,z)}{
\vartheta_{00}(\tau,z)\vartheta_{01}(\tau,z)\vartheta_{10}(\tau,z)}
$$
\end{enumerate}
\end{enumerate}
\end{lemma}

\begin{proof} Letting $j=k$ and $z_1=z_2=z$ in \eqref{n4:eqn:2022-1208a},
one has 
\begin{subequations}
{\allowdisplaybreaks
\begin{eqnarray}
& & \hspace{-15mm}
\Psi^{[M,1,0;\frac12]}_{j,j;\varepsilon'}
\,\ = \,\
-i \, q^{\frac{1}{M}j^2} e^{\frac{4\pi i}{M}jz} \, 
\frac{\eta(M\tau)^3 \vartheta_{11}(M\tau, 2z+2j\tau)
}{\vartheta_{11}(M\tau, z+j\tau+\frac12) \, 
\vartheta_{11}(M\tau, z+j\tau-\frac12)}
\nonumber
\\[3mm]
&=&
i \, q^{\frac{1}{M}j^2} e^{\frac{4\pi i}{M}jz} \, 
\frac{\vartheta_{00}(M\tau, z+j\tau) \, \vartheta_{01}(M\tau, z+j\tau) \, 
\vartheta_{11}(M\tau, z+j\tau)}{\vartheta_{10}(M\tau, z+j\tau)}
\label{n4:eqn:2022-1219a2}
\\[3mm]
& & \hspace{-15mm}
\Psi^{[M,1,0;0]}_{j,j;\varepsilon'}
\,\ = \,\
-i \, q^{\frac{1}{M}j^2} e^{\frac{4\pi i}{M}jz} \, 
\frac{\eta(M\tau)^3 \vartheta_{11}(M\tau, 2z+2j\tau)
}{\vartheta_{11}(M\tau, z+j\tau)^2}
\nonumber
\\[3mm]
&=&
-i \, q^{\frac{1}{M}j^2} e^{\frac{4\pi i}{M}jz} \, 
\frac{\vartheta_{00}(M\tau, z+j\tau) \, \vartheta_{01}(M\tau, z+j\tau) \, 
\vartheta_{10}(M\tau, z+j\tau)}{\vartheta_{11}(M\tau, z+j\tau)}
\label{n4:eqn:2022-1219a1}
\end{eqnarray}}
\end{subequations}
since
$$\left\{
\begin{array}{lcl}
\eta(\tau)^3 \vartheta_{11}(\tau, 2z) 
&=&
\vartheta_{00}(\tau, z) \, \vartheta_{01}(\tau, z) \, 
\vartheta_{10}(\tau, z) \, \vartheta_{11}(\tau, z)
\\[2mm]
\vartheta_{11}(\tau, z \pm \frac12)
&=&
\mp \, \vartheta_{10}(\tau, z)
\end{array}\right.
$$
Then the formulas in Lemma \ref{n4:lemma:2022-1219a} follow 
immediately from \eqref{n4:eqn:2022-1219c1} $\sim$ 
\eqref{n4:eqn:2022-1219c4} and the above formulas 
\eqref{n4:eqn:2022-1219a2} and \eqref{n4:eqn:2022-1219a1}.
\end{proof}

\medskip

%(line=3678) %%
%(label=thm:2022-1217a) %%
\begin{thm} \,\ 
\label{thm:2022-1217a}
For $M \in \nnn$ and non-negative integers $k_1$ and $k_2$ 
satisfying $2k_1+k_2=M-1$ and \eqref{n4:eqn:2022-1213a}, 
the characters of the N=4 module 
$H(\Lambda^{(M)[K(1),0](\heartsuit)}_{k_1,k_2})$ 
$(\heartsuit \, =$ {\rm I} or {\rm III)} 
are given by the following formulas:
{\allowdisplaybreaks 
\begin{eqnarray*}
& & \hspace{-7mm}
{\rm ch}^{(+)}_{H(\Lambda^{(M)[K(1),0]{\rm (I)}}_{k_1,k_2})}(\tau,z)
\,\ = \,\ 
\Big[- \, q^{\frac{1}{M}j^2} e^{\frac{4\pi i}{M}jz}
\\[2mm]
& &
\times \,\ \frac{
\vartheta_{00}(M\tau, \, z+j\tau)
\vartheta_{01}(M\tau, \, z+j\tau)
\vartheta_{11}(M\tau, \, z+j\tau)
}{\vartheta_{10}(M\tau, \, z+j\tau)} \cdot
\frac{\vartheta_{00}(\tau,z)}{
\vartheta_{01}(\tau,z)\vartheta_{10}(\tau,z)\vartheta_{11}(\tau,z)}
\bigg]_{j=k_1+\frac12}
% ch^{(-)}
\\[3mm]
& & \hspace{-7mm}
{\rm ch}^{(-)}_{H(\Lambda^{(M)[K(1),0]{\rm (I)}}_{k_1,k_2})}(\tau,z)
\,\ = \,\ \Big[
q^{\frac{1}{M}j^2} e^{\frac{4\pi i}{M}jz}
\\[2mm]
& &
\times \,\ \dfrac{
\vartheta_{00}(M\tau, \, z+j\tau)
\vartheta_{01}(M\tau, \, z+j\tau)
\vartheta_{10}(M\tau, \, z+j\tau)
}{\vartheta_{11}(M\tau, \, z+j\tau)} \cdot
\frac{\vartheta_{01}(\tau,z)}{
\vartheta_{00}(\tau,z)\vartheta_{10}(\tau,z)\vartheta_{11}(\tau,z)}
\bigg]_{j=k_1+\frac12}
% (III)
\\[3mm]
& & \hspace{-7mm}
{\rm ch}^{(+)}_{H(\Lambda^{(M)[K(1),0]{\rm (III)}}_{k_1,k_2})}(\tau,z)
\,\ = \,\ \Big[
q^{\frac{1}{M}j^2} e^{\frac{4\pi i}{M}jz} 
\\[2mm]
& & \hspace{-5mm}
\times \,\ \frac{
\vartheta_{00}(M\tau, \, z+j\tau)
\vartheta_{01}(M\tau, \, z+j\tau)
\vartheta_{11}(M\tau, \, z+j\tau)
}{\vartheta_{10}(M\tau, \, z+j\tau)} \cdot
\frac{\vartheta_{00}(\tau,z)}{
\vartheta_{01}(\tau,z)\vartheta_{10}(\tau,z)\vartheta_{11}(\tau,z)}
\bigg]_{j=-(k_1+\frac12)}
% ch^{(-)}
\\[3mm]
& & \hspace{-7mm}
{\rm ch}^{(-)}_{H(\Lambda^{(M)[K(1),0]{\rm (III)}}_{k_1,k_2})}(\tau,z)
\,\ = \,\ \Big[
- \, q^{\frac{1}{M}j^2} e^{\frac{4\pi i}{M}jz}
\\[2mm]
& & \hspace{-5mm}
\times \,\ \dfrac{
\vartheta_{00}(M\tau, \, z+j\tau)
\vartheta_{01}(M\tau, \, z+j\tau)
\vartheta_{10}(M\tau, \, z+j\tau)
}{\vartheta_{11}(M\tau, \, z+j\tau)} \cdot
\frac{\vartheta_{01}(\tau,z)}{
\vartheta_{00}(\tau,z)\vartheta_{10}(\tau,z)\vartheta_{11}(\tau,z)}
\bigg]_{j=-(k_1+\frac12)}
\end{eqnarray*}}
\end{thm}

\begin{proof} These formulas are obtained easily from 
\eqref{n4:eqn:2022-1219c1} and \eqref{n4:eqn:2022-1219c2}
and Lemmas \ref{lemma:2022-1227a} and \ref{n4:lemma:2022-1219a}. 
In the case $\heartsuit =$ I, the proof goes as follows:
{\allowdisplaybreaks
\begin{eqnarray*}
& & \hspace{-15mm}
{\rm ch}^{(+)}_{H(\Lambda^{(M)[K(1),0]{\rm (I)}}_{k_1,k_2})}(\tau,z)
\,\ = \,\ \frac{
\Psi^{[M,1,0;\frac12]}_{k_1+\frac12, k_1+\frac12; \frac12}(\tau,z)
}{\overset{N=4}{R}{}^{(+)}(\tau,z)}
\,\ = \,\ \frac{
\Psi^{[M,1,0;\frac12]}_{j,j; \frac12}(\tau,z)
}{\overset{N=4}{R}{}^{(+)}(\tau,z)}
\\[2mm]
&=&
- \, q^{\frac{1}{M}j^2} e^{\frac{4\pi i}{M}jz} \, 
\frac{
\vartheta_{00}(M\tau, \, z+j\tau)
\vartheta_{01}(M\tau, \, z+j\tau)
\vartheta_{11}(M\tau, \, z+j\tau)
}{\vartheta_{10}(M\tau, \, z+j\tau)} 
\\[0mm]
& &
\times \,\
\frac{\vartheta_{00}(\tau,z)}{
\vartheta_{01}(\tau,z)\vartheta_{10}(\tau,z)\vartheta_{11}(\tau,z)}
\\[2mm]
& & \hspace{-15mm}
{\rm ch}^{(-)}_{H(\Lambda^{(M)[K(1),0]{\rm (I)}}_{k_1,k_2})}(\tau,z)
\,\ = \,\ 
- \, \frac{
\Psi^{[M,1,0;0]}_{k_1+\frac12, k_1+\frac12; \frac12}(\tau,z)
}{\overset{N=4}{R}{}^{(-)}(\tau,z)}
\,\ = \,\ 
- \, \frac{
\Psi^{[M,1,0;0]}_{j,j; \frac12}(\tau,z)
}{\overset{N=4}{R}{}^{(-)}(\tau,z)}
\\[2mm]
&=&
q^{\frac{1}{M}j^2} e^{\frac{4\pi i}{M}jz} \, 
\frac{
\vartheta_{00}(M\tau, \, z+j\tau)
\vartheta_{01}(M\tau, \, z+j\tau)
\vartheta_{10}(M\tau, \, z+j\tau)
}{\vartheta_{11}(M\tau, \, z+j\tau)} 
\\[0mm]
& &
\times \,\
\frac{\vartheta_{01}(\tau,z)}{
\vartheta_{00}(\tau,z)\vartheta_{10}(\tau,z)\vartheta_{11}(\tau,z)}
\end{eqnarray*}}
proving the formula in the case $\heartsuit =$ I.
The proof of the rests is quite similar.
\end{proof}

\medskip

%(line=3814) %%
%(label=thm:2022-1217b) %%
\begin{thm} \,\
\label{thm:2022-1217b}
For $M \in \nnn$ and non-negative integers $k_1$ and $k_2$ 
satisfying $2k_1+k_2=M-1$ and \eqref{n4:eqn:2022-1213a}, 
the twisted characters of the N=4 module 
$H^{\rm tw}(\Lambda^{(M)[K(1),0](\heartsuit)}_{k_1,k_2})$ 
$(\heartsuit \, =$ {\rm I} or {\rm III)} 
are given by the following formulas:
{\allowdisplaybreaks 
\begin{eqnarray*}
& & \hspace{-5mm}
{\rm ch}^{(+){\rm tw}}_{H(\Lambda^{(M)[K(1),0]{\rm (I)}}_{k_1,k_2})}(\tau,z)
\,\ = \,\ 
\Big[ \, q^{\frac{1}{M}j^2} e^{\frac{4\pi i}{M}jz}
\\[2mm]
& &
\times \,\ \frac{
\vartheta_{00}(M\tau, \, z+j\tau)
\vartheta_{01}(M\tau, \, z+j\tau)
\vartheta_{11}(M\tau, \, z+j\tau)
}{\vartheta_{10}(M\tau, \, z+j\tau)}
\cdot \frac{\vartheta_{10}(\tau,z)}{
\vartheta_{00}(\tau,z)\vartheta_{01}(\tau,z)\vartheta_{11}(\tau,z)}
\bigg]_{j=-k_1}
\\[3mm]
& & \hspace{-5mm}
{\rm ch}^{(-){\rm tw}}_{H(\Lambda^{(M)[K(1),0]{\rm (I)}}_{k_1,k_2})}(\tau,z)
\,\ = \,\ \Big[ \, 
q^{\frac{1}{M}j^2} e^{\frac{4\pi i}{M}jz}
\\[2mm]
& &
\times \,\ \frac{
\vartheta_{00}(M\tau, \, z+j\tau)
\vartheta_{01}(M\tau, \, z+j\tau)
\vartheta_{10}(M\tau, \, z+j\tau)
}{\vartheta_{11}(M\tau, \, z+j\tau)}
\cdot \frac{\vartheta_{11}(\tau,z)}{
\vartheta_{00}(\tau,z)\vartheta_{01}(\tau,z)\vartheta_{10}(\tau,z)}
\bigg]_{j=-k_1}
% (III)
\\[3mm]
& & \hspace{-5mm}
{\rm ch}^{(+){\rm tw}}_{H(\Lambda^{(M)[K(1),0]{\rm (III)}}_{k_1,k_2})}(\tau,z)
\,\ = \,\ \Big[ \, - \, 
q^{\frac{1}{M}j^2} e^{\frac{4\pi i}{M}jz}
\\[2mm]
& &
\times \,\ \frac{
\vartheta_{00}(M\tau, \, z+j\tau)
\vartheta_{01}(M\tau, \, z+j\tau)
\vartheta_{11}(M\tau, \, z+j\tau)
}{\vartheta_{10}(M\tau, \, z+j\tau)}
\cdot \frac{\vartheta_{10}(\tau,z)}{
\vartheta_{00}(\tau,z)\vartheta_{01}(\tau,z)\vartheta_{11}(\tau,z)}
\bigg]_{j=k_1+1}
\\[3mm]
& & \hspace{-5mm}
{\rm ch}^{(-){\rm tw}}_{H(\Lambda^{(M)[K(1),0]{\rm (III)}}_{k_1,k_2})}(\tau,z)
\,\ = \,\ \Big[ \, - \, 
q^{\frac{1}{M}j^2} e^{\frac{4\pi i}{M}jz}
\\[2mm]
& &
\times \,\ \frac{
\vartheta_{00}(M\tau, \, z+j\tau)
\vartheta_{01}(M\tau, \, z+j\tau)
\vartheta_{10}(M\tau, \, z+j\tau)
}{\vartheta_{11}(M\tau, \, z+j\tau)}
\cdot \frac{\vartheta_{11}(\tau,z)}{
\vartheta_{00}(\tau,z)\vartheta_{01}(\tau,z)\vartheta_{10}(\tau,z)}
\bigg]_{j=k_1+1}
\end{eqnarray*}}
\end{thm}

\begin{proof} These formulas are obtained easily from 
\eqref{n4:eqn:2022-1219c3} and \eqref{n4:eqn:2022-1219c4}
and Lemmas \ref{lemma:2022-1227a} and \ref{n4:lemma:2022-1219a}. 
In the case $\heartsuit =$ I, the proof goes as follows:
{\allowdisplaybreaks 
\begin{eqnarray*}
& & \hspace{-10mm}
{\rm ch}^{(+){\rm tw}}_{H(\Lambda^{(M)[K(1),0]{\rm (I)}}_{k_1,k_2})}(\tau,z)
\, = \,\ - \, 
\frac{\Psi^{[M,1,0,\frac12]}_{-k_1,-k_1: 0}(\tau, \, z, z, \, 0)}{
\overset{N=4}{R}{}^{(+){\rm tw}}(\tau,z)}
\,\ = \,\ - \, 
\frac{\Psi^{[M,1,0,\frac12]}_{j,j: 0}(\tau, \, z, z, \, 0)}{
\overset{N=4}{R}{}^{(+){\rm tw}}(\tau,z)}
\\[3mm]
&=& 
q^{\frac{1}{M}j^2} e^{\frac{4\pi i}{M}jz} \, 
\frac{
\vartheta_{00}(M\tau, \, z+j\tau)
\vartheta_{01}(M\tau, \, z+j\tau)
\vartheta_{11}(M\tau, \, z+j\tau)
}{\vartheta_{10}(M\tau, \, z+j\tau)}
\\[0mm]
& &
\times \,\ 
\frac{\vartheta_{10}(\tau,z)}{
\vartheta_{00}(\tau,z)\vartheta_{01}(\tau,z)\vartheta_{11}(\tau,z)}
% (I) (-)
\\[2mm]
& & \hspace{-10mm}
{\rm ch}^{(-){\rm tw}}_{H(\Lambda^{(M)[K(1),0]{\rm (I)}}_{k_1,k_2})}(\tau,z)
\, = \,\ - \, 
\frac{\Psi^{[M,1,0,0]}_{-k_1,-k_1: 0}(\tau, \, z, z, \, 0)}{
\overset{N=4}{R}{}^{(-){\rm tw}}(\tau,z)}
\,\ = \,\ - \, 
\frac{\Psi^{[M,1,0,0]}_{j,j: 0}(\tau, \, z, z, \, 0)}{
\overset{N=4}{R}{}^{(-){\rm tw}}(\tau,z)}
\\[3mm]
&=& 
q^{\frac{1}{M}j^2} e^{\frac{4\pi i}{M}jz} \, 
\frac{
\vartheta_{00}(M\tau, \, z+j\tau)
\vartheta_{01}(M\tau, \, z+j\tau)
\vartheta_{10}(M\tau, \, z+j\tau)
}{\vartheta_{11}(M\tau, \, z+j\tau)}
\\[0mm]
& &
\times \,\ 
\frac{\vartheta_{11}(\tau,z)}{
\vartheta_{00}(\tau,z)\vartheta_{01}(\tau,z)\vartheta_{10}(\tau,z)}
\end{eqnarray*}}
proving the formulas in the case $\heartsuit =$ I.
The proof of the rests is quite similar.
\end{proof}

\medskip

Now Corollary \ref{cor:2022-1222a} and Theorems \ref{thm:2022-1217a}
and \ref{thm:2022-1217b} complete the proof of 
Theorem \ref{n4:thm:2022-1203a} stated in section \ref{sec:introduction}.

\section{Examples $\sim$ the cases $M=1$ and $M=2$}
\label{sec:examples}

In this section, we compute the formulas in Theorem 
\ref{n4:thm:2022-1203a} in the cases $M=1$ and $M=2$.

\medskip

%(line=3931) %%
%(label=ex:2022-1216a) %%
\begin{ex} 
\label{ex:2022-1216a}
In the case $M=1$, one has $I^{[1]}=\{\frac12\}$ and 
$I^{[1], R}=\{0\}$ by \eqref{n4:eqn:2022-1203a}
and $c^{[1]}=0$ and $
h^{[1,\frac12]}= s^{[1,\frac12]}=h^{[1,0]R}=s^{[1,0]R}=0$
by \eqref{n4:eqn:2023-106a} and \eqref{n4:eqn:2023-106b}.

\medskip

\noindent
{\rm 1)} \,\ For $j=\frac12 \in I^{[1]}$, Theorem \ref{n4:thm:2022-1203a} gives
{\allowdisplaybreaks 
\begin{eqnarray*}
& & \hspace{-10mm}
{\rm ch}^{(+)}_{V^{[M=1,j=\frac12]}}(\tau,z)
\,\ = \,\ 
- \, q^{(\frac12)^2} e^{2\pi iz}
\\[2mm]
& &
\times \,\ \frac{
\vartheta_{00}(\tau, \, z+\frac12\tau)
\vartheta_{01}(\tau, \, z+\frac12\tau)
\vartheta_{11}(\tau, \, z+\frac12\tau)
}{\vartheta_{10}(\tau, \, z+\frac12\tau)}
\cdot \frac{\vartheta_{00}(\tau,z)}{
\vartheta_{01}(\tau,z)\vartheta_{10}(\tau,z)\vartheta_{11}(\tau,z)}
\\[0mm]
&=& 1
% ch^{(-)}
\\[3mm]
& & \hspace{-10mm}
{\rm ch}^{(-)}_{V^{[M=1,j=\frac12]}}(\tau,z)
\,\ = \,\ 
q^{(\frac12)^2} e^{2\pi iz}
\\[2mm]
& &
\times \,\ 
\dfrac{
\vartheta_{00}(\tau, \, z+\frac12\tau)
\vartheta_{01}(\tau, \, z+\frac12\tau)
\vartheta_{10}(\tau, \, z+\frac12\tau)
}{\vartheta_{11}(\tau, \, z+\frac12\tau)}
\cdot \frac{\vartheta_{01}(\tau,z)}{
\vartheta_{00}(\tau,z)\vartheta_{10}(\tau,z)\vartheta_{11}(\tau,z)}
\\[0mm]
&=& 1
\end{eqnarray*}}
since
\begin{equation}
\left\{
\begin{array}{lcr}
\vartheta_{00}(\tau, \, z+\frac{\tau}{2}) &=&
q^{-\frac18} \, e^{-\pi iz} \, \vartheta_{10}(\tau, z) \\[1mm]
\vartheta_{01}(\tau, \, z+\frac{\tau}{2}) &=& - \, i \, 
q^{-\frac18} \, e^{-\pi iz} \, \vartheta_{11}(\tau, z) \\[1mm]
\vartheta_{10}(\tau, \, z+\frac{\tau}{2}) &=&
q^{-\frac18} \, e^{-\pi iz} \, \vartheta_{00}(\tau, z) \\[1mm]
\vartheta_{11}(\tau, \, z+\frac{\tau}{2}) &=& - \, i \, 
q^{-\frac18} \, e^{-\pi iz} \, \vartheta_{01}(\tau, z)
\end{array} \right.
\label{n4:eqn:2023-107a}
\end{equation}

\medskip

\noindent
{\rm 1)} \,\ For $j=0 \in I^{[1], R}$, Theorem \ref{n4:thm:2022-1203a} gives
{\allowdisplaybreaks 
\begin{eqnarray*}
& & \hspace{-10mm}
{\rm ch}^{(+)}_{V^{[M=1,j=0]R}}(\tau,z)
\,\ = \,\ 
\frac{
\vartheta_{00}(\tau, z)
\vartheta_{01}(\tau, z)
\vartheta_{11}(\tau, z)
}{\vartheta_{10}(\tau, z)}
\cdot \frac{\vartheta_{10}(\tau,z)}{
\vartheta_{00}(\tau,z)\vartheta_{01}(\tau,z)\vartheta_{11}(\tau,z)}
\,\ = \,\ 1
\\[3mm]
& & \hspace{-10mm}
{\rm ch}^{(-)}_{V^{[M=1,j=0]R}}(\tau,z)
\,\ = \,\ 
\frac{
\vartheta_{00}(\tau, z)
\vartheta_{01}(\tau, z)
\vartheta_{10}(\tau, z)
}{\vartheta_{11}(\tau, z)}
\cdot \frac{\vartheta_{11}(\tau,z)}{
\vartheta_{00}(\tau,z)\vartheta_{01}(\tau,z)\vartheta_{10}(\tau,z)}
\,\ = \,\ 1
\end{eqnarray*}}
\end{ex}

\medskip

These are just in consistency with that $V^{[1, \frac12]}$ 
(resp. $V^{[1,0]R}$) is the trivial representation of the N=4 
superconformal algebra of Neveu-Schwarz type (resp. Ramond type). 

\medskip

Next, we consider the case $M=2$, 
where $I^{[2]}=\big\{\pm \frac12\big\}$ 
and $I^{[2], R}=\{0, 1\}$ by \eqref{n4:eqn:2022-1203a}.

\medskip

%(line=4041)
%(label=n4:prop:2023-106a)
\begin{prop} \,\ 
\label{n4:prop:2023-106a}
Let $j \in I^{[2]} =\big\{\pm \frac12\big\}$. Then 
\begin{enumerate}
\item[{\rm 1)}]$c^{[M=2]}=-3$ \,\ and \,\ 
$\left\{
\begin{array}{ccc}
h^{[M=2,j=\frac12]} &=&-\frac14 \\[1mm]
s^{[M=2,j=\frac12]} &=&-\frac12
\end{array}\right.$ \,\ and \,\ $\left\{
\begin{array}{ccc}
h^{[M=2,j=-\frac12]} &=&-\frac14 \\[1mm]
s^{[M=2,j=-\frac12]} &=&-\frac32
\end{array}\right.$
\item[{\rm 2)}] the characters are as follows: 
\begin{enumerate}
\item[{\rm (i)}] $
{\rm ch}_{V^{[M=2,j=\pm \frac12]}}^{(+)}(\tau,z) = i \, 
\bigg[\dfrac{\eta(2\tau)}{\eta(\tau)}\bigg]^3 \dfrac{
\vartheta_{00}(2\tau, z \pm \frac{\tau}{2}) \, 
\vartheta_{00}(\tau,z)
}{
\vartheta_{10}(2\tau, z \pm \frac{\tau}{2}) \, 
\vartheta_{11}(2\tau,2z)}$
\item[{\rm (ii)}] $
{\rm ch}_{V^{[M=2,j=\pm \frac12]}}^{(-)}(\tau,z) = \pm 
\bigg[\dfrac{\eta(2\tau)}{\eta(\tau)}\bigg]^3 \dfrac{
\vartheta_{01}(2\tau, z\pm\frac{\tau}{2}) \, 
\vartheta_{01}(\tau,z)
}{
\vartheta_{11}(2\tau, z\pm\frac{\tau}{2}) \, 
\vartheta_{11}(2\tau,2z)}$
\end{enumerate}

\item[{\rm 3)}] the leading terms (= the terms of the least degree 
of $q$) are as follows:
\begin{enumerate}
\item[{\rm (i)}] the leading term of 
${\rm ch}_{V^{[2,\frac12]}}^{(\pm)}(\tau,z)
= \, 
q^{\frac18-\frac14} \dfrac{e^{-\pi iz}}{1-e^{-4\pi iz}}
\, = \,
q^{-\frac{1}{24}c^{[2]}+h^{[2, \frac12]}} \cdot \, 
\dfrac{e^{2\pi i s^{[2, \frac12]}z}}{1-e^{-4\pi iz}}$
\item[{\rm (ii)}] the leading term of 
${\rm ch}_{V^{[2,-\frac12]}}^{(\pm)}(\tau,z)
=  
q^{\frac18-\frac14} \dfrac{e^{-3\pi iz}}{1-e^{-4\pi iz}}
= \,
q^{-\frac{1}{24}c^{[2]}+h^{[2, -\frac12]}} \cdot 
\dfrac{e^{2\pi i s^{[2, -\frac12]}z}}{1-e^{-4\pi iz}}$
\end{enumerate}
\end{enumerate}
\end{prop}

\begin{proof} 2) By Theorem \ref{n4:thm:2022-1203a}, we have 
\begin{subequations}
{\allowdisplaybreaks 
\begin{eqnarray}
& & \hspace{-7mm}
{\rm ch}_{V^{[M=2,j=\frac12]}}^{(+)}(\tau,z)
\,\ = \,\ 
- \, q^{\frac{1}{2}(\frac12)^2} e^{\frac{4\pi i}{2} \cdot \frac12 z}
\nonumber
\\[2mm]
& &
\times \,\ \frac{
\vartheta_{00}(2\tau, \, z+\frac{\tau}{2}) 
\vartheta_{01}(2\tau, \, z+\frac{\tau}{2})
\vartheta_{11}(2\tau, \, z+\frac{\tau}{2})
}{\vartheta_{10}(2\tau, \, z+\frac{\tau}{2})} \cdot
\frac{\vartheta_{00}(\tau,z)}{
\vartheta_{01}(\tau,z)\vartheta_{10}(\tau,z)\vartheta_{11}(\tau,z)}
\label{eqn:2023-102a1}
% (-)
\\[3mm]
& & \hspace{-7mm}
{\rm ch}_{V^{[M=2,j=\frac12]}}^{(-)}(\tau,z)
\,\ = \,\ 
q^{\frac{1}{2}(\frac12)^2} e^{\frac{4\pi i}{2} \cdot \frac12 z}
\nonumber
\\[2mm]
& &
\times \,\ \dfrac{
\vartheta_{00}(2\tau, \, z+\frac{\tau}{2})
\vartheta_{01}(2\tau, \, z+\frac{\tau}{2})
\vartheta_{10}(2\tau, \, z+\frac{\tau}{2})
}{\vartheta_{11}(2\tau, \, z+\frac{\tau}{2})} \cdot
\frac{\vartheta_{01}(\tau,z)}{
\vartheta_{00}(\tau,z)\vartheta_{10}(\tau,z)\vartheta_{11}(\tau,z)}
\label{eqn:2023-102a2}
\end{eqnarray}}
\end{subequations}
and
\begin{subequations}
{\allowdisplaybreaks 
\begin{eqnarray}
& & \hspace{-7mm}
{\rm ch}_{V^{[M=2,j=-\frac12]}}^{(+)}(\tau,z)
\,\ = \,\ 
q^{\frac{1}{2}(\frac12)^2} e^{\frac{4\pi i}{2} \cdot (-\frac12) z}
\nonumber
\\[2mm]
& &
\times \,\ \frac{
\vartheta_{00}(2\tau, \, z-\frac{\tau}{2})
\vartheta_{01}(2\tau, \, z-\frac{\tau}{2})
\vartheta_{11}(2\tau, \, z-\frac{\tau}{2})
}{\vartheta_{10}(2\tau, \, z-\frac{\tau}{2})} \cdot
\frac{\vartheta_{00}(\tau,z)}{
\vartheta_{01}(\tau,z)\vartheta_{10}(\tau,z)\vartheta_{11}(\tau,z)}
\label{eqn:2023-102b1}
% (-)
\\[3mm]
& & \hspace{-7mm}
{\rm ch}_{V^{[M=2,j=-\frac12]}}^{(-)}(\tau,z)
\,\ = \,\ 
- \, q^{\frac{1}{2}(\frac12)^2} e^{\frac{4\pi i}{2} \cdot (-\frac12) z}
\nonumber
\\[2mm]
& &
\times \,\ \dfrac{
\vartheta_{00}(2\tau, \, z-\frac{\tau}{2})
\vartheta_{01}(2\tau, \, z-\frac{\tau}{2})
\vartheta_{10}(2\tau, \, z-\frac{\tau}{2})
}{\vartheta_{11}(2\tau, \, z-\frac{\tau}{2})} \cdot
\frac{\vartheta_{01}(\tau,z)}{
\vartheta_{00}(\tau,z)\vartheta_{10}(\tau,z)\vartheta_{11}(\tau,z)}
\label{eqn:2023-102b2}
\end{eqnarray}}
\end{subequations}
Rewriting the above equations by using 
\begin{subequations}
\begin{equation} \left\{
\begin{array}{ccr}
\vartheta_{00}\Big(2\tau, z+\dfrac{\tau}{2}\Big) \, 
\vartheta_{10}\Big(2\tau, z+\dfrac{\tau}{2}\Big)
&=&
q^{-\frac18} \, e^{-\pi iz} \, 
\dfrac{\eta(2\tau)^2}{\eta(\tau)} \, \vartheta_{00}(\tau, z)
\\[4mm]
\vartheta_{01}\Big(2\tau, z+\dfrac{\tau}{2}\Big) \, 
\vartheta_{11}\Big(2\tau, z+\dfrac{\tau}{2}\Big)
&=&
- \, i \, q^{-\frac18} \, e^{-\pi iz} \, 
\dfrac{\eta(2\tau)^2}{\eta(\tau)} \, \vartheta_{01}(\tau, z)
\end{array} \right.
\label{eqn:2023-103c}
\end{equation}
and
\begin{equation} \left\{
\begin{array}{ccr}
\vartheta_{00}\Big(2\tau, z-\dfrac{\tau}{2}\Big) \, 
\vartheta_{10}\Big(2\tau, z-\dfrac{\tau}{2}\Big)
&=&
q^{-\frac18} \, e^{\pi iz} \, 
\dfrac{\eta(2\tau)^2}{\eta(\tau)} \, \vartheta_{00}(\tau, z)
\\[4mm]
\vartheta_{01}\Big(2\tau, z-\dfrac{\tau}{2}\Big) \, 
\vartheta_{11}\Big(2\tau, z-\dfrac{\tau}{2}\Big)
&=&
i \, q^{-\frac18} \, e^{\pi iz} \, 
\dfrac{\eta(2\tau)^2}{\eta(\tau)} \, \vartheta_{01}(\tau, z)
\end{array} \right.
\label{eqn:2023-103d}
\end{equation}
\end{subequations}
and 
\begin{equation}
\left\{
\begin{array}{rcr}
\vartheta_{00}(\tau,z) \, \vartheta_{01}(\tau,z)
&=&
\dfrac{\eta(\tau)^2}{\eta(2\tau)} \, \vartheta_{01}(2\tau,2z)
\\[4mm]
\vartheta_{10}(\tau,z) \, \vartheta_{11}(\tau,z)
&=&
\dfrac{\eta(\tau)^2}{\eta(2\tau)} \, \vartheta_{11}(2\tau,2z)
\end{array}\right.
\label{eqn:2023-103b}
\end{equation}
we obtain the formulas in 2).

\medskip

\noindent
3) is obtained by calculation using 2) and the product expression of  
$\vartheta_{ab}(\tau,z)$:
\begin{subequations}
{\allowdisplaybreaks
\begin{eqnarray}
\vartheta_{00}(\tau, z) &=& 
\prod_{n=1}^{\infty} \, 
(1-q^n)(1+e^{2\pi iz}q^{n-\frac12})(1+e^{-2\pi iz}q^{n-\frac12})
\label{n4:eqn:2023-108b1}
\\[0mm]
\vartheta_{01}(\tau, z) &=&
\prod_{n=1}^{\infty} \, 
(1-q^n)(1-e^{2\pi iz}q^{n-\frac12})(1-e^{-2\pi iz}q^{n-\frac12})
\label{n4:eqn:2023-108b2}
\\[0mm]
\vartheta_{10}(\tau, z) &=&
e^{\frac{\pi i \tau}{4}} \, e^{\pi iz} \, 
\prod_{n=1}^{\infty} \, (1-q^n)(1+e^{2\pi iz}q^n)(1+e^{-2\pi iz}q^{n-1})
\nonumber
\\[0mm]
&=&
e^{\frac{\pi i \tau}{4}} \, e^{-\pi iz} \, 
\prod_{n=1}^{\infty} \, (1-q^n)(1+e^{2\pi iz}q^{n-1})(1+e^{-2\pi iz}q^n)
\label{n4:eqn:2023-108b3}
\\[0mm]
\vartheta_{11}(\tau, z) &=&
e^{\frac{\pi i \tau}{4}} \, e^{\pi i(z+\frac12)} \, 
\prod_{n=1}^{\infty} \, (1-q^n)(1-e^{2\pi iz}q^n)(1-e^{-2\pi iz}q^{n-1})
\nonumber
\\[0mm]
&=&
e^{\frac{\pi i \tau}{4}} \, e^{-\pi i(z+\frac12)} \, 
\prod_{n=1}^{\infty} \, (1-q^n)(1-e^{2\pi iz}q^{n-1})(1-e^{-2\pi iz}q^n)
\label{n4:eqn:2023-108b4}
\end{eqnarray}}
\end{subequations}

\vspace{-6mm}

\end{proof}

%(line=4261) %%
%(label=n4:prop:2023-106b) %%
\begin{prop} \,\ 
\label{n4:prop:2023-106b}
Let $j \in I^{[2]R} =\big\{0,1\big\}$. Then 
\begin{enumerate}
\item[{\rm 1)}] $c^{[M=2]}=-3$ \,\ and \,\ 
$\left\{
\begin{array}{ccc}
h^{[M=2,j=0]R} &=&-\frac18 \\[1mm]
s^{[M=2,j=0]R} &=&0
\end{array}\right.$ \,\ and \,\ $\left\{
\begin{array}{ccc}
h^{[M=2,j=1]R} &=&\frac38 \\[1mm]
s^{[M=2,j=1]R} &=&1
\end{array}\right.$
\item[{\rm 2)}] the characters are as follows:
\begin{enumerate}
\item[${\rm (i)}^+$] $
{\rm ch}_{V^{[M=2,j=0]R}}^{(+)}(\tau,z) = 
\bigg[\dfrac{\eta(2\tau)}{\eta(\tau)}\bigg]^3 \dfrac{
\vartheta_{00}(2\tau, z) \, 
\vartheta_{10}(\tau,z)
}{
\vartheta_{10}(2\tau, z) \, 
\vartheta_{01}(2\tau,2z)}$
\item[${\rm (i)}^-$] $
{\rm ch}_{V^{[M=2,j=0]R}}^{(-)}(\tau,z) = 
\bigg[\dfrac{\eta(2\tau)}{\eta(\tau)}\bigg]^3 \dfrac{
\vartheta_{01}(2\tau, z) \, 
\vartheta_{11}(\tau,z)
}{
\vartheta_{11}(2\tau, z) \, 
\vartheta_{01}(2\tau,2z)}$
\item[${\rm (ii)}^+$] $
{\rm ch}_{V^{[M=2,j=1]R}}^{(+)}(\tau,z) = 
\bigg[\dfrac{\eta(2\tau)}{\eta(\tau)}\bigg]^3 \dfrac{
\vartheta_{10}(2\tau, z) \, 
\vartheta_{10}(\tau,z)
}{
\vartheta_{00}(2\tau, z) \, 
\vartheta_{01}(2\tau,2z)}$
\item[${\rm (ii)}^-$] $
{\rm ch}_{V^{[M=2,j=1]R}}^{(-)}(\tau,z) = - \, 
\bigg[\dfrac{\eta(2\tau)}{\eta(\tau)}\bigg]^3 \dfrac{
\vartheta_{11}(2\tau, z) \, 
\vartheta_{11}(\tau,z)
}{
\vartheta_{01}(2\tau, z) \, 
\vartheta_{01}(2\tau,2z)}$
\end{enumerate}
\item[{\rm 3)}] the leading terms are as follows:
\begin{enumerate}
\item[{\rm (i)}] the leading term of 
${\rm ch}_{V^{[M=2,j=0]R}}^{(\pm)}(\tau,z)
= \, 1 \, = \,\ 
q^{-\frac{1}{24}c^{[2]}+h^{[2, 0]R}} \, e^{2\pi is^{[2, 0]R}z}$

\item[{\rm (ii)}] the leading term of 
${\rm ch}_{V^{[M=2,j=1]R}}^{(\pm)}(\tau,z)
\, = \,\  
q^{\frac12} \, e^{2\pi iz} (1 \pm e^{-2\pi iz})^2$
$$
= \,\ q^{-\frac{1}{24}c^{[2]}+h^{[2, 1]R}} \, 
e^{2\pi is^{[2, 1]R}z} \, (1 \pm e^{-2\pi iz})^2
$$
\end{enumerate}
\end{enumerate}
\end{prop}

\begin{proof} By Theorem \ref{n4:thm:2022-1203a}, we have 
\begin{subequations}
%{\allowdisplaybreaks 
\begin{eqnarray}
& & \hspace{-5mm}
{\rm ch}_{V^{[M=2,j=0]R}}^{(+)}(\tau,z)
\, = \,\  
\frac{
\vartheta_{00}(2\tau, \, z)
\vartheta_{01}(2\tau, \, z)
\vartheta_{11}(2\tau, \, z)
}{\vartheta_{10}(2\tau, \, z)}
\cdot \frac{\vartheta_{10}(\tau,z)}{
\vartheta_{00}(\tau,z)\vartheta_{01}(\tau,z)\vartheta_{11}(\tau,z)}
\nonumber
\\[1mm]
& &
\label{eqn:2023-102c1}
\\[-2mm]
& & \hspace{-5mm}
{\rm ch}_{V^{[M=2,j=0]R}}^{(-)}(\tau,z)
\, = \,\  
\frac{
\vartheta_{00}(2\tau, \, z)
\vartheta_{01}(2\tau, \, z)
\vartheta_{10}(2\tau, \, z)
}{\vartheta_{11}(2\tau, \, z)}
\cdot \frac{\vartheta_{11}(\tau,z)}{
\vartheta_{00}(\tau,z)\vartheta_{01}(\tau,z)\vartheta_{10}(\tau,z)}
\nonumber
\\[1mm]
& &
\label{eqn:2023-102c2}
\end{eqnarray}
% }
\end{subequations}
and
\begin{subequations}
{\allowdisplaybreaks 
\begin{eqnarray}
& & \hspace{-5mm}
{\rm ch}_{V^{[M=2,j=1]R}}^{(+)}(\tau,z)
\,\ = \,\ - \, 
q^{\frac12} e^{2\pi iz}
\nonumber
\\[2mm]
& & \times \,\ 
\frac{
\vartheta_{00}(2\tau, \, z+\tau)
\vartheta_{01}(2\tau, \, z+\tau)
\vartheta_{11}(2\tau, \, z+\tau)
}{\vartheta_{10}(2\tau, \, z+\tau)}
\cdot \frac{\vartheta_{10}(\tau,z)}{
\vartheta_{00}(\tau,z)\vartheta_{01}(\tau,z)\vartheta_{11}(\tau,z)}
\label{eqn:2023-102d1}
\\[3mm]
& & \hspace{-5mm}
{\rm ch}_{V^{[M=2,j=1]R}}^{(-)}(\tau,z)
\,\ = \,\ - \, 
q^{\frac12} e^{2\pi iz}
\nonumber
\\[2mm]
& & \times \,\ 
\frac{
\vartheta_{00}(2\tau, \, z+\tau)
\vartheta_{01}(2\tau, \, z+\tau)
\vartheta_{10}(2\tau, \, z+\tau)
}{\vartheta_{11}(2\tau, \, z+\tau)}
\cdot \frac{\vartheta_{11}(\tau,z)}{
\vartheta_{00}(\tau,z)\vartheta_{01}(\tau,z)\vartheta_{10}(\tau,z)}
\label{eqn:2023-102d2}
\end{eqnarray}}
\end{subequations}
Rewriting the above equations by using 
\begin{equation}
\left\{
\begin{array}{rcr}
\vartheta_{00}(2\tau,z) \, \vartheta_{10}(2\tau,z)
&=&
\dfrac{\eta(2\tau)^2}{\eta(\tau)} \, \vartheta_{10}(\tau,z)
\\[4mm]
\vartheta_{01}(2\tau,z) \, \vartheta_{11}(2\tau,z)
&=&
\dfrac{\eta(2\tau)^2}{\eta(\tau)} \, \vartheta_{11}(\tau,z)
\end{array}\right.
\label{eqn:2023-103a}
\end{equation}
and
\begin{equation}
\left\{
\begin{array}{lcr}
\vartheta_{00}(2\tau, z+\tau) &=&
q^{-\frac14} \, e^{-\pi iz} \, \vartheta_{10}(2\tau, z) \\[1mm]
\vartheta_{01}(2\tau, z+\tau) &=& - \, i \, 
q^{-\frac14} \, e^{-\pi iz} \, \vartheta_{11}(2\tau, z) \\[1mm]
\vartheta_{10}(2\tau, z+\tau) &=&
q^{-\frac14} \, e^{-\pi iz} \, \vartheta_{00}(2\tau, z) \\[1mm]
\vartheta_{11}(2\tau, z+\tau) &=& - \, i \, 
q^{-\frac14} \, e^{-\pi iz} \, \vartheta_{01}(2\tau, z) 
\end{array}\right.
\label{n4:eqn:2023-108a}
\end{equation}
and \eqref{eqn:2023-103b}, we obtain the formulas in 2).

\medskip

\noindent
3) is obtained by calculation using 2) and 
\eqref{n4:eqn:2023-108b1} $\sim$ \eqref{n4:eqn:2023-108b4}.
\end{proof}

\section{$SL_2(\zzz)$-invariance of the subspace of characters}
\label{sec:SL(2Z)-invariance}
%(label=sec:SL(2Z)-invariance)
%(line=4478) %%

In this section we consider the characters of non-irreducible 
N=4 modules in the case $m=1$. Applying the results in section
\ref{sec:non-irred} to the case $m=1$, we obtain the following:

\medskip

%(line=4486) %%
%(label=n4:lemma:2023-103a) %%
\begin{lemma} 
\label{n4:lemma:2023-103a}
The numerators of these N=4 modules 
$\ddot{H}(\Lambda^{(M)[K(1),0] (\heartsuit)}_{k_1,k_2})$ are given as follows:
\begin{enumerate}
\item[{\rm 1)}]
\begin{enumerate}
\item[{\rm (i)}] $\big[\overset{N=4}{R}{}^{(+)} \cdot 
{\rm ch}^{(+)}_{\ddot{H}(\Lambda^{(M)[K(1), 0]{\rm (I)}}_{k_1,k_2})}
\big](\tau,z)
\,\ = \,\ 
\Psi^{[M,1,0; \frac12]}_{k_1+\frac12, \, M-(k_1+k_2+\frac12); \, \frac12}
(\tau, z, z,0)$
% 1) (iii)
\item[{\rm (ii)}] $\big[\overset{N=4}{R}{}^{(+)} \cdot 
{\rm ch}^{(+)}_{\ddot{H}(\Lambda^{(M)[K(1), 0]{\rm (III)}}_{k_1,k_2})}\big](\tau,z)
\,\ = \,\ - \, 
\Psi^{[M,1,0; \frac12]}_{M-(k_1+\frac12), \, k_1+k_2+\frac12; \, \frac12}
(\tau, z, z,0)$
\end{enumerate}
% 2)
\item[{\rm 2)}]
\begin{enumerate}
\item[{\rm (i)}] $\big[\overset{N=4}{R}{}^{(-)} \cdot 
{\rm ch}^{(-)}_{\ddot{H}(\Lambda^{(M)[K(1), 0]{\rm (I)}}_{k_1,k_2})}\big](\tau,z)
\,\ = \,\ 
- \, \Psi^{[M,1,0; 0]}_{k_1+\frac12, \, M-(k_1+k_2+\frac12); \, \frac12}
(\tau, z, z,0)$
% 2) (iii)
\item[{\rm (ii)}] $\big[\overset{N=4}{R}{}^{(-)} \cdot 
{\rm ch}^{(-)}_{\ddot{H}(\Lambda^{(M)[K(1), 0]{\rm (III)}}_{k_1,k_2})}\big](\tau,z)
\,\ = \,\ 
\Psi^{[M,1,0; 0]}_{M-(k_1+\frac12), \, k_1+k_2+\frac12; \, \frac12}
(\tau, z, z,0)$
\end{enumerate}
\item[$1)^{\rm tw}$]
\begin{enumerate}
\item[{\rm (i)}] $\big[\overset{N=4}{R}{}^{(+){\rm tw}} \cdot 
{\rm ch}^{(+){\rm tw}}_{\ddot{H}(\Lambda^{(M)[K(1), 0]{\rm (I)}}_{k_1,k_2})}
\big](\tau,z)
\,\ = \,\ - \, 
\Psi^{[M,1,0; \frac12]}_{M-k_1, \, k_1+k_2+1; \, 0}(\tau, z, z,0)$
% 1) (iii)
\item[{\rm (ii)}] $\big[\overset{N=4}{R}{}^{(+){\rm tw}} \cdot 
{\rm ch}^{(+){\rm tw}}_{\ddot{H}(\Lambda^{(M)[K(1), 0]{\rm (III)}}_{k_1,k_2})}\big](\tau,z)
\,\ = \,\ 
\Psi^{[M,1,0; \frac12]}_{k_1+1, \, M-(k_1+k_2); \, 0}(\tau, z, z,0)$
\end{enumerate}
% 2)
\item[$2)^{\rm tw}$]
\begin{enumerate}
\item[{\rm (i)}] $\big[\overset{N=4}{R}{}^{(-){\rm tw}} \cdot 
{\rm ch}^{(-){\rm tw}}_{\ddot{H}(\Lambda^{(M)[K(1), 0]{\rm (I)}}_{k_1,k_2})}\big](\tau,z)
\,\ = \,\ 
- \, \Psi^{[M,1,0; 0]}_{M-k_1, \, k_1+k_2+1; \, 0}(\tau, z, z,0)$
% 2) (iii)
\item[{\rm (ii)}] $\big[\overset{N=4}{R}{}^{(-){\rm tw}} \cdot 
{\rm ch}^{(-){\rm tw}}_{\ddot{H}(\Lambda^{(M)[K(1), 0]{\rm (III)}}_{k_1,k_2})}\big](\tau,z)
\,\ = \,\ 
\Psi^{[M,1,0; 0]}_{k_1+1, \, M-(k_1+k_2); \, 0}(\tau, z, z,0)$
\end{enumerate}
\end{enumerate}
\end{lemma}

\begin{proof}
These formulas are obtained easily by computing the formulas in 
Lemmas \ref{n4:lemma:2023-102a} and \ref{n4:lemma:2023-102b}
in the case $(m,m_2)=(1,0)$, noticing that
$\Psi^{[M,1,1; \varepsilon]}_{j,k; \varepsilon'}
= \Psi^{[M,1,0;\varepsilon]}_{j,k; \varepsilon'}$ and using 
Lemma \ref{lemma:2022-1021a}.
\end{proof}

\medskip

%(line=4560) %%
%(label=n4:thm:2023-103a) %%
\begin{thm}
\label{n4:thm:2023-103a}
Let $M \in \nnn$, then
\begin{enumerate}
\item[{\rm 1)}] the $\ccc$-linear span of 
$$
\bigcup\limits_{\heartsuit = {\rm I}, \, {\rm III}}
\Big\{{\rm ch}^{(\pm)}_{\ddot{H}(\Lambda^{(M)[K(1), 0] (\heartsuit)}_{k_1,k_2})}
(\tau,z),
\,\  
{\rm ch}^{(+) {\rm tw}}_{\ddot{H}(\Lambda^{(M)[K(1), 0] (\heartsuit)}_{k_1,k_2})}
(\tau,z)
\,\ ; \,\ 
(k_1, k_2) \, \text{satisfies \eqref{n4:eqn:2023-101a}}\Big\}
$$
is $SL_2(\zzz)$-invariant,

\item[{\rm 2)}] the $\ccc$-linear span of 
$$
\bigcup\limits_{\heartsuit = {\rm I}, \, {\rm III}}
\Big\{
{\rm ch}^{(-) {\rm tw}}_{\ddot{H}(\Lambda^{(M)[K(1), 0] (\heartsuit)}_{k_1,k_2})}
(\tau,z)
\,\ ; \,\ 
(k_1, k_2) \, \text{satisfies \eqref{n4:eqn:2023-101a}}\Big\}
$$
is $SL_2(\zzz)$-invariant.
\end{enumerate}
\end{thm}

\begin{proof} In view of Lemma \ref{n4:lemma:2023-103a},
we define the parameters $(j_1,j_2)$ and compute the range of 
$(j_1,j_2)$ by using \eqref{n4:eqn:2023-101a} as follows:
\begin{enumerate}
\item[{\rm 1)}] \,\ In the non-twisted case;
\begin{enumerate}
\item[{\rm (I)}] for $\Pi^{(M),{\rm (I)}}_{k_1,k_2}$, we put 
$\left\{
\begin{array}{ccl}
j_1 &:=& k_1+\frac12 \\[2mm]
j_2 &:=& M-(k_1+k_2+\frac12)
\end{array} \right. , $ then $\left\{
\begin{array}{ccl}
j_1 &\geq & \frac12 \\[1mm]
j_1+j_2 &\leq & M \\[1mm]
j_2 &\geq &j_1
\end{array} \right. $
\item[{\rm (III)}] for $\Pi^{(M),{\rm (III)}}_{k_1,k_2}$, we put 
$\left\{
\begin{array}{ccl}
j_1 &:=& M-(k_1+\frac12) \\[2mm]
j_2 &:=& k_1+k_2+\frac12 
\end{array} \right. , $ then $\left\{
\begin{array}{ccl}
j_1 &\leq & M-\frac12 \\[1mm]
j_1+j_2 &\geq & M+1 \\[1mm]
j_2 &\leq & j_1
\end{array} \right. $
\end{enumerate}
\item[{\rm 2)}] \,\ In the twisted case;
\begin{enumerate}
\item[$({\rm I})^{\rm tw}$] for $\Pi^{(M),{\rm (I)}}_{k_1,k_2}$, we put 
$\left\{
\begin{array}{ccl}
j_1 &:=& M-k_1 \\[2mm]
j_2 &:=& k_1+k_2+1
\end{array} \right. , $ then $\left\{
\begin{array}{ccl}
j_1 &\leq & M \\[1mm]
j_1+j_2 &\geq & M+1 \\[1mm]
j_2 &\leq &j_1
\end{array} \right. $
\item[$({\rm III})^{\rm tw}$] for $\Pi^{(M),{\rm (III)}}_{k_1,k_2}$, we put 
$\left\{
\begin{array}{ccl}
j_1 &:=& k_1+1 \\[2mm]
j_2 &:=& M-(k_1+k_2)
\end{array} \right. , $ then $\left\{
\begin{array}{ccl}
j_1 &\geq & 1 \\[1mm]
j_1+j_2 &\leq & M \\[1mm]
j_2 &\geq &j_1
\end{array} \right. $
\end{enumerate}
\end{enumerate}

Then, since
$\Psi^{[M, 1,0; \varepsilon]}_{j_1,j_2;\varepsilon'}(\tau, z.z,0)
=
\Psi^{[M, 1,0; \varepsilon]}_{j_2,j_1;\varepsilon'}(\tau, z.z,0)$
by Lemma \ref{lemma:2022-1021a}, we see that
\begin{enumerate}
\item[{\rm 1)}] \,\ 
$\{(j_1,j_2) \in {\rm (I)}\} \, \cup \, \{(j_1,j_2) \in {\rm (III)}\}$ 
\,\ fills the domain  
$$
\{(j_1,j_2) \, \in \, (\tfrac12\zzz_{\rm odd})^2 \,\ ; \,\ 
0 < j_1, \, j_2 < M\}
\, / \, \sim
$$
\item[{\rm 2)}] \,\ 
$\{(j_1,j_2) \in ({\rm I})^{\rm tw}\} \, \cup \, 
\{(j_1,j_2) \in ({\rm III})^{\rm tw}\}$ \,\  
fills the domain  
$$
\{(j_1,j_2) \, \in \, \zzz^2 \,\ ; \,\ 0 < j_1, \, j_2 \leq M\}
\, / \, \sim
$$
\end{enumerate}
with the equivalence relation \lq \lq \hspace{0.3mm} $\sim$ " 
defined by $(j_1,j_2) \sim (j_2, j_1)$.

Then, by the modular transformation properties of the functions 
$\Psi^{[M,1,0;\varepsilon]}_{j,k;\varepsilon}$
in Lemma \ref{n4:lemma:2022-1207a}
together with the modular transformation formulas
\eqref{n4:eqn:2022-1210c1} and \eqref{n4:eqn:2022-1210c2}
of the N=4 denominators, we obtain the 
$SL_2(\zzz)$-invariance of the space of these characters,
proving Theorem \ref{n4:thm:2023-103a}
\end{proof}

%\section{•â'«}
%\subsection{•â'«1}
%\subsection{•â'«2}

\end{document}